\pdfoutput=1

\documentclass[11pt,a4paper,twoside]{article}

\usepackage{etoolbox}
\csundef{leftbar}
\csundef{endleftbar}

\usepackage[utf8]{inputenc}
\usepackage[T1]{fontenc}
\usepackage[english]{babel} 
\usepackage{lmodern}
\usepackage{makeidx} 
\usepackage{graphicx}
\usepackage{caption}
\usepackage{eurosym}
\usepackage{xfrac}
\usepackage{listings}
\usepackage{mathrsfs}
\usepackage{stmaryrd}
\usepackage{eqnarray}
\usepackage{multirow}
\usepackage{bold-extra}
\usepackage{blkarray}

\usepackage{amsmath}
\usepackage{amsthm}
\usepackage{thmtools}
\declaretheoremstyle[
	spaceabove=\topsep, spacebelow=\topsep,
	headfont=\normalfont\itshape,
	notefont=\normalfont, notebraces={(}{)},
	bodyfont=\normalfont,
	postheadspace=.5em,
	qed=
]{remarque}
\declaretheorem[numbered=no,name=Remark,style=remarque]{rema}
\declaretheoremstyle[
	spaceabove=\topsep, spacebelow=\topsep,
	headfont=\normalfont\bfseries,
	notefont=\normalfont, notebraces={(}{)},
	bodyfont=\normalfont,
	postheadspace=.5em,
]{defn}
\declaretheorem[numberwithin=section,name=Definition,style=defn]{defi}
\declaretheorem[sibling=defi,name=Example,style=defn]{exe}

\declaretheoremstyle[
	spaceabove=\topsep, spacebelow=\topsep,
	headfont=\normalfont\bfseries,
	notefont=\normalfont, notebraces={(}{)},
	bodyfont=\itshape,
	postheadspace=.5em,
	qed=
]{proposition}
\declaretheorem[sibling=defi,name=Proposition,style=proposition]{prop}
\declaretheorem[sibling=defi,name=Theorem,style=proposition]{theo}
\declaretheorem[numbered=no,name=Theorem,style=proposition]{theor}
\declaretheorem[sibling=defi,name=Lemma,style=proposition]{lem}
\declaretheorem[sibling=defi,name=Corollary,style=proposition]{cor}

\usepackage{amssymb}
\usepackage{amstext}

\usepackage[top=2cm, bottom=2cm, left=2cm, right=2cm]{geometry}

\usepackage{tikz}
\usetikzlibrary{arrows,matrix,cd}

\usepackage{bbm}
\usepackage{multicol}

\usepackage{import}
\usepackage{whitecdp}

\usepackage{lettrine}

  {\vspace*{\fill}\clearpage\restoregeometry}
  
\usepackage[french,lined]{algorithm2e}

\newcommand{\N}{\mathbb{N}}
\newcommand{\Z}{\mathbb{Z}}
\newcommand{\R}{\mathbb{R}}

\newcommand{\Q}{\mathbb{Q}}

\newcommand{\A}{\mathbb{A}}

\newcommand{\Er}{\mathrm{E}}

\newcommand{\I}{\mathrm{I}}

\newcommand{\m}{\mathfrak{m}}
\newcommand{\p}{\mathfrak{p}}

\newcommand{\rf}{\mathfrak{r}}

\newcommand{\Kr}{\mathrm{K}}

\newcommand{\MW}{\mathrm{MW}}
\newcommand{\Um}{\mathrm{Um}}
\newcommand{\WMS}{\mathrm{WMS}}
\newcommand{\wms}{\mathrm{wms}}
\newcommand{\Nis}{\mathrm{Nis}}

\newcommand{\Sr}{\mathrm{S}}

\newcommand{\CH}{\mathrm{CH}}
\newcommand{\Hr}{\mathrm{H}}
\newcommand{\RS}{\mathrm{RS}}
\newcommand{\cd}{\mathrm{cd}}
\newcommand{\Cr}{\mathrm{C}}
\newcommand{\Nr}{\mathrm{N}}
\newcommand{\Kbf}{\mathbf{K}}

\DeclareMathOperator{\Id}{Id}

\DeclareMathOperator{\diag}{diag}

\DeclareMathOperator{\GL}{GL}

\DeclareMathOperator{\Hom}{Hom}

\DeclareMathOperator{\Spec}{Spec}

\DeclareMathOperator{\colim}{colim}
\DeclareMathOperator{\Proj}{Proj}

\DeclareMathOperator{\Ob}{Ob}

\DeclareMathOperator{\Map}{Map}

\DeclareMathOperator{\codim}{codim}
\DeclareMathOperator{\Tor}{Tor}
\DeclareMathOperator{\Kos}{Kos}
\DeclareMathOperator{\sdim}{sdim}
\DeclareMathOperator{\sr}{sr}

\DeclareMathOperator{\rk}{rk}

\DeclareMathOperator{\hocolim}{hocolim}

\renewcommand{\Im}{\operatorname{Im}}

\makeatletter 

\newcommand*\Neginternal[3]{\mathpalette\Neg@{{#1}{#2}{#3}}}
\newcommand*\Neg@[2]{\Neg@@{#1}#2}
\newcommand*\Neg@@[4]{%
  \mathrel{\ooalign{%
    $\m@th#1#4$\cr
    \hidewidth$\m@th#3{#1}\mkern\muexpr#2*2$\hidewidth\cr
  }}%
}

\newcommand*\negslash[1]{\m@th#1\not\mathrel{\phantom{=}}}
\newcommand*\snegslash[1]{\rotatebox[origin=c]{60}{$\m@th#1-$}}
\newcommand*\ssnegslash[1]{\rotatebox[origin=c]{60}{$\m@th#1{\dabar@}\mkern-7mu{\dabar@}$}}
\newcommand*\sssnegslash[1]{\rotatebox[origin=c]{60}{$\m@th#1\dabar@$}}
\makeatother  

\usepackage[newparttoc,pagestyles,toctitles]{titlesec}
\usepackage{titletoc}

\renewcommand*\thesection{\arabic{section}}
\titleformat{\section}[block]{\Large\scshape\filcenter}{\thesection.}{12pt}{}[]

\newpagestyle{main}{

\sethead[\thepage][][\thesection.\hspace{.7em}\scshape\sectiontitle]
	{\thepage}{}{\thesection.\hspace{0.7em}\scshape\sectiontitle}
\setfoot[][][]
	{}{}{}}

\usepackage[style=english]{csquotes}

\usepackage[backend=biber,backref=true,style=alphabetic]{biblatex}

\DefineBibliographyExtras{english}{%
}%

\usepackage{authblk}

\usepackage{hyperref}

\begin{document}

\pagestyle{main}

\title{Motivic stable cohomotopy and unimodular rows}
\author{Samuel Lerbet}
\affil{Univ. Grenoble Alpes, CNRS, IF, 38000, Grenoble, France}
\date{\today} 

\maketitle

\tableofcontents

\begin{abstract}
We relate the group structure of van der Kallen on orbit sets of unimodular rows (\cite{vdK}) with values in a smooth algebra $A$ over a field $k$ with the motivic cohomotopy groups of $X=\Spec A$ with coefficients in $\A^n\setminus 0$ in the sense of \cite{AF}. In the last section, we compare the motivic cohomotopy theory studied in this paper and defined by $\A^{n+1}\setminus 0$ or, equivalently, by an $\A^1$-weakly equivalent quadric $Q_{2n+1}$ to that considered in \cite{AF}, defined by a quadric $Q_{2n}$, by means of explicit morphisms $Q_{2n+1}\rightarrow Q_{2n}$, $Q_{2n}\times\mathbb{G}_m\rightarrow Q_{2n+1}$ of quadrics.
\end{abstract}

\section{Introduction}

\subsection*{Some history}

\paragraph*{Unimodular rows.} Let $n$ be a non-negative integer. If $A$ is a ring, \emph{unimodular rows} of length $n$ with values in $A$ control surjective $A$-module homomorphisms $A^n\rightarrow A$ and are thus of interest in algebraic $\Kr$-theory. They form a set $\Um_n(A)$ which is endowed with a natural action of $\GL_n(A)$ on the right so that we may consider the orbit set of the action of the subgroup $\Er_n(A)$ of elementary matrices. Building on earlier work of Vaserstein and Suslin in the case $n=2$ (\cite[Section 5]{SuVa}), van der Kallen introduced in \cite{vandK} a group structure on this orbit set under finiteness assumptions on the ring $A$. He uses the notion of \emph{stable dimension} $\sdim$ of a ring $A$ which is deduced from Bass's \emph{stable range} condition. This is in fact a set $(\mathrm{DS})_r$ of conditions indexed by integers $r\geqslant 1$, defined on \cite[p. 231]{BBproj}, where $\mathrm{DS}$ stand for the French \emph{domaine stable}; $\mathrm{DS}(A)$ is then the infinimum of the set of integers $r$ such that $A$ satisfies $(\mathrm{DS})_r$ if this set is not empty, and $\mathrm{DS}(A)=\infty$ if this set is empty. Finally, by definition, $\sdim(A)=\mathrm{DS}(A)-1$ (\cite[Introduction, 1.2]{vandK}). Then van der Kallen assumes that $\sdim(A)$ satisfies the inequality $\sdim(A)\leqslant 2n-4$ and he shows that one may then endow the orbit set $\Um_n(A)/\Er_n(A)$ with the structure of an abelian group functorially in $A$. This was accomplished in the early eighties and is stated in \cite[Theorem 3.3]{vdK}: thanks to this theorem, the group structure may be easily understood in terms of the elementary Mennicke--Newman lemma. Definitions and results on this subject are recalled in Section \ref{uni}.

\paragraph*{Stable cohomotopy.} The group law introduced by van der Kallen is an algebraic analogue of Borsuk's construction of \emph{cohomotopy groups}. Let $X$ be a CW-complex. If $n$ is an integer and if $X$ is “small” relative to $n$—technically, we assume that the dimension $d$ of $X$ satisfies $d\leqslant 2n-2$—, then Borsuk equips the set $[X,\Sr^{n}]$ of homotopy classes of continuous maps from $X$ to the $n$-sphere $\Sr^n$ with the structure of an abelian group denoted by $\pi^n(X)$, see \cite{Bor}. Roughly speaking, given continuous maps $f,g:X\rightarrow\Sr^n$, Borsuk lifts them into a map to the wedge sum $\Sr^n\vee\Sr^n$ and composes with the fold map $\Sr^n\vee\Sr^n\rightarrow\Sr^n$ to obtain a map $X\rightarrow\Sr^n$ which is by definition the sum of $f$ and $g$: in particular, it is easy to see that this group structure is functorial in $X$.

\begin{rema}
Here, $2n-2$ should be thought of as $2(n-1)$ and $n-1$ as the connectedness of $\Sr^n$: more generally, if $Y$ is a pointed space that is $m$-connected, then Borsuk's construction yields a group structure on $[X,Y]$ for any CW-complex $X$ of dimension $d\leqslant 2m$ which is functorial in $X$ \emph{and} $Y$.
\end{rema}

\paragraph*{The topological comparison theorem.} If $A$ is an algebra over a field $k$, then unimodular rows of length $n$ with values in $A$ may be interpreted as $k$-morphisms from $\Spec A$ to $\A_k^n\setminus 0=\A^n\setminus 0$. Similarly, if $X$ is a finite CW-complex of dimension $d$, then unimodular rows of length $n$ with values in the ring $\Cr(X)$ of real valued continuous functions on $X$ may be viewed as continuous maps $X\rightarrow\R^n\setminus\{0\}$. Composing with the standard homotopy equivalence $\R^n\setminus\{0\}\rightarrow\Sr^{n-1}$ allows one to think about unimodular rows with values in $\Cr(X)$ in terms of the cohomotopy group $\pi^{n-1}(X)$. It is easy to see that elements of the same orbit under the action of $\Er_n(\Cr(X))$ induce continuous maps $X\rightarrow\Sr^{n-1}$ which are homotopic, thus producing a map \[\Phi:\Um_n(\Cr(X))/\Er_n(\Cr(X))\rightarrow\pi^{n-1}(X).\] It is then known that the ring $\Cr(X)$ is of stable dimension $d$ by work of Vaserstein \cite{LNV}, so that the source and the target of the map $\Phi$ both have a natural group structure if $d\leqslant 2n-4$ following van der Kallen and Borsuk respectively and in fact:

\begin{theor}[\protect{\cite[Theorem 7.7]{vandK}}]
If $X$ is a finite CW-complex of dimension $d\leqslant 2n-4$, then the map $\Phi:\Um_n(\mathrm{C}(X))/\Er_n(\mathrm{C}(X))\rightarrow\pi^{n-1}(X)$ is a group isomorphism.
\end{theor}

This theorem gave van der Kallen the intuition that one should be able to think about the group law on $\Um_n(A)/\Er_n(A)$ in terms of yet-to-be-defined cohomotopy groups for schemes which would shed light on the somewhat difficult-to-follow computations of \cite{vandK}. However, the technology to do homotopy theory in the category of schemes was not available at the time of van der Kallen's work so that this remained a heuristic for some time.

\paragraph*{Motivic homotopy theory.} The situation has changed in the last thirty years and we now have a well-established homotopy theory for schemes: the \emph{$\A^1$-homotopy theory} initiated by Morel and Voevodsky in \cite{MoVo}. In particular, assuming one is working over a field $k$, one has (several) models of the topological sphere in $\A^1$-homotopy theory, one of those being the $k$-scheme $\A^n\setminus 0$. This allows one to define cohomotopy groups $[X,\A^n\setminus 0]_{\A^1}$ if $X$ is of suitable ($\A^1$-cohomological) dimension: this is done by Asok and Fasel in \cite{AF} and we state the results that will be useful to us in Section \ref{coh}. 

\paragraph*{The motivic comparison theorem.} Now let $A$ be a smooth affine $k$-algebra, that is, we assume that $A$ is of finite type over $k$, and let $X$ denote its spectrum. It is easy to see that unimodular rows in the same orbit under the action of $\Er_n(A)$ induce $\A^1$-homotopic morphisms $X\rightarrow\A^n\setminus 0$ so that we have a well-defined map \[\Psi:\Um_n(A)/\Er_n(A)\rightarrow[X,\A^n\setminus 0]_{\A^1}.\] Assume further that the Krull dimension $d$ of $A$ satisfies $d\leqslant 2n-4$. Then since $A$ is Noetherian, we have the inequality $\sdim(A)\leqslant d$, see for instance \cite[Th\'eor\`eme 1]{BBproj}, hence the orbit set $\Um_n(A)/\Er_n(A)$ has a group structure following van der Kallen, while its definition and results of Nisnevich on the homonymous cohomological dimension show that the $\A^1$-cohomological dimension of a scheme is bounded above by its Krull dimension, hence $[X,\A^n\setminus 0]_{\A^1}$ also has a group structure following Asok and Fasel. The main result of the first part of this paper confirms van der Kallen's intuition and may be stated as follows.

\begin{theor}[Theorem \ref{groupiso}]
Assume that the ring $A$ is of Krull dimension $d\leqslant 2n-4$. Then the map $\Psi$ is a group isomorphism.
\end{theor}

\subsection*{Structure of the proof}

The proof of Theorem \ref{groupiso} occupies Section \ref{proof}. It is divided in two subsections.
\begin{itemize}
	\item In the first subsection, we show that $\Psi$ is a bijection. The proof is split in two parts. First, given a smooth $k$-algebra $A$, we compare the orbit set $\Um_n(A)/\Er_n(A)$ to the set of \emph{naive} $\A^1$-homotopy classes of morphisms $X\rightarrow\A^n\setminus 0$: this is a particular case of \cite[Theorem 2.1]{JF}. Second, we appeal to a result of Asok, Hoyois and Wendt \cite[Theorem 2.2.4]{AHW2}, more precisely its consequence \cite[Corollary 4.2.6]{AHW2}, that states that \emph{with source an affine scheme} and \emph{with certain targets}, for example $\A^n\setminus 0$, naive $\A^1$-homotopy classes and ``true'' $\A^1$-homotopy classes actually coincide. This phenomenon was already observed by Morel in \cite[Remark 8.10]{Morel} for $k$ perfect and $n\neq 2$ in the particular case of target $\A^n\setminus 0$.
	\item In the second subsection, we prove that $\Psi$ is a group homomorphism. We introduce an open subset $V_{2n}$ of $(\A^n\setminus 0)\times(\A^n\setminus 0)$: this $k$-scheme allows us to produce an explicit, scheme-theoretic model of the fold map $(\A^n\setminus 0)\vee(\A^n\setminus 0)\rightarrow\A^n\setminus 0$ using the procedure of the Mennicke--Newman lemma. This model of the fold map is the link between cohomotopy groups and the group structure on the orbit set. To see that the datum of maps $u,v:X\rightarrow\A^n\setminus 0$ for $X$ small is equivalent (up to homotopy) to the datum of a map $X\rightarrow V_{2n}$, that is, that we can lift maps $X\rightarrow(\A^n\setminus 0)\times(\A^n\setminus 0)$ to maps $X\rightarrow V_{2n}$ (up to homotopy), we need to study the $\A^1$-connectedness of the map $V_{2n}\rightarrow(\A^n\setminus 0)\times(\A^n\setminus 0)$. This is done by using the motivic Blakers--Massey theorem to compare the fibre and the cofibre of this inclusion, the $\A^1$-connectedness of the cofibre being controlled by the purity theorem.
\end{itemize}

\subsection*{Comparing cohomotopy theories}

After this investigation of the cohomotopy theory defined by unimodular rows or, equivalently, by $\A^n\setminus 0$, we compare it to the cohomotopy theory studied in \cite{AF}. The latter is given by another model of sphere, the quadric \[Q_{2n}=k[x_1,\ldots,x_n,y_1,\ldots,y_n,z]/\langle x_1y_1+\cdots+x_ny_n-z(1-z)\rangle.\] This comparison is made more convenient by replacing $\A^n\setminus 0$ by the quadric \[Q_{2n-1}=k[x_1,\ldots,x_n,y_1,\ldots,y_n]/\langle x_1y_1+\cdots+x_ny_n-1\rangle\] whose projection on the first $n$ factors yields a morphism to $\A^n\setminus 0$ which is in fact an isomorphism in the $\A^1$-homotopy category (slightly more can be said about this map, see our discussion of Jouanolou devices below): thus, this replacement is harmless from the point of view of cohomotopy theory. 

With this comparison in view, we produce morphisms of $k$-schemes $\eta_n:Q_{2n+1}\rightarrow Q_{2n}$ and $\mu_n:\mathbb{G}_m\times Q_{2n}\rightarrow Q_{2n+1}$ for all $n>0$, where $\mathbb{G}_m$ is the multiplicative group over $k$. 

First, Morel's computation of $\A^1$-homotopy sheaves of motivic spheres gives an isomorphism $[Q_{2n+1},Q_{2n}]_{\A^1}\simeq\Kr_{-1}^\MW(k)$ where the left-hand side denotes the hom-set between $Q_{2n+1}$ and $Q_{2n}$ in the $\A^1$-homotopy category (see below for more details on notations) and $\Kr_*^\MW(k)$ denotes the Milnor--Witt $\Kr$-theory of $k$ (\cite[Definition 3.1]{Morel}): $\eta_n$ corresponds to the generator $\eta$ of Milnor--Witt $\Kr$-theory under this isomorphism, see Theorem \ref{genofkmin1}, which explains our interest in this map. Furthermore, in \cite[Section 7]{bhatwadekar_sridharan_2000}, given a Noetherian $\Q$-algebra $A$ of Krull dimension $d$ where $d$ an \emph{even} integer, the authors construct a group homomorphism $\Psi$ from $\Um_{d+1}(A)/\Er_{d+1}(A)$ to the \emph{Euler class group} $\Er^d(A)$ of $A$, the theory of these groups being the object of \cite{bhatwadekar_sridharan_2000}. The map $\Psi$ is constructed using the \emph{Euler class} of the rank $d$ projective module induced by an element of $\Um_{d+1}(A)$, namely the kernel of the corresponding surjective homomorphism of $A$-modules. This class is defined in \cite[Section 4]{bhatwadekar_sridharan_2000} and the authors observe in the course of the proof of \cite[Proposition 7.5]{bhatwadekar_sridharan_2000} that it is trivial for modules of odd rank: thus, it was not clear how to define such a map for $n$ odd. This was accomplished in \cite{DAS2015185} where the authors construct a map $\phi:\Um_{d+1}(A)/\Er_{d+1}(A)\rightarrow\Er^d(A)$ for \emph{any} $d$ and for any Noetherian ring $A$ of Krull dimension $d$ by introducing a \emph{strong} Euler class following a suggestion of Fasel; the same morphism was considered in \cite{VANDERKALLEN201565}, although a different strategy was used to prove that it is well-defined. The map $\eta_n$ then provides a geometric interpretation of $\phi$, see Proposition \ref{linkbetweenetamor}: this interpretation is made possible by the description of the cohomotopy theory defined by $Q_{2n}$ in terms of Euler class groups obtained in \cite[Section 3]{AF}, see in particular \cite[Theorem 3.1.13]{AF}. 

On the other hand, the morphism $\mu_n:\mathbb{G}_m\times Q_{2n}\rightarrow Q_{2n+1}$ also gives geometric meaning to the set-theoretic map $\delta$ from Euler class groups to orbit sets of unimodular rows constructed in \cite[Subsection 2.5]{DTZ1} and allows us to produce a new proof of the improvement of \cite[Corollary 2.11]{DTZ1} stated in \cite[Theorem 6.3]{das2018euler}, namely that $\delta_A$ is a group homomorphism. In addition, slightly modifying $\mu_n$ yields an $\A^1$-weak equivalence $\mathbb{G}_m\wedge Q_{2n}\rightarrow Q_{2n+1}$: see Theorem \ref{eqfaibleqpgmqi}.

\subsection*{Acknowledgements}

I would like to thank my advisor Jean Fasel without whom this work would certainly not have seen the light of day. His generosity in sharing his ideas and his patience in answering my numerous naive (in the non-technical sense) questions can hardly be overstated. It is also a pleasure to thank the referee for a very careful review to which the significantly improved readability of the present paper owes much (thus any remaining opacity rests solely on me). Finally, a significant part of this article was written while I was an \emph{élève} of the École normale supérieure: I also wish to thank the École for its material support.

\section{Some notations and facts of $\A^1$-homotopy theory}

\paragraph*{Categorical considerations.} A \emph{cogroup} object in a category $\mathsf{C}$ is a group object in the opposite category $\mathsf{C}^{op}$; a \emph{cocommutative} cogroup object in $\mathsf{C}$ is a commutative group object in $\mathsf{C}^{op}$. By definition, if $X$ is a (cocommutative) cogroup object in $\mathsf{C}$, then $\Hom_{\mathsf{C}}(X,Y)$ has a natural (commutative) group structure for any object $Y$ of $\mathsf{C}$. If $\mathsf{C}$ has finite coproducts hence an initial object $0$, a (cocommutative) cogroup may be described as an object $X$ of $\mathsf{C}$ together with morphisms $c:X\rightarrow X\coprod X$ (comultiplication or sometimes coaddition in the case of a cocommutative cogroup), $i:X\rightarrow X$ (coinverse) and $e:X\rightarrow 0$ (counit) such that, seen in $\mathsf{C}^{op}$, $(X,c,i,e)$ is a (commutative) group object. In the context of homotopy theory, we will speak of (cocommutative) $h$-cogroups, that is, (cocommutative) cogroup objects in the homotopy category: these are discussed for general model categories in \cite{Hovey}.

We have the following entirely formal lemma about functors and cogroups: 

\begin{lem}\label{functorscommcoprodcogrp}
If $\mathsf{C}$ and $\mathsf{C}'$ have finite coproducts and if $F:\mathsf{C}\rightarrow\mathsf{C}'$ is a functor that commutes with finite coproducts, then $F$ preserves cogroup objects in an obvious sense and given a cogroup object $X$ in $\mathsf{C}$, the map $\Hom_{\mathsf{C}}(X,Y)\rightarrow\Hom_{\mathsf{C}'}(F(X),F(Y))$ is a group homomorphism for any object $Y$ of $\mathsf{C}$. 
\end{lem}

\begin{proof}
Let $(X,c,i,e)$ be a cogroup in $\mathsf{C}$. By assumption on $F$, the map $\varphi:F(X)\coprod F(X)\rightarrow F(X\coprod X)$ induced by application of $F$ to the structure maps $X\rightarrow X\coprod X$ is an isomorphism. Thus $c$ induces a map $F(X)\xrightarrow{F(c)}F\left(X\coprod X\right)\xrightarrow{\varphi^{-1}}F(X)\coprod F(X)$. Similarly, $i$ induces a map $F(i):F(X)\rightarrow F(X)$ and $e$ induces a map $F(X)\rightarrow F(0)$ where $F(0)$ is initial because $F$ commutes with finite (hence empty) coproducts. These data endow $F(X)$ with the structure of a cogroup by commutation of $F$ with finite coproducts. Now let $Y$ be an object of $\mathsf{C}$. We denote the group laws on $\Hom_{\mathsf{C}}(X,Y)$ and $\Hom_{\mathsf{C}'}(F(X),F(Y))$ by $*$. Then we have $F(f)*F(g)=F(f)\coprod F(g)\circ(\varphi^{-1}\circ F(c))$ since the comultiplication of $F(X)$ is $\varphi^{-1}\circ F(c)$. By definition, $F(f)\coprod F(g)\circ\varphi^{-1}:F(X\coprod X)\rightarrow F(X)\coprod F(X)\rightarrow F(Y)$ is equal to $F(f\coprod g)$. Thus $F(f)*F(g)=F(f\coprod g)\circ F(c)=F(f\coprod g\circ c)=F(f*g)$ as required.
\end{proof}

Now let $\mathsf{A}$ be an additive category with finite direct sums denoted using the symbol $\bigoplus$ (thus we make no notational distinction between finite coproducts and finite products). Any object $X$ of $\mathsf{A}$ then has a natural cocommutative cogroup structure whose underlying coaddition $c^{add}:X\rightarrow X\oplus X$ is the map induced by $(\Id_X,\Id_X)$.

\begin{lem}\label{lem:uniqueness_cogroup_additive_cat}
Let $(X,c,i,e)$ be a cogroup structure on the object $X$ of the additive category $\mathsf{A}$. Then $c=c^{add}$.
\end{lem}

In particular, for any object $Y$ of $\mathsf{A}$, the group structure on the set $\Hom_{\mathsf{A}}(X,Y)$ induced by $c$ is precisely the group structure provided by the fact that $\mathsf{A}$ is an additive category: there are no “exotic” cogroup objects in an additive category. 

\begin{proof}
We denote by $\Hom_{\mathsf{A}}^c(X,X)$ the hom-set of morphisms from $X$ to $X$ in the category $\mathsf{A}$ endowed with the group structure induced by $c$ and by $\Hom_{\mathsf{A}}(X,X)$ the group of morphisms in the additive category $\mathsf{A}$ (whose group structure is thus induced by $c^{add}$). By definition, the neutral element of $\Hom_{\mathsf{A}}^c(X,X)$ is the morphism $X\xrightarrow{e}0\rightarrow X$. Since $0$ is a zero object of $\mathsf{A}$, it is also the neutral element of the group $\Hom_{\mathsf{A}}(X,X)$. Now denote by $*$ the group law underlying $\Hom_{\mathsf{A}}^c(X,X)$. Consider the structure morphisms $i_1,i_2:X\rightarrow X\oplus X$ and $p_1,p_2:X\oplus X\rightarrow X$ of the direct sums. Denoting by $*$ the group law on $\Hom_{\mathsf{A}}(X,X)$. By definition of a cogroup, $*$ is bilinear for composition, hence $p_1\circ(i_1*i_2)=(p_1\circ i_1)*(p_1\circ i_2)=p_1\circ i_1=\Id_X$ because $p_1\circ i_2$ is the neutral element of $\Hom_{\mathsf{A}}(X,X)$ hence of $\Hom_{\mathsf{A}}^c(X,X)$ as already observed. Similarly, $p_2\circ(i_1*i_2)=\Id_X$ and as a result, the equality $i_1*i_2=c^{add}=(\Id_X,\Id_X):X\rightarrow X\oplus X$ holds. On the other hand, $i_1*i_2=i_1\oplus i_2\circ c$ where $i_1\oplus i_2:X\oplus X\rightarrow X\oplus X$ is equal to $\Id_{X\oplus X}$ by definition. Thus $c=i_1*i_2=c^{add}$.
\end{proof}

For instance, we may apply Lemma \ref{functorscommcoprodcogrp} with $\mathsf{C}'$ additive, in which case the target of the map of said lemma has an unambiguous group structure thanks to Lemma \ref{lem:uniqueness_cogroup_additive_cat}.

\paragraph*{General notations.} Let $A$ be a ring. We denote by $A^\times$ be the group of invertible elements of $A$. If $(x_1,\ldots,x_n)$ is a family of elements of an $A$-module $M$—for instance an ideal of $A$—, we denote by $\langle x_1,\ldots,x_n\rangle$ the submodule of $M$ generated by $(x_1,\ldots,x_n)$. We say that the homomorphism $f:A^n\rightarrow M$ which sends the $i$-th vector of the standard basis of $A^n$ to $x_i$ is \emph{induced} by $(x_1,\ldots,x_n)$.

We let $k$ be a field, fixed throughout the paper (specific assumptions on $k$ are introduced as required). Given $k$-schemes $X$ and $Y$, we denote by $X\times Y$ the fibre product $X\times_{\Spec k} Y$. The $k$-algebra of global sections of an affine $k$-scheme $X$ is denoted by $k[X]$; the residue field of a $k$-scheme $X$ at a point $x$ is denoted by $k(x)$. The codimension of a point $x\in X$ in $X$, which is by definition the codimension of the irreducible closed subset $\overline{\{x\}}$ of $X$, is denoted by $\codim_X x$.

Let $Z$ be an affine $k$-scheme. Then the datum of a quasi-coherent sheaf on $Z$ is equivalent to the datum of a $k[Z]$-module: we usually make no difference in notation between the two objects, denoting by $M$ the quasi-coherent sheaf induced by the $k[Z]$-module $M$ and vice versa, and we may loosely speak of $M$ as a vector bundle on $Z$ if $M$ is a projective $k[Z]$-module of finite rank in view of the Serre--Swan correspondence between finite projective $k[Z]$-modules and locally free $\mathscr{O}_Z$-modules of finite rank (\cite[\href{https://stacks.math.columbia.edu/tag/00NX}{Tag 00NX}]{stacks-project}, equivalence between (2) and (7)). We let $\A^n$ denote the $k$-scheme $\A_k^n$, and similarly with $\A^n\setminus 0$ and $\mathbb{P}^n$.

Given a morphism of rings $A\rightarrow B$, we denote the $B$-module of Kähler differentials of $B$ over $A$ by $\Omega_{B/A}$. If $A$ is a field and $B$ is smooth over $A$ (and connected), then $\Omega_{B/A}$ is projective of constant rank $d=\dim B$ by \cite[\href{https://stacks.math.columbia.edu/tag/02G1}{Tag 02G1}]{stacks-project} and the Serre--Swan correspondence cited above; we denote by $\omega_{B/A}$ its determinant, that is, the top exterior power $\omega_{B/A}=\bigwedge^d\Omega_{B/A}$ of $\Omega_{B/A}$.

\paragraph*{The $\A^1$-homotopy categories.} All the $k$-schemes that we consider are assumed to be of finite type. By \emph{smooth $k$-scheme}, we mean a $k$-scheme that is smooth and separated (and of finite type). We let $\mathsf{Sm}_k$ denote the category of smooth $k$-schemes with all $k$-scheme morphisms as morphisms; we endow this category with the structure of a site using the Nisnevich topology \cite[Subsection 1.1]{Nisnevich1989}. We denote by $\mathsf{Spc}_k$ the category of simplicial presheaves on $\mathsf{Sm}_k$; an object of $\mathsf{Spc}_k$ is a \emph{space} over $k$ or $k$-space. Thus both smooth $k$-schemes, by their functor of points (viewed as a presheaf of sets, hence of discrete simplicial sets) restricted to $\mathsf{Sm}_k$, and simplicial sets, by the induced constant simplicial presheaf on $\mathsf{Sm}_k$, are examples of $k$-spaces. We finally let $\mathcal{H}(k)$ be the \emph{$\A^1$-homotopy category} of spaces as constructed in, \emph{e.g.}, \cite{AHW1}. If $\mathscr{X}$ and $\mathscr{Y}$ are spaces over $k$, we denote by $[\mathscr{X},\mathscr{Y}]_{\A^1}$ the hom-set in $\mathcal{H}(k)$ between $\mathscr{X}$ and $\mathscr{Y}$. An $\A^1$-weak equivalence is an isomorphism in $\mathcal{H}(k)$; we extend this terminology to morphisms of spaces whose image in $\mathcal{H}(k)$ is an isomorphism.

A \emph{point} of a space $\mathscr{X}$ is a morphism $*\rightarrow\mathscr{X}$ of simplicial presheaves, where $*$ is the final object of $\mathcal{H}(k)$; a space equipped with a point is a pointed space whose point is called the base point; a map between spaces which commutes with points is a map of pointed spaces. We denote by $\mathsf{Spc}_{k,\bullet}$ the category of pointed spaces. There is a pointed version $\mathcal{H}_\bullet(k)$ of $\mathcal{H}(k)$: the hom-set between $(\mathscr{X},x)$ and $(\mathscr{Y},y)$ in $\mathcal{H}_\bullet(k)$ is denoted by $[(\mathscr{X},x),(\mathscr{Y},y)]_{\A^1,\bullet}$. Often, the base point underlying a pointed space is implicit in its notation, that is, we write $\mathscr{X}$ for $(\mathscr{X},x)$.

Any space $\mathscr{X}$ can be turned into a pointed space $\mathscr{X}_+$ by adjoining a disjoint base point. If $\mathscr{X}$ and $\mathscr{Y}$ are pointed spaces, the wedge-sum $\mathscr{X}\vee\mathscr{Y}$ and the smash-product and $\mathscr{X}\wedge\mathscr{Y}$ are defined as usual. We denote by $\mathrm{S}^1$ the simplicial circle, pointed as usual and viewed as a space using the associated constant presheaf, and $\mathrm{S}^n$ the wedge product $\mathrm{S}^1\wedge\cdots\wedge\mathrm{S}^1$ of $n$ copies of $\mathrm{S}^1$. We also set $\Sigma X=\Sr^1\wedge X$ for any pointed space $k$.

\begin{exe}\label{exe:homotopy_type_punctured_a^n}
Let $n$ be a positive integer. Then we turn $\A^n\setminus 0$ into a pointed space with the point $(0,\ldots,0,1)\in(\A^n\setminus 0)(k)$; there is a pointed $\A^1$-weak equivalence $\A^{n}\setminus 0\simeq\Sr^{n-1}\wedge\mathbb{G}_m^{\wedge n}$ by \cite[§3, Example 2.20]{MoVo}.
\end{exe}

For $\mathscr{X}$ a pointed space, we set $\Omega\mathscr{X}=\mathbf{R}\Map(\mathrm{S}^1,\mathscr{X})$ where $\Map$ denotes the derived pointed mapping space in the simplicial category of simplicial presheaves and the letter $\mathbf{R}$ indicates that we take the derived functor; this assignment extends to a \emph{loop space} functor $\Omega$. Looping and suspension are adjoint, hence there is a functorial adjunction unit morphism $\mathscr{X}\rightarrow\Omega\Sigma\mathscr{X}$ (in $\mathcal{H}_\bullet(k)$) for any pointed space $\mathscr{X}$.

Finally, we denote by $\mathcal{SH}^{\Sr^1}(k)$ the \emph{$\Sr^1$-stable $\A^1$-homotopy category} described for instance in \cite[p. 2787]{AF} and by $\Sigma^\infty:\mathcal{H}_\bullet(k)\rightarrow\mathcal{SH}^{\Sr^1}(k)$ the infinite $\Sr^1$-suspension functor. It commutes with finite coproducts, that is, with wedge sum: this follows from the corresponding fact about the suspension $\Sigma$ at the level of pointed spaces. The set of morphisms between objects $E$ and $E'$ of $\mathcal{SH}^{\Sr^1}(k)$ is denoted by $[E,E']_{\A^1,\bullet}$. We also denote by $\mathcal{SH}(k)$ the $\mathbb{P}^1$-stable $\A^1$-homotopy category as constructed in \emph{e.g.} \cite{Morel_trieste} and by $\Sigma_{\mathbb{P}^1}^\infty:\mathcal{H}_\bullet(k)\rightarrow\mathcal{SH}(k)$ the infinite $\mathbb{P}^1$-suspension functor. Both $\mathcal{SH}^{\Sr^1}(k)$ and $\mathcal{SH}(k)$ are additive categories. Moreover, there are localisation functors $\mathcal{H}(k)\rightarrow\mathcal{SH}^{\Sr^1}(k)\rightarrow\mathcal{SH}(k)$. An isomorphism of $\mathcal{SH}(k)$, or a morphism of spaces whose image in $\mathcal{SH}(k)$ is an isomorphism, is said to be a $\mathbb{P}^1$-stable $\A^1$-weak equivalence (thus if the word $\mathbb{P}^1$-stable is not present, then the morphism in question is required to be an isomorphism in the \emph{un}stable $\A^1$-homotopy category $\mathcal{H}(k)$).

\paragraph*{Jouanolou devices.} Let $X$ be a topological space. Recall the notion of torsor under a sheaf of groups \cite[\href{https://stacks.math.columbia.edu/tag/02FO}{Tag 02FO}]{stacks-project}. If $X$ is a scheme, by definition, the datum of the structure of a torsor under a sheaf of groups $\mathscr{G}$ on a morphism $p:Y\rightarrow X$ is the datum of the structure of a torsor under $\mathscr{G}$ on the sheaf $U\mapsto\{s:U\rightarrow Y,p\circ s=\Id_U\}$ of sections of $p$. We have the following lemma about the existence of affine models of smooth schemes.

\begin{lem}[Jouanolou--Thomason, \protect{\cite[Proposition 4.4]{Weibel_KH_theory}}]\label{jouanolou}
Let $X$ be a separated $k$-scheme of finite type. Then there exists a morphism $\widetilde{X}\rightarrow X$ of $k$-schemes possessing the following properties.
\begin{itemize}
	\item The $k$-scheme $\widetilde{X}$ is affine.
	\item The morphism $\widetilde{X}\rightarrow X$ is a Zariski-locally trivial torsor under a vector bundle on $X$.
\end{itemize}
We call such an $X$-scheme $\widetilde{X}$ a \emph{Jouanolou device} for $X$.
\end{lem}

Given a Jouanolou device $\widetilde{X}$ of $X$, $\widetilde{X}\rightarrow X$ is Zariski-locally on $X$ isomorphic to $\A_X^n\rightarrow X$; in particular, it is a smooth $X$-scheme. Hence if $X$ is a smooth $k$-scheme, then $\widetilde{X}$ is a smooth $k$-scheme and the structure map $\widetilde{X}\rightarrow X$ is an $\A^1$-homotopy equivalence. We make two additional observations about this lemma.
\begin{itemize}
	\item Base change does \emph{not} preserve Jouanolou devices. That is, given a Jouanolou device $p:\widetilde{X}\rightarrow X$ of an object $X$ of $\mathsf{Sm}_k$ and a $k$-scheme morphism $f:Y\rightarrow X$ where $Y$ is an object of $\mathsf{Sm}_k$, form the following fibre product:
	\begin{center}
	\begin{tikzcd}
	\widetilde{Y} \arrow[r,"g"] \arrow[d,"q"] & \widetilde{X} \arrow[d,"p"] \\
	Y \arrow[r,"f"]                           & X 
	\end{tikzcd}
	\end{center}
	Then $q:\widetilde{Y}\rightarrow Y$ is not a Jouanolou device in general. In fact, the only obstacle is that $\widetilde{Y}$ is not necessarily an affine $k$-scheme. Note however that if $f$ is an affine morphism, then so is $g$ by base change: since $\widetilde{X}$ is affine, this implies that $\widetilde{Y}$ is indeed affine in this case. We summarise this situation by saying that Jouanolou devices are preserved by base change by an affine morphism.
	\item Assume $X$ is a smooth affine $k$-scheme in Lemma \ref{jouanolou}. Let $p:\widetilde{X}\rightarrow X$ be a Jouanolou device of $X$. Then $p$ has a section: there exists a morphism $s:X\rightarrow\widetilde{X}$ of $k$-schemes such that $p\circ s=\Id_X$. Indeed, since $X$ is affine, given a quasi-coherent $\mathscr{O}_X$-module $\mathscr{E}$, any Zariski-locally torsor under $\mathscr{E}$ is trivial by \cite[Proposition 16.5.16]{EGA_IV_4} (the crucial input is the vanishing of the Zariski sheaf cohomology group $\Hr_{\mathrm{Zar}}^1(X,\mathscr{E})$ which follows from the fact that $X$ is affine). Thus the sheaf of sections of $p$ has a global section $s$ as indicated previously (see \cite[\href{https://stacks.math.columbia.edu/tag/02FP}{Tag 02FP}]{stacks-project}).
\end{itemize}

As a result of combining these items, we obtain the following corollary:

\begin{cor}\label{factbyjouan}
Let $Y$ and $X$ be smooth $k$-schemes such that $Y$ is an affine $k$-scheme and let $\widetilde{X}\rightarrow X$ be a Jouanolou device. Then any morphism $Y\rightarrow X$ of $k$-schemes factors through $\widetilde{X}\rightarrow X$ into a morphism $Y\rightarrow\widetilde{X}$. In other words, the map $\Hom_{\mathsf{Sm}_k}(Y,\widetilde{X})\rightarrow\Hom_{\mathsf{Sm}_k}(Y,X)$ induced by the morphism $\widetilde{X}\rightarrow X$ is surjective.
\end{cor}

\begin{proof}
Note that the morphism $Y\rightarrow X$ is automatically affine as $Y$ is an affine $k$-scheme and $X$ is a separated $k$-scheme (\cite[\href{https://stacks.math.columbia.edu/tag/01SG}{Tag 01SG}, (2)]{stacks-project}). Let $\widetilde{Y}$ be the fibre product $\widetilde{Y}=Y\times_X\widetilde{X}$; then the morphism $\widetilde{Y}\rightarrow Y$ is a Jouanolou device for $Y$ and since $Y$ is affine, it has a section $s:Y\rightarrow\widetilde{Y}$; composing with the structural morphism $\widetilde{Y}\rightarrow\widetilde{X}$ of the fibre product yields a factorisation of the map $Y\rightarrow X$. 
\end{proof}

\paragraph*{Naive $\A^1$-homotopy.} Of particular interest to us, among $\A^1$-homotopic morphisms, will be \emph{naively} $\A^1$-homotopic morphisms:

\begin{defi}[naive $\A^1$-homotopy]
Let $\mathscr{X}$ and $\mathscr{Y}$ be $k$-spaces and let $f,g:\mathscr{X}\rightarrow\mathscr{Y}$ be $k$-morphisms. We say that $f$ and $g$ are \emph{naively $\A^1$-homotopic} when there exists a map $H:\mathscr{X}\times\A^1\rightarrow\mathscr{Y}$ of $k$-spaces such that $H(\text{--},0)=f$ and $H(\text{--},1)=g$; such an $H$ is a naive $\A^1$-homotopy between $f$ and $g$. The relation ``being naively $\A^1$-homotopic'' generates an equivalence relation which we denote by $\sim_{\A^1}$ and the quotient of $\Hom_{\mathsf{Spc}_k}(\mathscr{X},\mathscr{Y})$ by $\sim_{\A^1}$ is denoted by $\Hom_{\A^1}(\mathscr{X},\mathscr{Y})$; we then have a quotient map $\Hom_{\mathsf{Spc}_k}(\mathscr{X},\mathscr{Y})\rightarrow\Hom_{\A^1}(\mathscr{X},\mathscr{Y})$. 
\end{defi}

The relation $\sim_{\A^1}$ is compatible with composition in the sense that if $f\sim_{\A^1} f'$, then $g\circ f\sim_{\A^1} g\circ f'$ and $f\circ h\sim_{\A^1}f'\circ h$ for any maps $g$ and $h$. It follows that we have a category $\mathcal{H}_\Nr(k)$ of $k$-spaces up to \emph{naive} $\A^1$-homotopy where for any $k$-spaces $\mathscr{X}$ and $\mathscr{Y}$, $\Hom_{\mathcal{H}_\Nr(k)}(\mathscr{X},\mathscr{Y})=\Hom_{\A^1}(\mathscr{X},\mathscr{Y})$. There is an obvious full functor $\mathsf{Spc}_k\rightarrow\mathcal{H}_\Nr(k)$ and any functor that sends naively $\A^1$-equivalent morphisms to equal morphisms factors uniquely through this functor.

Let $\mathscr{X}$ and $\mathscr{Y}$ be $k$-spaces. For $t\in k$, let $i_t:\mathscr{X}\rightarrow\mathscr{X}\times\A^1$ denote the inclusion of the $k$-rational point $t$ in $\A^1$. Then a morphism $H:\mathscr{X}\times\A^1\rightarrow\mathscr{Y}$ is a naive $\A^1$-homotopy between $f$ and $g$ if $f=i_0^*H=H\circ i_0$ and $g=i_1^*H=H\circ i_1$. Letting $p:\mathscr{X}\times\A^1\rightarrow\mathscr{X}$ denote the first projection, we have that $p\circ i_0=p\circ i_1$ hence $i_0^*\circ p^*=i_1^*\circ p^*$ as maps between hom-sets in $\mathcal{H}(k)$. On the other hand, $p$ is an $\A^1$-weak equivalence by definition of $\mathcal{H}(k)$, hence $p^*$ is a bijection. It follows that $i_0^*=i_1^*$, hence $f=g$ in $\mathcal{H}(k)$. As a result, the obvious functor $\mathsf{Spc}_k\rightarrow\mathcal{H}(k)$ factors (uniquely) through the functor $\mathsf{Spc}_k\rightarrow\mathcal{H}_\Nr(k)$. For general $\mathscr{X}$ and $\mathscr{Y}$, the map $\Hom_{\A^1}(\mathscr{X},\mathscr{Y})\rightarrow[\mathscr{X},\mathscr{Y}]_{\A^1}$ induced by the functor $\mathcal{H}_\Nr(k)\rightarrow\mathcal{H}(k)$ is not injective nor surjective. For example, the main result of \cite{ASENS_2012_4_45_4_511_0} shows that this map is not surjective if $\mathscr{X}=\mathscr{Y}=\mathbb{P}^1$. An crucial difference between $\mathcal{H}_\Nr(k)$ and $\mathcal{H}(k)$ is that in the first category, any morphism from $\mathscr{X}$ to $\mathscr{Y}$ comes from a morphism of \emph{spaces} from $\mathscr{X}$ to $\mathscr{Y}$.

\paragraph*{The quadrics $Q_n$.} We recall the definition of the quadrics alluded to above:

\begin{defi}\label{def:quadrics}
Let $n$ be a non-negative integer. We set
\[
\begin{array}{rcl}
Q_{2n+1} &=& \displaystyle\Spec k[x_1,\ldots,x_{n+1},y_1,\ldots,y_{n+1}]/\left\langle\sum_{i=1}^{n+1} x_iy_i-1\right\rangle\subseteq\A^{2n+2} \\
Q_{2n}   &=& \displaystyle\Spec k[x_1,\ldots,x_n,y_1,\ldots,y_n,z]/\left\langle\sum_{i=1}^nx_iy_i-z(1-z)\right\rangle\subseteq\A^{2n+1}
\end{array}
\]
One readily checks that $Q_n$ is a smooth affine integral $k$-scheme of Krull dimension $n$ for all $n>0$.
\end{defi} 

\begin{rema}
In \cite{ADF}, a different quadric bears the name $Q_{2n}$, namely the quadric in $\A^{2n+1}=\Spec k[x_1,\ldots,x_n,y_1,\ldots,y_n,z]$ defined by the equation $x_1y_1+\cdots+x_ny_n=z(1+z)$. As noted in \cite{AF}, it is isomorphic to the quadric that we denote by $Q_{2n}$ via the change of variables $y_i\mapsto -y_i$ and $z\mapsto -z$. We use this isomorphism when appealing to the results of \cite{ADF}.
\end{rema}

Let $R$ be a $k$-algebra. We identify $Q_{2n}(R)$ (respectively $Q_{2n+1}(R)$) with the set of ordered triples $(x,y,z)$ (respectively ordered pairs $(x,y)$), where $x=(x_1,\ldots,x_n)$ and $y=(y_1,\ldots,y_n)$ (respectively $x=(x_1,\ldots,x_{n+1})$ and $y=(y_1,\ldots,y_{n+1})$) are row vectors in $R^n$ (respectively in $R^{n+1}$) and $z$ is an element of $R$, satisfying $xy^T=z(1-z)$ (respectively $xy^T=1$); see the first paragraph of \cite[Subsection 2.1]{AF} (p. 2794). We also occasionally simply denote the elements of $Q_{2n}(R)$ (respectively $Q_{2n+1}(R)$) as $(x_1,\ldots,x_n,y_1,\ldots,y_n,z)\in R^{2n+1}$ (respectively $(x_1,\ldots,x_{n+1},y_1,\ldots,y_{n+1})\in R^{2n+2}$) with $\sum x_iy_i=z(1-z)$ (respectively $\sum x_iy_i=1$). We point the quadric $Q_{2n}$ (respectively $Q_{2n+1}$) by $(0,\ldots,0,0,\ldots,0,0)\in Q_{2n}(k)$ (respectively $(0,\ldots,0,1,0,\ldots,0,1)\in Q_{2n+1}(k)$).

We record a few facts about the quadrics $Q_n$.

\begin{lem}\label{quadricjouandev}
Let $n$ be a non-negative integer. Then the map $p:Q_{2n+1}\rightarrow\A^{n+1}\setminus 0$ given by projection on the first $n+1$ coordinates is a Jouanolou device for $\A^{n+1}\setminus 0$.
\end{lem}

\begin{proof}
Indeed, consider the epimorphism \[\mathscr{O}^{n+1}\rightarrow\mathscr{O},\;(a_1,\ldots,a_{n+1})\mapsto\sum a_ix_i\] of sheaves of abelian groups on $\A^{n+1}\setminus 0\subseteq\A^{n+1}=\Spec k[x_1,\ldots,x_{n+1}]$. Then its kernel is a locally free $\mathscr{O}_X$-module and $p$ makes $Q_{2n+1}$ into a torsor under the associated vector bundle on $\A^{n+1}\setminus 0$.
\end{proof}

\begin{theo}\label{homotypequad}
Let $n$ be a non-negative integer. Then there are pointed $\A^1$-weak equivalences $Q_{2n+1}\simeq\Sr^n\wedge\mathbb{G}_m^{\wedge(n+1)}$ and $Q_{2n}\simeq(\mathbb{P}^1)^{\wedge n}\simeq\Sr^n\wedge\mathbb{G}_m^{\wedge n}$.
\end{theo}

\begin{proof}
There is a pointed $\A^1$-weak equivalence $\A^{n+1}\setminus 0\simeq\Sr^n\wedge\mathbb{G}_m^{\wedge(n+1)}$ as noted in Example \ref{exe:homotopy_type_punctured_a^n}, thus the statement about $Q_{2n+1}$ follows from the previous lemma and the fact that Jouanolou devices are $\A^1$-weak equivalences. The statement about $Q_{2n}$ may be deduced from \cite[Theorem 2.2.5]{ADF} and \cite[§3, Lemma 2.15, Corollary 2.18]{MoVo}.
\end{proof}

\begin{theo}\label{naivequadrics}
Let $n\geqslant 1$ be an integer and let $X$ be a smooth affine $k$-scheme. Then the map $\Hom_{\A^1}(X,Q_n)\rightarrow[X,Q_n]_{\A^1}$ is a bijection.
\end{theo}

\begin{proof}
See \cite[Theorem 4.2.1]{AHW2} for odd $n$---we may apply this theorem because $k$ is ind-smooth over its prime subfield which is perfect---and \cite[Corollary 3.1.1]{As} for $n$ even.
\end{proof}

\begin{rema}
In fact, more can be said about the quadrics $Q_n$: they are \emph{naive} spaces in the sense of \cite{AHW2}—this property is what is established in \cite[Theorem 4.2.1]{AHW2} and \cite[Corollary 3.1.1]{As} and it implies the above result.
\end{rema}

\paragraph*{$\A^1$-homotopy sheaves and $\A^1$-connectedness.} Recall that, in our context, an $h$-cogroup is a cogroup in the pointed $\A^1$-homotopy category $\mathcal{H}_\bullet(k)$. The simplicial spheres $\Sr^i$ for $i\geqslant 1$ have an $h$-cogroup structure in simplicial sets which yields an $h$-cogroup structure on $\Sr^i\wedge(\mathscr{U},*)$ for any pointed space $(\mathscr{U},*)$; this $h$-cogroup is cocommutative if $i\geqslant 2$, owing to the corresponding fact for $\Sr^i$. Observe that by the loop-suspension adjunction, we deduce that $\Omega(\mathscr{X},x)$ is an $h$-group for any pointed space $(\mathscr{X},x)$ and $\Omega^i(\mathscr{X},x)$ is a commutative $h$-group if $i\geqslant 2$. For more details on $h$-cogroups in the context of general model categories, we refer to \cite{Hovey}.

In particular, given an integer $i\geqslant 1$ and a pointed space $(\mathscr{X},x)$, we obtain a presheaf of groups on $\mathsf{Sm}_k$ by the formula $U\mapsto[\Sr^i\wedge(U_+),(\mathscr{X},x)]_{\A^1,\bullet}$. By definition, the Nisnevich sheafification of this presheaf is the $i$-th $\A^1$-homotopy sheaf of $(\mathscr{X},x)$ denoted by $\pi_i^{\A^1}(\mathscr{X},x)$ or $\pi_i^{\A^1}(\mathscr{X})$ if no confusion arises from this notation: it is a sheaf of groups for any $i\geqslant 1$ and a sheaf of \emph{abelian} groups if $i\geqslant 2$. We also denote by $\pi_0^{\A^1}(\mathscr{X},x)$ the Nisnevich sheaf associated with the presheaf $U\mapsto[\Sr^0\wedge(U_+),(\mathscr{X},x)]_{\A^1,\bullet}$ of pointed sets and again, we simplify the notation to $\pi_0^{\A^1}(\mathscr{X})$ if no confusion arises from this notation.

Let $\mathscr{X}$ and $\mathscr{Y}$ be pointed spaces, let $f:\mathscr{X}\rightarrow\mathscr{Y}$ be a map of pointed spaces and let $n$ be a non-negative integer. Then $\mathscr{X}$ is said to be \emph{$\A^1$-$n$-connected} (respectively \emph{$\A^1$-connected}, respectively \emph{$\A^1$-simply connected}) if $\mathbf{\pi}_i^{\A^1}(\mathscr{X})$ is trivial for all $i\leqslant n$ (respectively for $i=0$, respectively for all $i\leqslant 1$). Following \cite[p. 627]{AFCompEulerClass} (second-to-last paragraph before Theorem 4.1), we say that $f$ is $\A^1$-$n$-connected (respectively $\A^1$-connected, respectively $\A^1$-simply connected) if $\pi_i^{\A^1}(f):\pi_i^{\A^1}(\mathscr{X})\rightarrow\pi_i^{\A^1}(\mathscr{Y})$ is an isomorphism of sheaves for all $i\leqslant n$ (respectively for $i=0$, respectively for all $i\leqslant 1$).

This last definition can also be phrased in terms of the homotopy fibre $\mathscr{F}$ of $f$. The key property of homotopy fibres is that the $\A^1$-homotopy sheaves of $\mathscr{F}$ fit into a long exact sequence \[\cdots\rightarrow\pi_{i+1}^{\A^1}(\mathscr{Y})\rightarrow\pi_i^{\A^1}(\mathscr{F})\rightarrow\pi_i^{\A^1}(\mathscr{X})\rightarrow\pi_i^{\A^1}(\mathscr{Y})\rightarrow\pi_{i-1}^{\A^1}(\mathscr{F})\rightarrow\cdots\] \[\cdots\rightarrow\pi_1^{\A^1}(\mathscr{Y})\rightarrow\pi_0^{\A^1}(\mathscr{F})\rightarrow\pi_0^{\A^1}(\mathscr{X})\rightarrow\pi_0^{\A^1}(\mathscr{Y})\rightarrow *\] where exactness for sequences of sheaves of pointed sets is defined as usual: a sequence $A\xrightarrow{f}B\xrightarrow{g}C$ of pointed maps between pointed sheaves is exact at $B$ if $\Im f=g^{-1}(*)$ where $*$ is the point of $C$, and groups are pointed by their neutral elements. The following lemma is now clear:

\begin{lem}\label{lem:equivalent_condition_connectedness_maps}
Let $f:\mathscr{X}\rightarrow\mathscr{Y}$ be a pointed map of pointed spaces with homotopy fibre $\mathscr{F}$. Let $n$ be a positive integer. Then the following assertions are equivalent.
\begin{itemize}
	\item The map $f$ is $\A^1$-$n$-connected.
	\item The pointed space $\mathscr{F}$ is $\A^1$-$(n-1)$-connected and the map $\pi_{n+1}^{\A^1}(\mathscr{Y})\rightarrow\pi_n^{\A^1}(\mathscr{F})$ of sheaves is an epimorphism.
\end{itemize}
In particular, if $f$ is $\A^1$-$n$-connected, then $\mathscr{F}$ is $\A^1$-$(n-1)$-connected and if $\mathscr{F}$ is $\A^1$-$n$-connected, then $f$ is $\A^1$-$n$-connected.
\end{lem}

\begin{rema}
Beware that the above necessary and sufficient condition for $f$ to be $\A^1$-$n$-connected in terms of $\mathscr{F}$ is not quite standard, because of our (also non-standard) definition of $\A^1$-$n$-connectedness of maps. Usually, $f$ is defined to be $\A^1$-$n$-connected precisely if $\mathscr{F}$ is $\A^1$-$(n-1)$-connected; our definition asks for a stronger condition. We also thank the referee for pointing out that the homotopy fibre $\mathscr{F}$ being $\A^1$-$n$-connected is a sufficient but not a necessary condition. We prefer to keep the definition of $\A^1$-connectedness of maps introduced in \cite{AFCompEulerClass} because we appeal to the results of this paper. One agreeable consequence of the definition of \cite{AFCompEulerClass} is that if $f:\mathscr{X}\rightarrow\mathscr{Y}$ is $\A^1$-$n$-connected and if $\mathscr{X}$ (respectively $\mathscr{Y}$) is $\A^1$-$m$-connected with $m\leqslant n$, then $\mathscr{Y}$ (respectively $\mathscr{X}$) is $\A^1$-$m$-connected.
\end{rema}

\begin{exe}\label{exe:application_blakers_massey}
Let $k_0$ be a perfect field and let $u:\mathscr{X}_0\rightarrow\mathscr{Y}_0$ be a pointed map of pointed $k_0$-spaces. Assume that $\mathscr{X}$ is $\A^1$-simply connected, $\mathscr{Y}$ is $\A^1$-$(n-1)$-connected with $n\geqslant 3$ and the homotopy cofibre $\mathscr{C}$ of $u$ is $\A^1$-$(d+1)$-connected with $d\geqslant n+1$. Then the homotopy fibre $\mathscr{F}$ of $u$ is $\A^1$-$\delta$-connected where $\delta=\min(d,2n+1)$. Note that $u$ is $\A^1$-$\delta$-connected as a result but the claim on the homotopy fibre is sharper as observed in the previous remark. 

Indeed, since $k_0$ is perfect, we may apply the results of \cite{AFCompEulerClass} and in particular the Blakers--Massey theorem \cite[Theorem 4.1]{AFCompEulerClass}. Denote by $f$ the comparison map \cite[(4.6)]{AFCompEulerClass} of this theorem. Since $\mathscr{X}$ and $\mathscr{Y}$ are $\A^1$-simply connected, $u$ is $\A^1$-simply connected. According to the second item of \cite[Theorem 4.1]{AFCompEulerClass} which we may apply since $n-1\geqslant 2$, $f$ is $\A^1$-$(1+(n-1)+1)$-connected. Note that $\mathscr{C}$ is $\A^1$-$(d+1)$-connected thus the target of the comparison map $f$ is $\A^1$-$d$-connected. Since $d\geqslant n+1$ by assumption, we now deduce that $\mathscr{F}$ is $\A^1$-$(n+1)$-connected by the previous remark. In particular, $u$ is $\A^1$-$(n+1)$-connected and another application of the second item of the Blakers--Massey theorem shows that the comparison map $f$ is $\A^1$-$((n+1)+(n-1)+1)$-connected and is thus $\A^1$-$(2n+1)$-connected. It follows that $\pi_i^{\A^1}(\mathscr{F})\simeq\pi_{i}^{\A^1}(\Omega\mathscr{C})\simeq\pi_{i+1}^{\A^1}(\mathscr{C})$ for all $i\leqslant 2n+1$: in particular, if $i\leqslant d$, then $\pi_i^{\A^1}(\mathscr{F})$ is trivial and thus $\mathscr{F}$ is $\A^1$-$\delta$-connected where $\delta=\min(d,2n+1)$.
\end{exe}

\begin{lem}\label{lem:a1_conn_base_change}
Let $k/k_0$ be a field extension inducing a scheme morphism $f:\Spec k\rightarrow\Spec k_0$. Assume that $k_0$ is perfect. Let $\mathscr{X}_0$ be a pointed $k_0$-space and set $\mathscr{X}=f^*\mathscr{X}_0$, endowed with the induced base point as a $k$-space. Let $n$ be a non-negative integer. If $\mathscr{X}_0$ is an $\A^1$-$n$-connected $k_0$-space, then $\mathscr{X}$ is an $\A^1$-$n$-connected $k$-space.
\end{lem}

\begin{proof}
Let $i\leqslant n$. The morphism $f$ is essentially smooth by \cite[Lemma A.2]{Hoyois+2015+173+226}. Recall that the local rings for the Nisnevich topology on $\mathsf{Sm}_k$ are the henselian local $k$-algebras that are essentially of finite type. Let $A$ be such a $k$-algebra. Then $A$ is local henselian and essentially of finite type as a $k_0$-algebra because $f$ is essentially of finite type. Applying $\pi_0$ to the map of \cite[Lemma A.4, (1)]{Hoyois+2015+173+226}, we see that $\pi_i^{\A^1}(\mathscr{X})(A)=\pi_i^{\A^1}(\mathscr{X}_0)(A)$. Thus assuming that $\pi_i^{\A^1}(\mathscr{X}_0)$ is the trivial Nisnevich sheaf on $\mathsf{Sm}_{k_0}$, $\pi_i^{\A^1}(\mathscr{X})$ is the trivial Nisnevich sheaf on $\mathsf{Sm}_k$ which implies the claim.
\end{proof}

\begin{lem}\label{lem:a1_conn_base_change_map}
Let $k/k_0$ be a field extension with $k_0$ perfect inducing a scheme morphism $f:\Spec k\rightarrow\Spec k_0$, and let $u:\mathscr{X}_0\rightarrow\mathscr{Y}_0$ be a pointed map of pointed $k_0$-spaces. Let $n\geqslant 0$ be an integer. If $u$ is $\A^1$-$n$-connected, then the induced map $f^*u:f^*\mathscr{X}_0\rightarrow f^*\mathscr{Y}_0$ of pointed $k$-spaces is $\A^1$-$n$-connected.
\end{lem}

\begin{proof}
Modulo the identification of stalks of $\A^1$-homotopy sheaves described in the above proof, given a henselian local $k$-algebra $A$ essentially of finite type, the morphism $\pi_i^{\A^1}(f^*u):\pi_i^{\A^1}(f^*\mathscr{X}_0)(A)\rightarrow\pi_i^{\A^1}(f^*\mathscr{Y}_0)(A)$ coincides with $\pi_i^{\A^1}(u)$ for any $i$. Thus if $\pi_i^{\A^1}(u)$ is an isomorphism for all $i\leqslant n$, then $\pi_i^{\A^1}(f^*u)$ is an isomorphism for all $i\leqslant n$: this is the claim of Lemma \ref{lem:a1_conn_base_change_map}.
\end{proof}

Additionally, we say that $\mathscr{X}$ is \emph{simplicially} $m$-connected if its stalks (at points for the Nisnevich topology) are $m$-connected simplicial sets. For instance, for any pointed space $\mathscr{X}$, $\Sigma^{m+1}\mathscr{X}$ is simplicially $m$-connected. Now the following theorem is due to Morel.

\begin{theo}[Morel]\label{simpliciallyconn}
Let $m\in\N$ and let $\mathscr{X}$ be a pointed $k$-space; assume that $\mathscr{X}$ is the pullback to $k$ of a $k_0$-space $\mathscr{X}_0$ where $k/k_0$ is a field extension with $k_0$ perfect. If $\mathscr{X}_0$ is simplicially $m$-connected, then $\mathscr{X}$ is $\A^1$-$m$-connected.
\end{theo}

\begin{proof}
By \cite[Theorem 6.38]{Morel}, $\mathscr{X}_0$ is $\A^1$-$m$-connected and we conclude by Lemma \ref{lem:a1_conn_base_change}.
\end{proof}

\begin{prop}\label{connectednessquadrics}
Let $n$ be a positive integer. Then the quadrics $Q_{2n}$ and $Q_{2n+1}$ are $\A^1$-$(n-1)$-connected.
\end{prop}

\begin{proof}
The $k$-schemes $\A^{n+1}\setminus 0\simeq\Sigma^n\mathbb{G}_m^{\wedge(n+1)}$ and $(\mathbb{P}^1)^{\wedge n}\simeq\Sigma^n\mathbb{G}_m^{\wedge n}$ are defined over $\Z$ hence over the prime subfield $k_0$ of $k$ which is perfect and the $k_0$-spaces from which they are pulled back are simplicially $(n-1)$-connected since $\A^{n+1}\setminus 0\simeq\Sigma^n\mathbb{G}_m^{\wedge(n+1)}$ and $(\mathbb{P}^1)^{\wedge n}\simeq\Sigma^n\mathbb{G}_m^{\wedge n}$. Therefore they are $\A^1$-$(n-1)$-connected according to the previous theorem. Since $Q_{2n}$ and $Q_{2n+1}$ are defined over $\Z$, they are pulled back from the prime subfield of $k$, hence it suffices to show that $Q_{2n}$ and $Q_{2n+1}$ are $\A^1$-equivalent to simplicially $(n-1)$-connected pointed spaces. The proposition now follows from the equivalences from Theorem \ref{homotypequad}.
\end{proof}

\subparagraph*{$\A^1$-chain connectedness.} One convenient way to prove that a space is $\A^1$-connected in the sense described above is to show that it is $\A^1$-\emph{chain} connected according to the following definition.

\begin{defi}[\protect{\cite[p. 2000]{AsoMo}}]
Let $X$ be a smooth $k$-scheme. Let $L$ be a separable finitely generated extension of $k$; then an \emph{elementary $\A^1$-equivalence} between points $x_0$ and $x_1$ in $X(L)$ is a morphism $f:\A_L^1\rightarrow X$ such that $f(0)=x_0$ and $f(1)=x_1$; we denote by $\sim$ the equivalence relation on $X(L)$ generated by elementary $\A^1$-equivalence. We then say that $x$ and $x'$ in $X(L)$ are \emph{$\A^1$-equivalent} if $x\sim x'$.

We say that $X$ is \emph{$\A^1$-chain connected} if the quotient set $X(F)/{\sim}$ is a singleton for any separable finite type extension $F/k$.
\end{defi}

\begin{exe}\label{examplesaonechaincon}
It is easy to see that the product of $\A^1$-chain connected spaces is $\A^1$-chain connected: the point is that $\sim$ is compatible with products in an obvious sense. Furthermore, $\A^1$ is $\A^1$-chain connected, an elementary $\A^1$-equivalence between $a\in L$ and $b\in L$ being given by $t\mapsto(1-t)a+tb$ in coordinates for any separable finite type extension $L/k$; hence $\A^n$ is $\A^1$-chain connected for all $n$. 

Also note that if an object $V$ of $\mathsf{Sm}_k$ which is pointed as a space, that is, endowed with a $k$-point $x_0$, and if $V$ is the union of $\A^1$-chain connected open subschemes $U$ and $U'$ containing the base point of $V$, then $V$ is $\A^1$-chain connected. Indeed, if $L/k$ is a separable finite type extension and if $x\in V(L)$, then $x\in U(L)$ or $x\in U'(L)$: in each case, we see that $x\sim x_0$ in $V(L)$, so that $V(L)/{\sim}$ is equal to the class of $x_0$ for $\sim$ as required.
\end{exe}

The useful property of $\A^1$-chain connectedness as far as this paper is concerned is the following.

\begin{prop}[\protect{\cite[Proposition 2.2.7]{AsoMo}}]\label{aonechaincoimpaoneco}
Let $X$ be a smooth $k$-scheme that is pointed as a space. If $X$ is $\A^1$-chain connected, then $X$ is $\A^1$-connected, that is, $\pi_0^{\A^1}(X)=*$.
\end{prop}

\paragraph*{Strictly $\A^1$-invariant sheaves, Eilenberg-MacLane spaces and $\A^1$-cohomological dimension.} Recall the following fundamental definition from \cite{Morel}: a Nisnevich sheaf $\mathbf{F}$ of abelian groups on $\mathsf{Sm}_k$ is called \emph{strictly $\A^1$-invariant} if the map $\Hr^i(X,\mathbf{F})\rightarrow\Hr^i(X\times\A^1,\mathbf{F})$ induced by the projection on the first factor is a bijection for any $i\geqslant 0$ and any object $X$ of $\mathsf{Sm}_k$, where $\Hr$ denotes cohomology of abelian sheaves on the site $(\mathsf{Sm}_k,\Nis)$.

\begin{exe}
If $n\geqslant 2$ and if $\mathscr{X}$ is a pointed space pulled back from a perfect subfield of $k$, then $\pi_n^{\A^1}(\mathscr{X})$ is strictly $\A^1$-invariant (\cite[Corollary 6.2]{Morel}) and so is its twist by any torsor under a sheaf of groups.
\end{exe}

If $n\geqslant 1$ is an integer and $\mathbf{F}$ is a strictly $\A^1$-invariant sheaf, then there exists a pointed space $\Kr(\mathbf{F},n)$ called the $n$-th \emph{Eilenberg--MacLane space} of $\mathbf{F}$ with the following property mirroring the case of algebraic topology: $\pi_n^{\A^1}(\Kr(\mathbf{F},n))=\mathbf{F}$ and $\pi_i^{\A^1}(\Kr(\mathbf{F},n))$ is trivial for any $i\neq n$. It follows as in topology that $\Kr(\mathbf{F},n)$ represents the functor $\Hr^n(\text{--},\mathbf{F}):\mathsf{Sm}_k\rightarrow\mathsf{Ab}$. We may then define Nisnevich cohomology with coefficients in $\mathbf{F}$ for spaces by setting $\Hr^n(\mathscr{X},\mathbf{F})=[\mathscr{X},\Kr(\mathbf{F},n)]_{\A^1}$: it extends Nisnevich cohomology for smooth $k$-schemes in an obvious sense.

In \cite[Definition 1.1.4]{AF}, the authors introduce a notion of dimension for the objects of $\mathsf{Sm}_k$ using strictly $\A^1$-invariant sheaves: let $d$ be an integer; a smooth $k$-scheme $X$ has $\A^1$-cohomological dimension $\leqslant d$ if $\Hr^i(X,\mathbf{F})=0$ for any $i>d$ and any strictly $\A^1$-invariant sheaf $\mathbf{F}$ on $\mathsf{Sm}_k$; the $\A^1$-cohomological dimension $\mathrm{cd}_{\A^1}(X)$ of $X$ is then the smallest element of the set of integers $d$ such that $X$ has $\A^1$-cohomological dimension $\leqslant d$. One deduces immediately from its definition that the $\A^1$-cohomological dimension of $X$ is a finite integer and is in fact bounded above by the Nisnevich cohomological dimension of $X$, in particular by the Krull dimension of $X$ (\cite[Theorem 1.32]{Nisnevich1989}).

\begin{exe}\label{aonecohodimquad}
It follows from Item (3) of \cite[Proposition 1.1.5]{AF} that if $n>0$, then $Q_{2n}$ and $Q_{2n+1}$ are of $\A^1$-cohomological dimension $n$. In particular, since $\A^{n+1}\setminus 0$ is $\A^1$-equivalent to $Q_{2n+1}$, $\A^{n+1}\setminus 0$ is of $\A^1$-cohomological dimension $n$.
\end{exe}

\paragraph*{An obstruction theory lemma.} We will need the following easy consequence of the formalism of Postnikov towers, for which we could not find an exact reference although the result is of course entirely well-known (\emph{e.g.} we essentially spell out the reasoning of the first paragraph of the proof of \cite[Proposition 1.2.1]{AF} in a more general setting).

\begin{lem}\label{lemobsth}
Let $f:\mathscr{E}\rightarrow\mathscr{B}$ be a pointed map of pointed $\A^1$-connected spaces. Assume $f$ is $\A^1$-$m$-connected with $m\geqslant 1$. Then for any smooth $k$-scheme $X$ of $\A^1$-cohomological dimension $d\leqslant m$, the map \[f_*:[X,\mathscr{E}]_{\A^1}\rightarrow[X,\mathscr{B}]_{\A^1}\] is a bijection. 
\end{lem}

\begin{proof}
We apply \cite[Theorem 6.1.1]{AFSplittingOutisde}\footnote{Among the hypotheses of this theorem is that $f$ be an $\A^1$-fibration and that $\mathscr{B}$ be $\A^1$-fibrant. However, in view of the result we wish to prove, it is clear that we may replace $f$ by an $\A^1$-fibration $f':\mathscr{E'}\rightarrow\mathscr{B}$ which is such that $f$ factors through $f'$ by an acyclic cofibration $\mathscr{E}\rightarrow\mathscr{E}'$; similarly, we may replace $\mathscr{B}$ by an $\A^1$-weakly equivalent $\A^1$-fibrant object.}, whose notations we borrow, and we use the discussion which follows (p. 1055). According to this discussion, for any $i$, given a map $g:X\rightarrow\mathscr{E}^{(i)}$ (in $\mathcal{H}(k)$), the obstruction to the existence of a map $X\rightarrow\mathscr{E}^{(i+1)}$ lifting $X\rightarrow\mathscr{E}^{(i)}$ in the following way (eventual lift indicated with a dotted arrow): 
\begin{center}
\begin{tikzcd}
                                   & \mathscr{E}^{(i+1)} \arrow[d,"p^{(i)}"] \\
X \arrow [ru,dotted] \arrow[r,"g"] & \mathscr{E}^{(i)}
\end{tikzcd}
\end{center}
is the triviality of an element of the Nisnevich cohomology group $\Hr^{i+1}(X,\mathbf{\pi}_i^{\A^1}(\mathscr{F})(\mathscr{P}))$ where $\mathscr{F}$ is the homotopy fibre of $f$ and $\mathbf{\pi}_i^{\A^1}(\mathscr{F})(\mathscr{P})$ is a certain twist of the $\A^1$-homotopy sheaf $\mathbf{\pi}_i^{\A^1}(\mathscr{F})$; moreover, given the triviality of this element, the set of lifts of $g$ can be identified with a quotient of the Nisnevich cohomology group $\Hr^{i}(X,\mathbf{\pi}_i^{\A^1}(\mathscr{F})(\mathscr{P}))$.  It follows that if the groups $\Hr^{i+1}(X,\mathbf{\pi}_i^{\A^1}(\mathscr{F})(\mathscr{P}))$ and $\Hr^{i}(X,\mathbf{\pi}_i^{\A^1}(\mathscr{F})(\mathscr{P}))$ vanish, then the map $p_*^{(i)}:[X,\mathscr{E}^{(i+1)}]_{\A^1}\rightarrow[X,\mathscr{E}^{(i)}]_{\A^1}$ is a bijection. 

Since $\mathscr{F}$ is $\A^1$-$m$-connected by assumption, the sheaf $\mathbf{\pi}_i^{\A^1}(\mathscr{F})(\mathscr{P})$ is trivial for $i\leqslant m$ so that $p_*^{(i)}$ is indeed a bijection for all $i\leqslant m$. 

On the other hand, as noted still later in the discussion \cite[p. 1055]{AFSplittingOutisde}, since $m+1>\mathrm{cd}_{\A^1}(X)$ by assumption, the map $p_*^{(i)}$ is a bijection for all $i\geqslant m+1$: indeed, just as above, it suffices to show that the cohomology groups $\Hr^{i+1}(X,\mathbf{\pi}_i^{\A^1}(\mathscr{F})(\mathscr{P}))$ and $\Hr^{i}(X,\mathbf{\pi}_i^{\A^1}(\mathscr{F})(\mathscr{P}))$ are zero and this is true by assumption on $X$ since $\mathbf{\pi}_i^{\A^1}(\mathscr{F})(\mathscr{P})$ is strictly $\A^1$-invariant for all $i$ (\cite[Corollary 6.2]{Morel}; this is still true if $i=1$ since $\mathbf{\pi}_1^{\A^1}(\mathscr{F})(\mathscr{P})$ is in fact trivial in this case as $m\geqslant 1$). 

Hence the map $p_*^{(i)}:[X,\mathscr{E}^{(i+1)}]_{\A^1}\rightarrow[X,\mathscr{E}^{(i)}]_{\A^1}$ is a bijection for all $i$. Since $\mathscr{E}^{(0)}=\mathscr{B}$ and by \cite[Theorem 6.1.1, (iv)]{AFSplittingOutisde}, this implies that $[X,\mathscr{E}]_{\A^1}\rightarrow[X,\mathscr{B}]_{\A^1}$ is a bijection as required.
\end{proof}

\begin{rema}
In particular, the above result is true if $X$ is of Krull dimension $d\leqslant m$. Also note that the sheaf $\mathscr{P}$ used to twist the homotopy sheaves of the fibre $\mathscr{F}$ is a $\pi_1^{\A^1}(\mathscr{B})$-torsor, hence it is trivial if $\mathscr{B}$ is $\A^1$-simply connected which is the case in the applications that we have in mind: this slightly simplifies the presentation.
\end{rema}

\paragraph*{Milnor--Witt $\Kr$-theory.} The Milnor--Witt $\Kr$-theory of a field $F$, which was defined by Morel—the definition that we use is \cite[Definition 3.1]{Morel}—, is denoted by $\Kr_*^\MW(F)$. It is the $\Z$-graded unital associative ring given by the following generators and relations. There is a generator $[a]$ in degree $+1$ for each $a\in F^\times$, and a generator $\eta$ in degree $-1$. The relations are the following.
\begin{itemize}
	\item For every $a\in F$ such that $a$ and $1-a$ belong to $F^\times$, we have $[a][1-a]=0$.
	\item For every $a$ and $b$ in $F^\times$, we have $[ab]=[a]+[b]+\eta[a][b]$.
	\item For every $a\in F^\times$, we have $\eta[a]=[a]\eta$.
	\item The relation $\eta(2+[-1])=0$ holds.
\end{itemize}
Given $a\in F^\times$, we set $\langle a\rangle=1+\eta[a]$.

The unramified sheaf (in the sense of \cite[Chapter 2]{Morel}) on $\mathsf{Sm}_k$ induced by $\Kr_n^\MW$ is denoted by $\mathbf{K}_n^\MW$. It is a strictly $\A^1$-invariant sheaf of $\Kr_0^\MW(k)$-modules. We use the notations of \cite{FaLCWG} for Milnor--Witt $\Kr$-theory twisted by a graded line bundle (see \cite[Subsection 1.3]{FaLCWG} for the definition of this last notion). If $(L,a)$ is a graded line bundle on $F$, there is a simply transitive action of $F^\times$ on the complement $L^0$ of $0$ in $L$, thus a morphism $\Z[F^\times]\rightarrow\Z[L^0]$ where $\Z[G]$ is the group algebra of the group $G$ and $\Z[E]$ is the free abelian group on the set $E$, which is a $\Z[G]$-module if $E$ is a $G$-set; on the other hand, there is a morphism $\Z[F^\times]\rightarrow\Kr_*^\MW(F)$ given by $\alpha\mapsto\langle\alpha\rangle$. We then set \[\Kr_*^\MW(F,L,a)=\Kr_*^\MW(F)\otimes_{\Z[F^\times]}\Z[L^0].\] This is the Milnor--Witt $\Kr$-theory of $F$ twisted by $(L,a)$. Here it should be noted that the integer $a$ is decorative in the \emph{definition} of twisted Milnor--Witt $\Kr$-theory, though it is essential to keep track when using canonical isomorphisms such as switching $(L,a)\otimes_F(L',a')\simeq(L',a')\otimes_F(L,a)$ in the category of \emph{graded} vector bundles (see the paragraph following \cite[Definition 1.18]{FaLCWG}). Observe also that the choice of $l\in L^0$ induces an isomorphism $\Kr_*^\MW(F)\rightarrow\Kr_*^\MW(F,L,a)$ given by $\alpha\mapsto\alpha\otimes l$. In particular, Milnor--Witt $\Kr$-theory twisted by the trivial graded line bundle is naturally isomorphic to untwisted Milnor--Witt $\Kr$-theory. We use this isomorphism to write generic elements of $\Kr_*^\MW(F,L,a)$ as $\alpha\otimes l$ where $\alpha\in\Kr_*^\MW(F)$ and $l\in L^0$.

There is also a global version of Milnor--Witt $\Kr$-theory: if $(L,a)$ is a graded line bundle on a smooth $k$-scheme $X$, then the authors construct a sheaf $\mathbf{K}_n^\MW(X,L,a)$ on the Nisnevich site $\Nis(X)$ of $X$ in \cite[Subsection 2.2]{AFCompEulerClass}: just as in the case of fields, as far as the definition of twisted Milnor--Witt $\Kr$-theory is concerned, the integer $a$ is decorative and $\mathbf{K}_n^\MW(X,L,a)=\mathbf{K}_n^\MW(X,L)$. The sheaf $\mathbf{K}_n^\MW(X,L,a)$ coincides with the restriction of $\mathbf{K}_n^\MW$ to $\Nis(X)$ if $(L,a)$ is trivial. Cohomology is functorial with respects to twists: if $L$ is a line bundle over a smooth $k$-scheme $X$ and if $f:Y\rightarrow X$ is a scheme over $X$ with $Y$ a smooth $k$-scheme, then there is a homomorphism $f^*:\Hr^i(X,\Kbf_n^\MW(L,a))\rightarrow\Hr^i(Y,\Kbf_n^\MW(f^*L,a))$ (\cite[Construction 2.7]{AFCompEulerClass}).

The graded ring structure on Milnor--Witt $\Kr$-theory induces an operation $\Kr_j^\MW(F)\times\Kr_n^\MW(F)\rightarrow\Kr_{j+n}^\MW(F)$ for any field $F$ (for $(n,j)\in\Z\times\Z$) which in turn induces a morphism of sheaves $\mathbf{K}_j^\MW\times\mathbf{K}_jn^\MW\rightarrow\mathbf{K}_{j+n}^\MW$ on $\mathsf{Sm}_k$ (see for instance the end of \cite[Chapter 2, 1.1. Residues]{MWM}). By functoriality of sheaf cohomology with respect to the coefficient sheaf, we obtain a map \[\Hr^i(X,\mathbf{K}_j^\MW)\times\Hr^m(X,\mathbf{K}_n^\MW)\rightarrow\Hr^{i+m}(X,\mathbf{K}_{j+n}^\MW).\] We make use of this map in the following lemma (with $m=0$ and $n=1$) which shows that one can compute the cohomology of $X\times\mathbb{G}_m$ with coefficients in Milnor--Witt $\Kr$-theory in terms of the cohomology of $X$. Let $t$ be the parameter of $\mathbb{G}_m$. For any $X$, let $p_X:X\times\mathbb{G}_m\rightarrow X$ denote the projection morphism and denote by $t_X$ the pullback of the parameter $t$ to $X\times\mathbb{G}_m$. Observe that since $[t]$ belongs to $\Hr^0(\mathbb{G}_m,\Kbf_1^\MW)$, $[t_X]$ belongs to $\Hr^0(X\times\mathbb{G}_m,\Kbf_1^\MW)$.

\begin{lem}\label{cohomologyprodwithgm}
Let $X$ be a smooth $k$-scheme. Let $i\in\N$ and let $j\in\Z$. Then \[\Hr^i(X\times\mathbb{G}_m,\mathbf{K}_j^\MW)=p_X^*\Hr^i(X,\mathbf{K}_j^\MW)\oplus p_X^*\Hr^i(X,\mathbf{K}_{j-1}^\MW)[t_X]\] as a $\Kr_0^\MW(k)$-module. Moreover, the morphism $\iota:X\rightarrow X\times\mathbb{G}_m$ given by the invertible global section $1\in\mathscr{O}_{X}(X)^\times$ induces the splitting above.
\end{lem}

\begin{rema}
Observe that $p_X\circ\iota=\Id$ so that $p_X^*$ is injective. Moreover, the last statement means that the map $\iota$ satisfies $\iota^*(p_X^*\alpha+p_X^*\beta[t_X])=\alpha$.
\end{rema}

\begin{proof}
This is \cite[Chapter 2, Lemma 3.1.8]{MWM} for the first part of the statement and \cite[Chapter 2, Remark 3.1.9]{MWM} for the second.
\end{proof}

In particular, if $X=\Spec k$, one finds that \[\Hr^0(\mathbb{G}_m,\Kbf_1^\MW)\cong\Kr_1^\MW(k)\oplus\Kr_{0}^\MW(k)[t].\] 

Now for any pair $(i,j)$ with $i\in\N$ and $j\in\Z$, we denote by $\Kr_{i,j}$ the Eilenberg--Mac Lane space $\Kr(\mathbf{K}_j^\MW,i)$. Recall that given a pointed space $\mathscr{X}$, we set $\Hr^i(\mathscr{X},\Kbf_j^\MW)=[\mathscr{X},\Kr_{i,j}]_{\A^1,\bullet}$ which recovers the usual definition of the left-hand side in terms of sheaf cohomology if $\mathscr{X}$ is a pointed smooth $k$-scheme.

\begin{lem}\label{lem:cohomology_algebraic_suspension}
Let $i\geqslant 2$ and let $j\in\Z$. Let $X$ be a pointed smooth $k$-scheme. The map  $[X\wedge\mathbb{G}_m,\Kr_{i,j}]_{\A^1,\bullet}\rightarrow[X\times\mathbb{G}_m,\Kr_{i,j}]_{\A^1,\bullet}$ factors through the inclusion of $p_X^*\Hr^i(X,\mathbf{K}_{j-1}^\MW)[t_X]$ into an isomorphism \[\Hr^i(X\wedge\mathbb{G}_m,\Kbf_{j}^\MW)\xrightarrow{\cong}p_X^*\Hr^i(X,\Kbf_{j-1}^\MW)[t_X]\] of $\Kr_0^\MW(k)$-modules.
\end{lem}

\begin{proof} 
Let $Y$ be a pointed smooth $k$-scheme. The cofibre sequence $X\vee Y\rightarrow X\times Y\rightarrow X\wedge Y$ induces an exact sequence \[[\Sigma(X\times Y),\Kr_{i,j}]_{\A^1,\bullet}\rightarrow[\Sigma(X\vee Y),\Kr_{i,j}]_{\A^1,\bullet}\rightarrow[X\wedge Y,\Kr_{i,j}]_{\A^1,\bullet}\rightarrow[X\times Y,\Kr_{i,j}]_{\A^1,\bullet}\rightarrow[X\vee Y,\Kr_{i,j}]_{\A^1,\bullet}\] of $\Kr_0^\MW(k)$-modules. We have a pointed $\A^1$-weak equivalence $\Omega\Kr_{i+1,j}\simeq\Kr_{i,j}$, hence the last three terms in the above sequence may be read as \[[\Sigma(X\wedge Y),\Kr_{i+1,j}]_{\A^1,\bullet}\rightarrow[\Sigma(X\times Y),\Kr_{i+1,j}]_{\A^1,\bullet}\rightarrow[\Sigma(X\vee Y),\Kr_{i+1,j}]_{\A^1,\bullet}\] by adjunction. As noted previously when defining $\A^1$-homotopy sheaves, the suspension $\Sigma(X\times Y)$ is an $h$-cogroup in a natural way. The projection morphisms $p_Z:X\times Y\rightarrow Z$ and inclusions $i_Z:Z\rightarrow X\vee Y$ for $Z=X$ and $Z=Y$ yield by composition and after suspension morphisms \[\Sigma(X\times Y)\xrightarrow{\Sigma p_Z}\Sigma Z\xrightarrow{\Sigma i_Z}\Sigma(X\vee Y).\] We may now use the $h$-cogroup structure on $\Sigma(X\times Y)$ to multiply these morphisms, yielding a morphism $\Sigma(i_X\circ p_X)\cdot\Sigma(i_Y\circ p_Y):\Sigma(X\times Y)\rightarrow\Sigma(X\vee Y)$. By construction, this map splits the suspension $\Sigma(X\vee Y)\rightarrow\Sigma(X\times Y)$ of the inclusion $X\vee Y\rightarrow X\times Y$ (see also \cite[p. 72]{Morel}). Hence the morphism \[[X\times Y,\Kr_{i,j}]_{\A^1,\bullet}\cong[\Sigma(X\times Y),\Kr_{i+1,j}]_{\A^1,\bullet}\rightarrow[\Sigma(X\vee Y),\Kr_{i+1,j}]_{\A^1,\bullet}\cong[X\vee Y,\Kr_{i,j}]_{\A^1,\bullet}\] splits. Similarly and more directly, the morphism $[\Sigma(X\times Y),\Kr_{i,j}]_{\A^1,\bullet}\rightarrow[\Sigma(X\vee Y),\Kr_{i,j}]_{\A^1,\bullet}$ splits. In particular, it is surjective hence the morphism $[\Sigma(X\vee Y),\Kr_{i,j}]_{\A^1,\bullet}\rightarrow[X\wedge Y,\Kr_{i,j}]_{\A^1,\bullet}$ is the zero morphism so that the morphism $[X\wedge Y,\Kr_{i,j}]_{\A^1,\bullet}\rightarrow[X\times Y,\Kr_{i,j}]_{\A^1,\bullet}$ is injective. 

Assume now that $Y=\mathbb{G}_m$---in particular, $\cd_{\A^1}(\mathbb{G}_m)\leqslant\dim(\mathbb{G}_m)=1$ so that $\Hr^i(\mathbb{G}_m,\Kbf_{j+1}^\MW)=0$. Then \[[X\vee\mathbb{G}_m,\Kr_{i,j}]_{\A^1,\bullet}\xrightarrow{\cong}[X,\Kr_{i,j}]_{\A^1,\bullet}\oplus[\mathbb{G}_m,\Kr_{i,j}]_{\A^1,\bullet}\] with $[\mathbb{G}_m,\Kr_{i,j}]_{\A^1,\bullet}\cong\Hr^i(\mathbb{G}_m,\Kbf_j^\MW)=0$ by assumption and with the map $[X\vee\mathbb{G}_m,\Kr_{i,j}]_{\A^1,\bullet}\rightarrow[X,\Kr_{i,j}]_{\A^1,\bullet}$ induced by the structural map $X\rightarrow X\vee\mathbb{G}_m$ whose composition with the cofibration $X\vee\mathbb{G}_m\rightarrow X\times\mathbb{G}_m$ is the map $\iota$ of Lemma \ref{cohomologyprodwithgm}. Hence the sequence \[0\rightarrow[X\wedge\mathbb{G}_m,\Kr_{i,j}]_{\A^1,\bullet}\rightarrow[X\times\mathbb{G}_m,\Kr_{i,j}]_{\A^1,\bullet}\xrightarrow{\iota^*}[X,\Kr_{i,j}]_{\A^1,\bullet}\rightarrow 0\] is split-exact with $p_X^*:[X,\Kr_{i,j}]_{\A^1,\bullet}\rightarrow[X\times\mathbb{G}_m,\Kr_{i,j}]_{\A^1,\bullet}$ splitting $\iota_X^*$ again by Lemma \ref{cohomologyprodwithgm}. The lemma now follows directly from the cited lemma.
\end{proof}

\paragraph*{The $\A^1$-Brouwer degree and computations.}

Morel's \emph{$\A^1$-Brouwer degree} \cite[Corollary 6.43]{Morel} may be viewed an isomorphism \[\deg^{\A^1}:\Hom_{\mathcal{SH}(k)}(\Sigma_{\mathbb{P}^1}^\infty\Sr^0,\Sigma_{\mathbb{P}^1}^\infty\Sr^0)\xrightarrow{\cong}\Kr_0^\MW(k)\] of rings. Thus if a pointed space $\mathscr{X}$ is such that $\Sigma_{\mathbb{P}^1}^\infty\mathscr{X}\simeq\Sigma_{\mathbb{P}^1}^\infty\Sr^0$ in $\mathcal{SH}(k)$---we say that $\mathscr{X}$ has the $\mathbb{P}^1$-stable $\A^1$-homotopy type of the unit object---, then the $\A^1$-Brouwer degree of an element of $\Hom_{\mathcal{H}_\bullet(k)}(\mathscr{X},\mathscr{X})$ makes sense: by definition, it is the $\A^1$-Brouwer degree of the image of this element in $\Hom_{\mathcal{SH}(k)}(\Sigma_{\mathbb{P}^1}^\infty\mathscr{X},\Sigma_{\mathbb{P}^1}^\infty\mathscr{X})\simeq\Hom_{\mathcal{SH}(k)}(\Sigma_{\mathbb{P}^1}^\infty\Sr^0,\Sigma_{\mathbb{P}^1}^\infty\Sr^0)$ (the $\A^1$-Brouwer degree does not depend on the choice of the isomorphism $\Sigma_{\mathbb{P}^1}^\infty\mathscr{X}\simeq\Sigma_{\mathbb{P}^1}^\infty\Sr^0$).
 
Let $n\geqslant 1$ be an integer. The $\A^1$-homotopy type of $\A^{n+1}\setminus 0$ was described in Example \ref{exe:homotopy_type_punctured_a^n} shows that $\A^{n+1}\setminus 0$ has the $\mathbb{P}^1$-stable $\A^1$-homotopy type of the unit object so that the $\A^1$-Brouwer degree of an endomorphism of $\A^{n+1}\setminus 0$ makes sense. In fact, \cite[Corollary 6.43]{Morel} yields an isomorphism $[\A^{n+1}\setminus 0,\A^{n+1}\setminus 0]_{\A^1,\bullet}\cong\Kr_0^\MW(k)$ of rings: this isomorphism is induced by the functor $\Sigma_{\mathbb{P}^1}^\infty$, which provides a bijection $[\A^{n+1}\setminus 0,\A^{n+1}\setminus 0]_{\A^1,\bullet}\xrightarrow{\simeq}\Hom_{\mathcal{SH}(k)}(\Sigma_{\mathbb{P}^1}^\infty\Sr^0,\Sigma_{\mathbb{P}^1}^\infty\Sr^0)$, and by $\deg^{\A^1}$. The next lemma gives a concrete representative of elements of the type $\langle a\rangle\in\Kr_0^\MW(k)$ in $[\A^{n+1}\setminus 0,\A^{n+1}\setminus 0]_{\A^1,\bullet}$.

\begin{lem}\label{lem:explicit_rank_one_representative_2n+1}
Let $n\geqslant 1$. Then the map of schemes $f_a^n:\A^{n+1}\setminus 0\rightarrow\A^{n+1}\setminus 0$ defined by \[f_a^n(x_1,x_2,\ldots,x_{n+1})=(ax_1,x_2,\ldots,x_{n+1})\] satisfies $\deg^{\A^1}(f_a^n)=\langle a\rangle$ for any $a\in k^\times$.
\end{lem}
 
\begin{proof}
Note that $f_a^n$ clearly defines a pointed endomorphism of $\A^{n+1}\setminus 0$. Consider the morphism $g:\A^{n+1}\rightarrow\A^{n+1}$ sending $(y_1,\ldots,y_{n+1})$ to $(ay_1,y_2,\ldots,y_{n+1})$. It induces a finite morphism $u:\mathbb{P}^{n+1}\rightarrow\mathbb{P}^{n+1}$, where $\A^{n+1}$ is embedded into $\mathbb{P}^{n+1}=\Proj k[y_0,\ldots,y_{n+1}]$ by $D(y_0)$, such that $u([y_0:y_1:\cdots:y_{n+1}])=[y_0:ay_1:y_2:\cdots:y_{n+1}]$. Thus by \cite[Proposition 15]{KW}, the $\A^1$-Brouwer degree of the morphism $\overline{g}:\mathbb{P}^{n+1}/\mathbb{P}^n\rightarrow\mathbb{P}^{n+1}/\mathbb{P}^n$ is the $\A^1$-local Brouwer degree at $0$ of the morphism $g$. Using for instance the formula for this $\A^1$-local Brouwer degree derived from \cite[Main theorem, Lemma 9]{KW}, we see that $\deg^{\A^1}\overline{g}=\langle a\rangle$ (the matrix $(\frac{\partial g_i}{\partial x_j})_{i,j}$ is diagonal with the first coefficient equal to $a$ and the others to $1$).

On the other hand, there is a pushout square of spaces
\begin{center}
\begin{tikzcd}
\A^{n+1}\setminus 0 \arrow[r] \arrow[d] & \A^{n+1} \arrow[d] \\
\A^{n+1} \arrow[r]                      & \Sigma\A^{n+1}\setminus 0
\end{tikzcd}
\end{center}
since the left vertical and top horizontal arrows are cofibrations. There is an induced $\A^1$-equivalence $\Sigma\A^{n+1}\setminus 0\sim\A^{n+1}/\A^{n+1}\setminus 0$. Moreover, the inclusion $\A^{n+1}\rightarrow\mathbb{P}^{n+1}$ induces an $\A^1$-equivalence $\mathbb{A}^{n+1}/\A^{n+1}\setminus 0\rightarrow\mathbb{P}^{n+1}/(\mathbb{P}^{n+1}\setminus 0)$ by the purity theorem \cite[Theorem 2.23, §3, p. 115]{MoVo}. Moreover, the collapse map $\mathbb{P}^{n+1}/\mathbb{P}^n\rightarrow\mathbb{P}^{n+1}\setminus(\mathbb{P}^{n+1}\setminus 0)$ (collapsing $\mathbb{P}^n\subseteq\mathbb{P}^{n+1}\setminus 0$ to a point) is an isomorphism as is observed following the proof of \cite[Lemma 10]{KW}. Thus there is an $\A^1$-weak equivalence $\Sigma\A^{n+1}\setminus 0\simeq\mathbb{P}^{n+1}/\mathbb{P}^n$ and unwinding the definitions, we see that $\overline{g}=\Sigma f_a$. Since the $\A^1$-Brouwer degree is invariant under simplicial suspension by definition, this establishes the claim.
\end{proof}

Similarly, given an integer $n>1$, the $\A^1$-homotopy type of $Q_{n}$ that was described in Theorem \ref{homotypequad} shows that $Q_{n}$ has the $\mathbb{P}^1$-stable $\A^1$-homotopy type of the unit object so that the $\A^1$-Brouwer degree of an endomorphism of $Q_{n}$ makes sense. Here again, \cite[Corollary 6.43]{Morel} yields an isomorphism $[Q_{n},Q_{n}]_{\A^1,\bullet}\cong\Kr_0^\MW(k)$ of rings.

\begin{cor}\label{cor:explicit_rank_one_representative_2n+1_quadric}
Let $n\geqslant 1$. Then the endomorphism $g_a^n$ of $Q_{2n+1}$ given by \[g_a^n(x_1,\ldots,x_{n+1},y_1,\ldots,y_{n+1})=(ax_1,x_2,\ldots,x_{n+1},a^{-1}y_1,y_2,\ldots,y_{n+1})\] satisfies $\deg^{\A^1} g_a^n=\langle a\rangle$ for any $a\in k^\times$.
\end{cor}

\begin{proof}
Indeed, denoting by $p:Q_{2n+1}\rightarrow\A^{n+1}\setminus 0$ the standard Jouanolou device, we have $p\circ g_a^n=f_a^n\circ p$ where $f_a^n$ is the morphism of Lemma \ref{lem:explicit_rank_one_representative_2n+1}. Here $p$ is an $\A^1$-weak equivalence, thus a $\mathbb{P}^1$-stable weak equivalence, that is, an isomorphism in $\mathcal{SH}(k)$. Hence $\deg^{\A^1}g_a^n=\deg^{\A^1}f_a^n=\langle a\rangle$.
\end{proof}

\begin{lem}\label{lem:explicit_rank_one_representative_2n}
Let $n\geqslant 1$. Then the map of schemes $h_a^n:Q_{2n}\rightarrow Q_{2n}$ defined by \[h_a^n(x_1,\ldots,x_n,y_1,\ldots,y_n,z)=(ax_1,x_2,\ldots,x_n,a^{-1}y_1,y_2,\ldots,y_n,z)\] satisfies $\deg^{\A^1}(h_a^n)=\langle a\rangle$ for any $a\in k^\times$.
\end{lem}

\begin{proof}
If $n=1$, recall that the endomorphism $u_a:[x:y]\mapsto[ax:y]$ of $\mathbb{P}^1$ (pointed by $[1:0]$) has $\A^1$-Brouwer degree $\langle a\rangle$ by \cite[Lemma 6.3.4]{Morel_trieste}. The $\A^1$-weak equivalence $Q_2\sim\mathbb{P}^1$ described in \cite[Proposition 2.1.1]{ADF} is a morphism of schemes $\varphi:Q_2\rightarrow\mathbb{P}^1$ given by $(x,y,z)\mapsto[\frac{x}{z}:1]$ on $D(z)$ and by $(x,y,z)\mapsto[1:\frac{y}{1-z}]$ on $D(1-z)$; in particular, it is pointed if we point $Q_2$ by $(0,0,0)$. Consequently, we see that $u_a\circ\varphi=\varphi\circ h_a^1$ hence $\deg^{\A^1}(h_a^1)=\deg^{\A^1}(u_a)=\langle a\rangle$. For $n\geqslant 2$, then modulo the $\A^1$-weak equivalence $Q_{2n}\xrightarrow{\sim}\Sigma Q_{2n-1}$ described in \cite[Theorem 4.2.3]{ADF}, $h_a^n=\Sigma g_a^{n-1}$ where $g_a^{n-1}$ is the morphism of Corollary \ref{cor:explicit_rank_one_representative_2n+1_quadric}. Since the $\A^1$-Brouwer degree is invariant under simplicial suspension, this completes the proof.
\end{proof}

\begin{rema}
The careful reader will have noticed that $k$ actually has to be assumed to be infinite and perfect in the proof of Lemma \ref{lem:explicit_rank_one_representative_2n} in case $n>1$, notably because the intersection of the open subschemes $V_{2n}^i$ of the proof of \cite[Theorem 4.2.6]{ADF} has to have a rational point. There is a way to bypass these difficulties. Namely, there is an explicit $\A^1$-weak equivalence $Q_{2(n+1)}\sim Q_{2n}\wedge\mathbb{P}^1$ which is such that $h_a^{n+1}=h_a^n\wedge\Id_{\mathbb{P}^1}$ for all $n$ and over an arbitrary base field. The proof relies on the well-known suspension formula $\hocolim(X\leftarrow X\times Y\rightarrow Y)\sim\Sigma(X\wedge Y)$ (the morphisms of the homotopy pushout diagram are the projection maps) with $X=Q_{2n}$ and $Y=\mathbb{G}_m$. We do not have the space to discuss this further here but hope to come back to it in future work.
\end{rema}

\paragraph*{Recollections on the Rost--Schmid complex.} We still use the notions and notations of \cite{FaLCWG}, including the functor $D$ introduced in \cite[Subsection 1.3]{FaLCWG}: this functor sends a vector bundle $V$ over a scheme $X$ to the pair $D(V)=(\det V,\rk V)$ where $\det V$ is the top exterior power $\bigwedge^{\rk V} V$ of $V$.

Recall from \cite[Subsection 2.1]{FaLCWG} the Rost--Schmid complex which computes the Chow--Witt groups of a smooth $k$-scheme twisted by a (graded) line bundle. The fundamental building block is Milnor--Witt $\Kr$-theory twisted by a graded line bundle $\Kr_*^\MW(F,L,a)$. We then denote the cohomological Rost--Schmid complex associated with an object $X$ of $\mathsf{Sm}_k$, a graded line bundle $(L,a)$ over $X$ and an integer $i$ by $C_\RS(X,i,L,a)$: by definition, $\widetilde{\CH}^i(X,L,a)=\Hr^i(C_\RS(X,i,L,a))$. In degree $j\in\Z$, $C_\RS(X,i,L,a)^j$ is given by \[C_\RS(X,i,L,a)^j=\bigoplus_{x\in X^{(j)}}\Kr_{i-j}^\MW(k(x),\Lambda_x^X\otimes_{k(x)}(L(x),a)).\] Here $X^{(j)}$ denotes the set of points of codimension $j$ of $X$, $\Lambda_x^X=D(\m_x/\m_x^2)^{-1}$ for $x\in X^{(j)}$ ($\Lambda_x^X$ is thus graded by $j$) where $\m_x$ is the maximal ideal of the local ring $\mathscr{O}_{X,x}$, and $L(x)=L_x\otimes_{\mathscr{O}_{X,x}}k(x)$. We omit $(L,a)$ from the notation if $(L,a)$ is the trivial graded line bundle; we do the same for $X$ in $\Lambda_x^X$ unless there is possible confusion.

The complex $C_\RS(X,i,L,a)$ computes the cohomology of $X$ with values in the sheaf $\Kbf_i^\MW(X,L,a)$ on $\Nis(X)$ since it is a flasque resolution of this sheaf. Hence we may set \[\Hr^*(X,\Kbf_i^\MW(X,L,a))=\Hr^*(C_\RS(X,i,L,a)).\] For instance, $\Hr^n(X,\Kbf_n^\MW(X,L,a))=\widetilde{\CH}^n(X,L,a)$.

\section{Unimodular rows}\label{uni}

Recall the definition of unimodular rows:

\begin{defi}[unimodular row]
Let $A$ be a ring and let $n$ be a non-negative integer. A \emph{unimodular row of length $n$} with values in $A$ is an $n$-tuple $(a_1,\ldots,a_n)\in A^n$ such that the ideal generated by the $a_i$'s is $A$.
\end{defi}

The set of unimodular rows of length $n$ with values in $A$ is denoted by $\Um_n(A)$; it defines a functor from the category of rings to the category of sets in an obvious way. Unimodular rows may be thought of in (at least) two ways.
\begin{itemize}
	\item If $A$ is a ring, an $A$-morphism $f:A^n\rightarrow A$ corresponds to the choice of an $n$-tuple $(a_1,\ldots,a_n)$ of elements of $A$; the morphism $f$ is surjective if, and only if, the $n$-tuple $(a_1,\ldots,a_n)$ is a unimodular row.
	\item If $A$ is a $k$-algebra, an $n$-tuple of elements of $A$ corresponds to the choice of a $k$-morphism $X\rightarrow\A^n$ where $X=\Spec A$; the $n$-tuple is a unimodular row if, and only if, the morphism $X\rightarrow\A^n$ factors through the open subset $\A^n\setminus 0$.
\end{itemize}

Let $A$ be a ring and let $(a_1,\ldots,a_n)$ be an element of $\Um_n(A)$. Then the $A$-module homomorphism $f:A^n\rightarrow A$ naturally associated with $(a_1,\ldots,a_n)$ is surjective: since $A$ is a projective $A$-module, this implies that $f$ has a splitting, that is, there exists an $A$-module homomorphism $s:A\rightarrow A^n$ with $f\circ s=\Id_A$. The image $s(1)$ of $1$ by $s$ is a row $(\beta_1,\ldots,\beta_n)$ of elements of $A$ such that $a_1\beta_1+\cdots+a_n\beta_n=f\circ s(1)=1$. This explains the following definition:

\begin{defi}\label{splitting}
A \emph{splitting} of $(a_1,\ldots,a_n)$ is a row $(b_1,\ldots,b_n)$ of elements $A$ such that $\sum a_ib_i=1$.
\end{defi}

\begin{rema}
This notion of splitting allows one to produce a more explicit proof that the map $p_*$ from Lemma \ref{quadricjouandev} is surjective. Indeed, let $n$ be a positive integer and let $u:\Spec A\rightarrow\A^n\setminus 0$ be a morphism of $k$-schemes; we have to show that there exists a map $v:X\rightarrow Q_{2n-1}$ such that $p\circ v=u$. Then $u$ corresponds to a unimodular row $(a_1,\ldots,a_n)$: if $(b_1,\ldots,b_n)$ is a splitting of $a$, we obtain an element $(a,b)$ of $Q_{2n-1}(k[X])$, thus a morphism $v:X\rightarrow Q_{2n-1}$ which by construction satisfies $p\circ v=u$.
\end{rema}

Let $A$ be a ring. The general linear group $\GL_n(A)$ acts naturally on the right on $\Um_n(A)$, the action being induced by the natural action of $\GL_n(A)$ on $A^n$; thus, any subgroup of $\GL_n(A)$ acts on $\Um_n(A)$. It is in particular the case of the subgroup of \emph{elementary matrices}, denoted by $\Er_n(A)$ and defined as the subgroup of $\GL_n(A)$ generated by the set of matrices of the form $e_{ij}(\lambda)=\I_n+\lambda E_{ij}$ where $\I_n$ is the identity matrix, $\lambda$ is an element of $A$ and $E_{ij}$ ($i\neq j$) is the matrix whose entries are $0$ except the entry in the $i$-th row and $j$-th column which is equal to $1$. There exists a group structure on the orbit set $\Um_n(A)/\Er_n(A)$ under suitable dimension hypotheses of $A$ which van der Kallen described in terms of weak Mennicke symbols:

\begin{defi}[\protect{\cite[3.2 Notations]{vandK}}]
A \emph{weak Mennicke symbol} on $\Um_n(A)$ is a map $\phi:\Um_n(A)\rightarrow G$ from $\Um_n(A)$ to a group $G$ (denoted multiplicatively) which satisfies the following identities:
\begin{enumerate}
	\item if $v$ is a unimodular row of length $n$ and if $E$ is an elementary matrix, then $\phi(v)=\phi(vE)$;
	\item if $u=(q,a_2,\ldots,a_n)$ and $v=(1+q,a_2,\ldots,a_n)$ are unimodular rows of length $n$ with values in $A$ and if $r\in A$ is such that $r(1+q)=q$ mod $\langle a_2,\ldots,a_n\rangle$, then $\phi(u)=\phi(r,a_2,\ldots,a_n)\phi(v)$.
\end{enumerate}
\end{defi}

There exists a universal weak Mennicke symbol $\wms:\Um_n(A)\rightarrow\WMS_n(A)$ which induces a map \[\overline{\wms_n}:\Um_n(A)/\Er_n(A)\rightarrow\WMS_n(A)\] from the set of $\Er_n(A)$-orbits by virtue of the first item. The observation of van der Kallen is that $\overline{\wms_n}$ is a bijection if the stable dimension of $A$ is sufficiently small. We briefly mentioned the notion of stable dimension in the introduction. In the sequel, we denote the stable \emph{rank} by $\sr(A)$. Its definition is as follows: given a ring $A$, for any $r\geqslant 1$, denote by $(\mathrm{DS})_r$ the following assertion: 
\begin{quote}
$(\mathrm{DS})_r$: For any $(a_1,\ldots,a_{r+1})\in\Um_{r+1}(A)$, there exist $b_1,\ldots,b_r\in A$ such that the row $(a_1+b_1a_{r+1},\ldots,a_r+b_ra_{r+1})$ is unimodular.
\end{quote}
Then $\sr(A)$ is the smallest integer $r$ such that $A$ satisfies $(\mathrm{DS})_r$ and $\infty$ otherwise. It follows from \cite[Proposition 1]{BBproj} that if $\sr(A)<\infty$, then $A$ satisfies $(\mathrm{DS})_r$ for any $r\geqslant\sr(A)$. Finally, $\sdim(A)$ is defined by the equality $\sr(A)=\sdim(A)+1$ (\cite[Introduction, 1.2]{vdK}), hence $\sdim(A)=\infty$ if $\sr(A)=\infty$. The key property for us is the following:

\begin{exe}\label{exe:noetherian_stable_dim}
Let $A$ be a Noetherian ring of Krull dimension $d$. Then the inequality $\sdim(A)\leqslant d$ is satisfied according to \cite[Th\'eor\`eme 1]{BBproj}.
\end{exe}

\begin{theo}[van der Kallen, \protect{\cite[Theorem 3.6]{vandK}}]\label{thm:structure_vdK}
Let $A$ be a ring of stable dimension $d$ and let $n\geqslant 3$ be an integer such that $d\leqslant 2n-4$. Then the map $\overline{\wms_n}:\Um_n(A)/\Er_n(A)\rightarrow\WMS_n(A)$ is a bijection and the group $\WMS_n(A)$ is abelian.
\end{theo}

Let $n\geqslant 3$. In \cite{vdK}, several relations that may be satisfied by a map $\phi$ from $\Um_n(A)$ to a group $G$ are introduced. The relevant ones for us are the following.
\begin{itemize}
	\item[(MS1)] If $v\in\Um_n(A)$ and $E\in\Er_n(A)$, then $\phi(vE)=\phi(v)$.
	\item[(MS3)] If $(x,a_2,\ldots,a_n)$ and $(1-x,a_2,\ldots,a_n)$ are unimodular, then $\phi(x,a_2,\ldots,x_n)\phi(1-x,a_2,\ldots,a_n)=\phi(x(1-x),a_2,\ldots,a_n)$.
	\item[(MS4)] If $(f^2,a_2,\ldots,a_n)$ and $(g,a_2,\ldots,a_n)$ are unimodular, then $\phi(f^2,a_2,\ldots,a_n)\phi(g,a_2,\ldots,a_n)=\phi(f^2g,a_2,\ldots,a_n)$.
	\item[(MS5)] If $r(1+q)=q$ mod $\langle a_2,\ldots,a_n\rangle$ and if $(q,a_2,\ldots,a_n)$ is unimodular, then $\phi(r,a_2,\ldots,a_n)\phi(1+q,a_2,\ldots,a_n)=\phi(q,a_2,\ldots,a_n)$.
\end{itemize}
Theorem \ref{thm:structure_vdK} states that if $\sdim(A)\leqslant 2n-4$, then the universal map $\wms_n:\Um_n(A)\rightarrow\WMS_n(A)$ verifying (MS1) and (MS5) induces an bijection from $\Um_n(A)/\Er_n(A)$ to $\WMS_n(A)$ and that this group is abelian. On the other hand, (MS3) is closely linked to the well-known \emph{Mennicke--Newman lemma} which we now state:

\begin{prop}[Mennicke--Newman lemma]\label{mennnewlem}
Let $A$ be a commutative ring of stable dimension $d$ and let $n\geqslant 3$ be an integer such that $d\leqslant 2n-3$. Let $u=(u_1,\ldots,u_n)$ and $v=(v_1,\ldots,v_n)$ be unimodular rows of length $n$ with values in $A$. There then exist elementary matrices $E$ and $E'$ and elements $x,a_2,\ldots,a_n$ of $A$ such that $uE=(x,a_2,\ldots,a_n)$ and $vE'=(1-x,a_2,\ldots,a_n)$. 
\end{prop}

This is a generalised version of \cite[Lemma 3.2]{vdK}, although the version of the proof written below relies on the simple observation that the argument for this lemma holds in the greater generality of Proposition \ref{mennnewlem}. That $d\leqslant 2n-3$ (instead of the more usual condition $d\leqslant 2n-4$ as in \emph{e.g.} van der Kallen's theorem \cite[Theorem 3.6]{vandK}) suffices was also noticed in \cite[Lemma 3.2]{Sharma}.

\begin{proof}
The proof is divided in three steps.

(1) Observe that since $A$ is commutative, the row \[(u_2,\ldots,u_n,v_2,\ldots,v_n,u_1v_1)\] of length $2n-1$ is unimodular. Since $A$ is of stable dimension $d\leqslant 2n-3$, its stable rank $\sr(A)=\sdim(A)+1$ satisfies $\sr(A)\leqslant 2n-2$, hence by \cite[Proposition 1]{BBproj} $A$ satisfies $(\mathrm{DS})_{2n-2}$: there exist elements $\lambda_2,\ldots,\lambda_n$ and $\mu_2,\ldots,\mu_n$ of $A$ such that the row \[(u_2+\lambda_2u_1v_1,\ldots,u_n+\lambda_nu_1v_1,v_2+\mu_2u_1v_1,\ldots,v_n+\mu_nu_1v_1)\] of length $2n-2$ is unimodular. Since $(u_1,u_2+\lambda_1u_1v_1,\ldots,\lambda_nu_1v_1)$ (respectively $(v_1,v_2+\mu_2u_1v_1,\ldots,v_n+\mu_nu_1v_1)$) is in the orbit of $u$ (respectively of $v$) under the action of $\Er_n(A)$, we may assume that the row $(u_2,\ldots,u_n,v_2,\ldots,v_n)$ is unimodular of length $2n-2$. 

(2) We can then write \[u_1+v_1-1=\sum_{i=2}^n\alpha_iu_i+\beta_iv_i\] for some $\alpha_i\in A$ and $\beta_i\in A$. Replacing $u_1$ (respectively $v_1$) by $u_1-\sum\alpha_iu_i$ (respectively by $v_1-\sum\beta_iv_i$) which does not change the orbit of $u$ (respectively of $v$) under the action of $\Er_n(A)$, we may assume that $u_1+v_1=1$. 

(3) Finally assume that $(u_1,\ldots,u_n)$ and $(v_1,\ldots,v_n)$ are unimodular and that $u_1+v_1=1$. Let $x=u_1$. Since $u_1+v_1=1$, the equality $u_i-v_i-(u_i-v_i)(u_1+v_1)=0$ holds; therefore, \[u_i-u_1(u_i-v_i)=v_i+v_1(u_i-v_i)=a_i.\] By definition, the rows $(u_1,u_2,\ldots,u_n)$ and $(x=u_1,a_2=u_2-u_1(u_i-v_i),\ldots,a_n=u_n-u_1(u_n-v_n))$ (respectively $(v_1,v_2,\ldots,v_n)$ and $(1-x=v_1,a_2=v_2+v_1(u_2-v_2),\ldots,a_n=v_n+v_1(u_n-v_n))$) are equivalent. As a result, the elements $x$ and $a_i$ of $A$ thus defined satisfy the claim.
\end{proof}

For any commutative ring $A$, since the relation (MS4) is satisfied in $\WMS_n(A)$ \cite[Lemma 3.5 (v)]{vandK}, \cite[Lemma 3.1]{vdK} implies that (MS3) is satisfied in $\WMS_n(A)$. In particular, denoting by $\phi:\Um_n(A)\rightarrow G$ the universal map satisfying (MS1) and (MS3), there is an induced group homomorphism $\tau:G\rightarrow\WMS_n(A)$ such that $\tau\circ\phi=\wms_n$.

\begin{theo}[\protect{\cite[Theorem 3.3]{vdK}}]\label{thm:change_generators_unimodular_rows}
The map $\tau$ is an isomorphism if $A$ is Noetherian of dimension $d$ such that $2\leqslant d\leqslant 2n-4$.
\end{theo}

In fact, we argue that the proof of \cite[Theorem 3.3]{vdK} works for any ring $A$ of stable dimension $\sdim(A)=d\leqslant 2n-4$ (rings satisfying the hypotheses of Theorem \ref{thm:change_generators_unimodular_rows} above are of stable dimension $\leqslant 2n-4$ by Example \ref{exe:noetherian_stable_dim}). Indeed, assume that $A$ is of stable dimension $d\leqslant 2n-4$. In particular, $\wms_n=\tau\circ\phi$ is surjective by Theorem \ref{thm:structure_vdK} and thus $\tau$ is surjective. 

It remains to be proven that $\tau$ is injective. Observe that since $A$ is commutative, if $u=(x,a_2,\ldots,a_n)$ and $v=(1-x,a_2,\ldots,a_n)$ are elements of $\Um_n(A)$, then $\phi(u)\phi(v)=\phi(w)=\phi(v)\phi(u)$ in the universal group $G$ for (MS1) and (MS3), where $w=(x(1-x),a_2,\ldots,a_n)$. Note in particular that $\phi(u)\phi(v)$ is of the form $\phi(w)$ for some $w\in\Um_n(A)$. For \emph{any} ordered pair $(u,v)$ of elements of $\Um_n(A)$, by Proposition \ref{mennnewlem}, there exists $u'=(x,a_2,\ldots,a_n)$ and $v'=(1-x,a_2,\ldots,a_n)$ in $\Um_n(A)$ such that $u'=uE_1$ and $v'=vE_2$ for some $E_1$ and $E_2$ in $\Er_n(A)$. Thus $\phi(u)=\phi(u')$ and $\phi(v)=\phi(v')$ by (MS1) from which it follows according to the previous observation that $\phi(u)$ and $\phi(v)$ commute in $G$. Thus since $G$ is generated by $\phi(\Um_n(A))$, $G$ is abelian: from now on, we write $G$ additively. We also deduce from the previous observation that for any pair $(u,v)$ of elements of $\Um_n(A)$, $\phi(u)+\phi(v)$ is of the form $\phi(w)$ for some $w\in\Um_n(A)$: in other words, $\phi(\Um_n(A))$ is a submonoid of $G$. Now any element $g$ of $G$ may be written as $\sum_{i\in I}a_i\phi(u_i)$ for some finite set $I$, some $a_i\in\Z$ non-zero and some $u_i\in\Um_n(A)$. Grouping the elements of the form $a_i\phi(u_i)$ with $a_i>0$ on the one hand and with $a_i<0$ on the other hand and using the fact that $\phi(\Um_n(A))$ is a submonoid of $G$, we see that $g=\phi(w)-\phi(w')$ for some $w$ and $w'$ in $\Um_n(A)$. Finally assume that $\tau(g)=0$ with $g=\phi(w)-\phi(w')$ where $w$ and $w'$ lie in $\Um_n(A)$. Then $\wms_n(w)-\wms_n(w')=0$ hence $\wms_n(w)=\wms_n(w')$. It follows that $w$ and $w'$ lie in the same orbit under the action of $\Er_n(A)$, thus $\phi(w)=\phi(w')$ in view of (MS1): consequently, $g=\phi(w)-\phi(w')=0$ and $\tau$ is injective as claimed.

\begin{rema}
In particular, as far as we understand, the hypotheses that $A$ is Noetherian or of Krull dimension $d$ with $d$ satisfying $2\leqslant d$ in \cite[Theorem 3.3]{vdK} are not optimal and may be replaced with \emph{of stable dimension $\leqslant 2n-4$}. It is however the case that $A$ must be assumed to be Noetherian and its Krull dimension $d$ to be greater than or equal to $2$ in \cite{vdK} for other purposes of this paper, namely the study of the group homomorphism $\Um_{d+1}(A)/\Er_{d+1}(A)\rightarrow\Er^d(A)$ (with $\Er^d(A)$ the Euler class group of \cite{bhatwadekar_sridharan_2000} that we return to in the final section) in terms of weak Mennicke symbols for $A$ Noetherian of dimension $d$. Indeed we must then have $n=d+1\geqslant 3$ as this group homomorphism is understood with the left-hand side having the group structure of \cite{VANDERKALLEN1983363} which is constructed under these hypotheses.
\end{rema}

Instead of weak Mennicke symbols, we work with the universal map $\Um_n(A)/\Er_n(A)\rightarrow G$ induced by (MS1) and (MS3): since the rings of the main theorem \ref{groupiso} are Noetherian of dimension $d\leqslant 2n-4$, in the cases we consider, by Theorem \ref{thm:structure_vdK} and the discussion following Theorem \ref{thm:change_generators_unimodular_rows}, this map is a bijection onto an abelian group and thus endows $\Um_n(A)/\Er_n(A)$ with the structure of an abelian group that coincides with van der Kallen's. Concretely, the law on $\Um_n(A)/\Er_n(A)$ goes as follows. Given a unimodular row $u=(u_1,\ldots,u_n)$ in $\Um_n(A)$, we denote by $[u]$ or by $[u_1,\ldots,u_n]$ its orbit in $\Um_n(A)/\Er_n(A)$. If $u$ and $v$ are elements of $\Um_n(A)$, then there exist elements $x$ and $a_2,\ldots,a_n$ of $A$ such that $(x,a_2,\ldots,a_n)=u'$ and $(1-x,a_2,\ldots,a_n)=v'$ are unimodular and such that $[u]=[u']$ and $[v]=[v']$: we then have $[u]+[v]=[x(1-x),a_2,\ldots,a_n]$ in $\Um_n(A)/\Er_n(A)$. 

\section{Cohomotopy groups}\label{coh}

\subsection*{The motivic Borsuk group law}

We now turn to the cohomotopy group structure on $[U,\A^n\setminus 0]_{\A^1}$. Let $\mathscr{X}$ be an pointed space. Let $\nabla:\mathscr{X}\vee\mathscr{X}\rightarrow\mathscr{X}$ be the fold map and let $\Delta_{U}:U\rightarrow U\times U$ be the diagonal map. The idea of Borsuk is the following. Let $f$ and $g$ be morphisms from $U$ to $\mathscr{X}$. Then we are in the following situation.
\begin{center}
\begin{tikzcd}
                     &                                 & \mathscr{X}\vee\mathscr{X} \arrow[d] \arrow[r,"\nabla"] & \mathscr{X}\\ 
U \arrow[r,"\Delta_U"] & U\times U \arrow[r,"f\times g"] & \mathscr{X}\times\mathscr{X} 
\end{tikzcd}
\end{center}
If we wished to define a sum of $f$ and $g$, we could proceed as follows: assuming that $(f\times g)\circ\Delta_U$ factors uniquely (up to homotopy) through $\mathscr{X}\vee\mathscr{X}$, it induces a well-defined morphism $\overline{(f,g)}:U\rightarrow\mathscr{X}\vee\mathscr{X}$ (up to homotopy) and we may want to define the sum of $f$ and $g$ as $\nabla\circ\overline{(f,g)}:U\rightarrow\mathscr{X}$. Unfortunately, in general, there does not exist a lift $U\rightarrow\mathscr{X}\vee\mathscr{X}$ and, if it exists, it is not necessarily unique (even up to homotopy). However, under suitable $\A^1$-cohomological dimension and $\A^1$-connectedness assumptions on $U$ and $\mathscr{X}$ respectively (and a technical assumption on $\mathscr{X}$ which will be of no consequence to the material of this paper), there indeed exists a lift and it is unique.

\begin{prop}[\protect{\cite[Proposition 1.2.3]{AF}}]\label{bijgrpebor}
Let $\mathscr{X}$ be a pointed $k$-space pulled back from a pointed space $\mathscr{X}_0$ defined over a perfect subfield $k_0$ of $k$, let $n$ be an integer such that $n\geqslant 2$. If $\mathscr{X}$ is $\mathbb{A}^1$-$(n-1)$-connected, then composition with the natural morphism $\mathscr{X}\vee\mathscr{X}\rightarrow\mathscr{X}\times\mathscr{X}$ induces a bijection \[[U,\mathscr{X}\vee\mathscr{X}]_{\mathbb{A}^1}\rightarrow[U,\mathscr{X}\times\mathscr{X}]_{\A^1}\] for smooth $k$-scheme $U$ of $\A^1$-cohomological dimension $d\leqslant 2n-2$.
\end{prop}

\begin{proof}
According to Lemma \ref{lemobsth}, it suffices to show that the inclusion map $u:\mathscr{X}\vee\mathscr{X}\rightarrow\mathscr{X}\times\mathscr{X}$ is $\A^1$-$(2n-2)$-connected. We note that if $k_0$ is a perfect subfield of $k$ and $\mathscr{X}$ is pulled back from a $k_0$-space $\mathscr{X}_0$, then $u$ is pulled back from the inclusion map $u_0:\mathscr{X}_0\vee\mathscr{X}_0\rightarrow\mathscr{X}_0\times\mathscr{X}_0$. By Lemma \ref{lem:a1_conn_base_change_map}, we can replace $k$ by $k_0$ and $\mathscr{X}$ by $\mathscr{X}_0$ in proving the assertion claimed in the first sentence. We can now appeal to \cite[Corollary 3.3.11]{AWW} whose proof relies on the Blakers--Massey theorem as follows. Since $\mathscr{X}$ is $\A^1$-simply connected, so are $\mathscr{X}\vee\mathscr{X}$ and $\mathscr{X}\times\mathscr{X}$ and thus the homotopy fibre $\mathscr{F}$ of $u$ is $\A^1$-connected. In particular, by the first item of \cite[Theorem 4.1]{AFCompEulerClass}, the comparison map of the cite theorem is $\A^1$-$2$-connected. Since the homotopy cofibre of $u$ is $\mathscr{X}\wedge\mathscr{X}$ which is $\A^1$-$(2n-2)$-connected, the target of the comparison map of the Blakers--Massey theorem is $\A^1$-$(2n-2)$-connected. If $n=2$, this shows that $\mathscr{F}$ is $\A^1$-$2$-connected and thus $u$ is $\A^1$-$(2n-2)$-connected which is the claim. Else, by Example \ref{exe:application_blakers_massey}, $u$ (in fact $\mathscr{F}$) is $\A^1$-$\delta$-connected where $\delta=\min(2n-2,2n+1)=2n-2$ which completes the proof.
\end{proof}

Under the assumptions of Proposition \ref{bijgrpebor}, we can define $f+g$ by the procedure described above using this proposition: we denote the resulting morphism $U\rightarrow\mathscr{X}$ by $\tau(f,g)$. In \cite{AF}, the authors show that $\tau$ coincides with another group law defined in terms of a group structure on $[U_+,\Omega^i\Sigma^i\mathscr{X}]_{\A^1,\bullet}$ (\cite[Proposition 1.2.1]{AF}) which makes $[U,\mathscr{X}]_{\A^1}$ into an abelian group, see \cite[Proposition 1.2.5]{AF}.

\begin{rema}
Although we have to fix \emph{a} base point in $\mathscr{X}$ for this construction, the \emph{choice} of said base point has no consequence on the law induced on $[U,\mathscr{X}]_{\A^1}$ because the assumptions of Proposition \ref{bijgrpebor} imply that $\mathscr{X}$ is $\A^1$-simply connected.
\end{rema}

If $d\geqslant 0$ is an integer, we let $\mathsf{Sm}_k^{\cd_{\A^1}\leqslant d}$ be the full subcategory of $\mathsf{Sm}_k$ whose objects are those with $\A^1$-cohomological less than or equal to $d$.

\begin{lem}\label{bifunctoriality}
Let $n$ be an integer such that $n\geqslant 2$.
\begin{enumerate}
	\item Let $\mathscr{X}$ be a pointed $\A^1$-$(n-1)$-connected $k$-space pulled back from a space defined over a perfect subfield of $k$. Let $f:V\rightarrow U$ be a morphism of the category $\mathsf{Sm}_k^{\cd_{\A^1}\leqslant 2n-2}$. Then the map $f^*:[U,\mathscr{X}]_{\A^1}\rightarrow[V,\mathscr{X}]_{\A^1}$ is a group homomorphism. Consequently, the functor $[\text{–},\mathscr{X}]_{\A^1}:U\mapsto[U,\mathscr{X}]_{\A^1}$ from $\mathsf{Sm}_k^{\cd_{\A^1}\leqslant 2n-2}$ factors through the forgetful functor into a functor to the category $\mathsf{Ab}$ of abelian groups.
	\item Let $\mathscr{X}$ and $\mathscr{Y}$ be pointed $\A^1$-$(n-1)$-connected $k$-spaces pulled back from spaces defined over a perfect subfield of $k$. Let $p:\mathscr{Y}\rightarrow\mathscr{X}$ be a map of pointed spaces. Then $p_*$ induces a natural transformation between the functors $\mathsf{Sm}_k^{\cd_{\A^1}\leqslant 2n-2}\xrightarrow{[\text{–},\mathscr{Y}]_{\A^1}}\mathsf{Ab}$ and $\mathsf{Sm}_k^{\cd_{\A^1}\leqslant 2n-2}\xrightarrow{[\text{–},\mathscr{X}]_{\A^1}}\mathsf{Ab}$.
\end{enumerate}
\end{lem}

\begin{proof}
Both statements easily follow from the definition of Borsuk's law.
\end{proof}

\begin{defi}
Let $n$ be an integer such that $n\geqslant 2$. If $\mathscr{X}$ is a pointed $\A^1$-$(n-1)$-connected space that is pulled back from a perfect subfield of $k$, we call the functor $[\text{–},\mathscr{X}]_{\A^1}:\mathsf{Sm}_k^{\cd_{\A^1}\leqslant 2n-2}\xrightarrow{U\mapsto[U,\mathscr{X}]_{\A^1}}\mathsf{Ab}$ the \emph{motivic cohomotopy theory} or simply the cohomotopy theory defined by $\mathscr{X}$.
\end{defi}

We also refer to the statement of the second item of Lemma \ref{bifunctoriality} as \emph{functoriality of (motivic) cohomotopy theories}. Our comparison of cohomotopy theories in Section \ref{compcohothe} relies on this functoriality statement. Observe also that in case $\mathscr{X}$ is itself a smooth $k$-scheme of $\A^1$-cohomological dimension at most $2n-2$, it easily follows from the functoriality of cohomotopy groups and of cohomotopy theories that $[\mathscr{X},\mathscr{X}]_{\A^1}$ may naturally be endowed with the structure of a commutative ring with multiplication given by composition.

Let $n\geqslant 3$. The space $\mathscr{X}=\A^n\setminus 0$ pointed by $(0,\ldots,0,1)$ is defined over $\Z$, in particular over the prime subfield of $k$ which is perfect. Moreover, it is $\A^1$-$(n-2)$-connected since $\A^n\setminus 0\simeq\Sigma^{n-1}\mathbb{G}_m^{\wedge n}$ in $\mathcal{H}_\bullet(k)$. Thus if $U$ is a scheme of $\A^1$-cohomological dimension $d\leqslant 2n-4$, Proposition \ref{bijgrpebor} allows us to equip the set $[U,\A^n\setminus 0]_{\A^1}$ with the structure of an abelian group (\cite[Proposition 1.2.5]{AF}; we give more details in the following subsection).

\subsection*{Comparing homotopy and cohomotopy}

The results of this subsection and of the following subsection are not used until the last section.

Assume now that $n\geqslant 2$ and that $(Q,*)$ is a pointed smooth $k$-scheme which is $\A^1$-$(n-1)$-connected and of $\A^1$-cohomological dimension less than or equal to $2n-2$. We also assume that $(Q,*)$ is endowed with the structure of a cocommutative $h$-cogroup; we denote by $c:Q\rightarrow Q\vee Q$ the comultiplication (in $\mathcal{H}_\bullet(k)$). The main example that we have in mind is $Q_{2n}\simeq\Sigma^n\mathbb{G}_m^n$, which, as a second suspension, is a commutative $h$-cogroup, as we discussed when defining $\A^1$-homotopy sheaves and of $\A^1$-cohomological dimension at most $2n-2$ by Example \ref{aonecohodimquad}. Then there are \emph{a priori} two group structures on the set $[Q,Q]_{\A^1}$ of (unpointed) morphisms in $\mathcal{H}(k)$.
\begin{enumerate}
	\item On the one hand, since $Q$ is $\A^1$-simply connected, according to \cite[Lemma 2.1]{AFAlgVecBunOnSph}, the obvious map $[(Q,*),(Q,*)]_{\A^1,\bullet}\rightarrow[Q,Q]_{\A^1}$ is a bijection. Since the source is a group thanks to the $h$-cogroup structure on $(Q,*)$, the target inherits a group structure. This is the \emph{homotopical} group structure.
	\item The other is given by the motivic Borsuk group law described above. This is the \emph{cohomotopical} group structure. As we alluded to following the proof of Proposition \ref{bijgrpebor}, it may alternatively be described as follows. A small adaptation of the proof of \cite[Lemma 2.1]{AFAlgVecBunOnSph} shows that the natural map $[Q_+,(Q,*)]_{\A^1,\bullet}\rightarrow[Q_+,Q]_{\A^1}\rightarrow[Q,Q]_{\A^1}$ induced by the inclusion $Q\rightarrow Q_+$ is a bijection. The unit $(Q,*)\rightarrow\Omega^2\Sigma^2(Q,*)$ of the loop-suspension adjunction induces a map $[Q_+,(Q,*)]_{\A^1,\bullet}\rightarrow[Q_+,\Omega^2\Sigma^2(Q,*)]_{\A^1,\bullet}$ where $\Omega^2\Sigma^2(Q,*)$ is a commutative $h$-group, hence the target of this map is a group. Since $Q$ is of $\A^1$-cohomological dimension at most $2n-2$, this map is a bijection by \cite[Proposition 1.2.1]{AF} hence an abelian group structure on $[Q_+,(Q,*)]_{\A^1,\bullet}$. Thus we obtain an abelian group structure on $[Q,Q]_{\A^1}$: all in all, it is induced by the sequence of bijections \[[Q,Q]_{\A^1}\xrightarrow{\cong}[Q_+,(Q,*)]_{\A^1,\bullet}\xrightarrow{\cong}[Q_+,\Omega^2\Sigma^2(Q,*)]_{\A^1,\bullet}.\] This law coincides with the motivic Borsuk group law induced by Proposition \ref{bijgrpebor}, see \cite[Proposition 1.2.5]{AF}.
\end{enumerate}
But the following lemma holds.

\begin{lem}\label{cohomot=homot}
The homotopical and cohomotopical group structures on $[Q,Q]_{\A^1}$ coincide.
\end{lem}

\begin{proof}
The crux of the proof is that both structures can be compared with structures on hom-sets in the \emph{$\Sr^1$-stable} homotopy category which is additive, hence supports essentially only one (functorially well-behaved) group structure on hom-sets (Lemma \ref{lem:uniqueness_cogroup_additive_cat}).

Let us be more precise. First, all hom-sets in $\mathcal{SH}^{\Sr^1}(k)$ are understood to be endowed with the group structure coming with the additive structure on the category $\mathcal{SH}^{\Sr^1}(k)$. Next, denote by $[(Q,*),(Q,*)]_{\A^1,\bullet}^{hom}$ the set $[(Q,*),(Q,*)]_{\A^1,\bullet}$ endowed with the homotopical group structure induced by the $h$-cogroup structure given on $(Q,*)$. The infinite $\Sr^1$-suspension functor $\Sigma^\infty$ commutes with finite coproducts hence the map $\Sigma^\infty:[(Q,*),(Q,*)]_{\A^1,\bullet}^{hom}\rightarrow[\Sigma^\infty(Q,*),\Sigma^\infty(Q,*)]_{\A^1,\bullet}$ induced by the functor $\Sigma^\infty$ is a group homomorphism according to Lemma \ref{functorscommcoprodcogrp} (and ultimately Lemma \ref{lem:uniqueness_cogroup_additive_cat}). Moreover, the map $\Sigma^\infty$ is bijective by the same argument as in \cite[Proposition 1.2.2]{AF}\footnote{We thank the referee for pointing out to us that we have to assume that $Q$ is a smooth $k$-scheme or at least that $Q_+$ is $\omega$-compact as an object of $\mathcal{SH}^{\Sr^1}(k)$ for this argument to work.} hence it is a group isomorphism.

On the other hand, denote by $[Q_+,(Q,*)]_{\A^1,\bullet}^{coh}$ the set $[Q_+,(Q,*)]_{\A^1,\bullet}$ endowed with the cohomotopical group structure. Then given any $i\geqslant 2$, the unit of the loop-suspension adjunction at the object $\Sigma^i\mathscr(Q,*)$ yields a map $\Sigma^i(Q,*)\rightarrow\Omega\Sigma^{i+1}(Q,*)$ and thus by application of the functor $\Omega^i$ a map $\Omega^i\Sigma^i(Q,*)\rightarrow\Omega^{i+1}\Sigma^{i+1}(Q,*)$ in $\mathcal{H}_\bullet(k)$. The diagram
\begin{center}
\begin{tikzcd}
\protect{(Q,*)} \arrow[d] \arrow[rd]      & \\
\Omega^i\Sigma^i\protect{(Q,*)} \arrow[r] & \Omega^{i+1}\Sigma^{i+1}\protect{(Q,*)}
\end{tikzcd}
\end{center} 
whose vertical (respectively diagonal) arrow is the unit of the adjunction of $(\Sigma^i,\Omega^i)$ (respectively of the adjunction of $(\Sigma^{i+1},\Omega^{i+1})$) is then commutative. By the second item in the above discussion, the map $[Q_+,(Q,*)]_{\A^1,\bullet}^{coh}\rightarrow[Q_+,\Omega^2\Sigma^2(Q,*)]_{\A^1,*}^{h\text{-}gr}$ is a group isomorphism where the target is endowed with the group structure induced by the $h$-group structure of $\Omega^2\Sigma^2(Q,*)$. For any $i\geqslant 2$, the map $\Omega^i\Sigma^i(Q,*)\rightarrow\Omega^{i+1}\Sigma^{i+1}(Q,*)$ is a map of $h$-groups and is a bijection because of our assumption on the $\A^1$-cohomological dimension of $Q$ as noted in the proof of \cite[Proposition 1.2.1]{AF}. Hence the above commutative diagram induces a map \[[Q_+,(Q,*)]_{\A^1,\bullet}^{coh}\rightarrow\colim[Q_+,\Omega^i\Sigma^i(Q,*)]_{\A^1,\bullet}^{h\text{-}gr}\] where the arrows of the diagram defining the colimit on the right are group isomorphisms and the map $[Q_+,(Q,*)]_{\A^1,\bullet}^{coh}\rightarrow[Q_+,\Omega^2\Sigma^2(Q,*)]_{\A^1,\bullet}^{h\text{-}gr}$ is a group isomorphism. Hence the map $[Q_+,(Q,*)]_{\A^1,\bullet}^{coh}\rightarrow\colim[Q_+,\Omega^i\Sigma^i(Q,*)]_{\A^1,\bullet}^{h\text{-}gr}$ is a group isomorphism. The map of sets \[\Sigma^\infty:[Q_+,(Q,*)]_{\A^1,\bullet}\rightarrow[\Sigma^\infty Q_+,\Sigma^\infty(Q,*)]_{\A^1,\bullet}\] induced by the functor $\Sigma^\infty$ factors through the map $[Q_+,(Q,*)]_{\A^1,\bullet}\rightarrow\colim[Q_+,\Omega^i\Sigma^i(Q,*)]_{\A^1,\bullet}$ into the obvious map $\colim[Q_+,\Omega^i\Sigma^i(Q,*)]_{\A^1,\bullet}^{h\text{-}gr}\rightarrow[\Sigma^\infty Q_+,\Sigma^\infty(Q,*)]_{\A^1,\bullet}$ which is a group isomorphism as is noted in the proof of \cite[Proposition 1.2.2]{AF}. All in all, we see that the map $\Sigma^\infty:[Q_+,(Q,*)]_{\A^1,\bullet}^{coh}\rightarrow[\Sigma^\infty Q_+,\Sigma^\infty(Q,*)]_{\A^1,\bullet}$ is a group isomorphism.

Now let $f:Q_+\rightarrow(Q,*)$ be the pointed map given by the identity on $Q\subseteq Q_+$ and which sends the disjoint base point of $Q_+$ to $*$. Then there is a commutative diagram in the category of sets:
\begin{center}
\begin{tikzcd}
 & \protect{[(Q,*),(Q,*)]_{\A^1,\bullet}^{hom}} \arrow[r] \arrow[dd,"f^*"] \arrow[ld] & \protect{[\Sigma^\infty(Q,*),\Sigma^\infty(Q,*)]_{\A^1,\bullet}} \arrow[dd,"(\Sigma^\infty f)^*"] \\
\protect{[Q,Q]_{\A^1}} & & \\
 & \protect{[Q_+,(Q,*)]_{\A^1,\bullet}^{coh}} \arrow[r] \arrow[lu]                    & \protect{[\Sigma^\infty Q_+,\Sigma^\infty(Q,*)]_{\A^1,\bullet}}
\end{tikzcd}
\end{center}
The horizontal maps are group isomorphisms owing to the previous arguments while $(\Sigma^\infty f)^*$ is a group homomorphism because $\mathcal{SH}^{\Sr^1}(k)$ is additive. The map $f^*$ is a bijection since both diagonal arrows are, hence $(\Sigma^\infty f)^*$ is a group \emph{iso}morphism. Therefore the same is true of $f^*$ which implies that the cohomotopical and homotopical group structures, which are induced by transport of structure using the bijectivity of the diagonal maps, agree as required.
\end{proof}

\subsection*{A remark on the Hurewicz homomorphism}

\paragraph*{Generators of $\widetilde{\CH}^n(Q_{2n})$.} We describe generators of this $\Kr_0^\MW(k)$-module that will be useful to us. We set $Y=V(x_1,\ldots,x_n,z)\hookrightarrow Q_{2n}$ and $Y'=V(x_1,\ldots,x_n,1-z)\hookrightarrow Q_{2n}$. Then $Y$ and $Y'$ are integral, say $y$ and $y'$ are their respective generic point, and smooth as $k$-schemes. For each of these schemes, the morphism to $\A^n$ defined by the $n$-tuple $(y_1,\ldots,y_n)$ is an isomorphism of $k$-schemes: thus $\dim Y=\dim Y'=n$ and since $Q_{2n}$ is integral of dimension $2n$, it follows that $Y$ and $Y'$ are of codimension $n$. We deduced that denoting by $\overline{a}$ the reduction mod the maximal ideal $\m_v$ in $\mathscr{O}_{Q_{2n},v}$ (for $v\in\{y,y'\}$), $(\overline{x_1},\ldots,\overline{x_n})$ is a $k(v)$-basis of $\m_v/\m_v^2$. Indeed, it is a generating set in each case, since $z=\frac{1}{1-z}\sum x_iy_i$ in $\mathscr{O}_{Q_{2n},y}$ and $1-z=\frac{1}{z}\sum x_iy_i$ in $\mathscr{O}_{Q_{2n},y'}$, and $\m_v/\m_v^2$ is of dimension $n$ over $k(v)$ as $Q_{2n}$ is smooth over $k$ hence regular and $v$ is of codimension $n$ as noted previously. Thus we obtain elements \[\alpha_n=\langle 1\rangle\otimes(\overline{x_1}\wedge\cdots\wedge\overline{x}_n)^*\in\Kr_0^\MW(k(y),\Lambda_y)\subseteq C_\RS(Q_{2n},n)^n,\] \[\alpha_n'=\langle 1\rangle\otimes(\overline{x_1}\wedge\cdots\wedge\overline{x_n})^*\in\Kr_0^\MW(k(y'),\Lambda_{y'})\subseteq C_\RS(Q_{2n},n)^n.\] Under the isomorphism of quadrics between $Q_{2n}$ and the quadric bearing the same name and studied in \cite{ADF} that we fixed in the remark after Definition \ref{def:quadrics}, $\alpha_n'$ is the cycle of \cite[Lemma 4.2.6]{ADF}. In particular: 

\begin{lem}\label{lem:gen_explicite_1-z}
The class of $\alpha_n'$ in $\widetilde{\CH}^n(Q_{2n})$ generates the free $\Kr_0^\MW(k)$-module $\widetilde{\CH}^n(Q_{2n})$.
\end{lem}

Consider now the closed subscheme $T$ of $Q_{2n}$ defined by $T=V(x_1,\ldots,x_{n-1})$. It is integral, smooth over $k$ and $k[T]\cong k[y_1,\ldots,y_{n-1}][x_n,y_n,z]/\langle x_ny_n=z(1-z)\rangle$ hence $T$ is of dimension $n+1$. Thus its generic point $\tau$ is a point of codimension $n-1$ of $Q_{2n}$ and $y$ and $y'$ are codimension $1$ points of $T$. In particular, we get residue homomorphisms $\Kr_1^\MW(k(\tau),\Lambda_\tau)\rightarrow\Kr_0^\MW(k(v),\Lambda_v)$ for $v\in\{y,y'\}$. Moreover, $(\overline{x_1},\ldots,\overline{x_{n-1}})$ is a basis of $\m_\tau/\m_\tau^2$ over $k(\tau)$. In particular, we have a well-defined element $[x_n]\otimes(\overline{x_1}\wedge\cdots\wedge\overline{x_{n-1}})^*\in\Kr_1^\MW(k(\tau),\Lambda_\tau)\subseteq C_\RS(Q_{2n},n)^{n-1}$.

\begin{lem}\label{lem:gen_explicite_z}
The residue of $\omega=[x_n]\otimes(\overline{x_1}\wedge\cdots\wedge\overline{x_{n-1}})^*$ in $C_\RS(Q_{2n},n)^{n}$ is $\alpha_n+\alpha_n'$. In particular, $\alpha_n$ is a cycle of the complex $C_\RS(Q_{2n},n)$ and $\alpha_n=-\alpha_n'$ in $\widetilde{\CH}^n(Q_{2n})$; thus $\alpha_n$ generates the free $\Kr_0^\MW(k)$-module $\widetilde{\CH}^n(Q_{2n})$.
\end{lem}

\begin{proof}
We begin by introducing some notation. Set \[l=dx_n\wedge dy_1\wedge\cdots\wedge dy_n\in\omega_T,\;l'=dx_1\wedge\cdots\wedge dx_n\wedge dy_1\wedge\cdots\wedge dy_n\in\omega_Q.\] (recall that Kähler differential modules are implicitly over the base $k$). By definition of the equations defining $T$ and $Q$, for any $v$ in $T$ (respectively in $Q$), the image $l(v)$ (respectively $l'(v)$) of $l$ (respectively $l'$) in $\omega_T(v)$ (respectively in $\omega_Q(v)$) generates it as a $k(v)$-vector space. Also, in dealing with tensor products of graded line bundles over a field, we do not make the base field of the tensor product explicit. 

Let $v$ be a point of codimension $1$ in $T$. The map $d_\tau^v:\Kr_1^\MW(k(\tau),\Lambda_\tau)\rightarrow\Kr_0^\MW(k(v),\Lambda_v)$ is obtained as the following composition:
\begin{center}
\begin{tikzcd}
\Kr_1^\MW(k(\tau),\Lambda_\tau) \arrow[r,"\cong"] & \Kr_1^\MW(k(\tau),\omega_{k(\tau)}\otimes\omega_Q(\tau)^{-1}) \arrow[d,"\partial_\tau^v"] & \\
                                                  & \Kr_0^\MW(k(v),D(\p_v/\p_v^2)^{-1}\otimes\omega_T(v)\otimes\omega_Q(v)^{-1}) \arrow[r,"\cong"]   & \Kr_0^\MW(k(v),\Lambda_v)
\end{tikzcd}
\end{center}
where, as described in \cite[Subsection 2.1]{FaLCWG}, the isomorphisms are induced by the canonical isomorphism of \cite[Equation (1.4)]{FaLCWG} and $\partial_\tau^v$ is the twisted residue homomorphism of \cite[Equation (1.5)]{FaLCWG}: indeed, since $T$ is smooth over $k$ hence normal, following \cite[Subsection 2.1]{FaLCWG}, the transfer of \cite[Definition 1.24]{FaLCWG} plays no role in the definition of $d_\tau^v$. Here, $\p_v$ is the maximal ideal of the ring $\mathscr{O}_{T,v}$: since $T$ is normal and $v$ is of codimension $1$, this ring is a discrete valuation ring and there is in particular a uniformiser $\pi\in\p_v$. Now given $\alpha\in k(\tau)$, $\partial_\tau^v$ sends $\alpha\otimes (l(\tau)\otimes l'(\tau)^*)$ to $\partial_v^\pi(\alpha)\otimes(\overline{\pi}^*\otimes l(v)\otimes l'(v)^*)$ where $\partial_v^\pi$ is the untwisted residue homomorphism of \cite[Theorem 1.7]{FaLCWG} relative to $\pi$.

We take $\alpha=[x_n]$. We then see that $\partial_v^\pi(\alpha)=0$ if $x_n\in\mathscr{O}_{T,v}^\times$ (\cite[Theorem 1.7, 2.]{FaLCWG}). Thus if $d_\tau^v\omega$ is non-zero, then $x_n$ lies in the prime ideal $\rf$ of $k[Q_{2n}]$ corresponding to $v$. Since $v$ lies in $\overline{\{\tau\}}$, $\rf$ also contains $\langle x_1,\ldots,x_{n-1}\rangle$: thus $z(1-z)=\sum x_iy_i$ lies in $\rf$, so that $z\in\rf$ or $1-z\in\rf$ as $\rf$ is prime. If $z$ lies in $\rf$, we have an inclusion $\langle x_1,\ldots,x_n,z\rangle\subseteq\rf$ of ideals of height $n$ in $k[Q_{2n}]$: consequently, this inclusion is an equality and we have $v=y$. Similarly, $v=y'$ if $1-z$ lies in $\rf$. Thus it remains to show that $d_\tau^y\omega=\alpha_n$ and $d_\tau^{y'}\omega=\alpha_n'$.

Consider the case $v=y$. Then $(x_n,z)$ generates the maximal ideal of $\mathscr{O}_{T,y}$ and we even have $z=\frac{1}{1-z}x_ny_n$ in $\mathscr{O}_{T,y}$. Thus $x_n$ is a uniformiser of the discrete valuation ring $\mathscr{O}_{T,y}$. The isomorphisms in the above composition, which come from \cite[Equation (1.4)]{FaLCWG} since we consider smooth schemes over a field, now read as follows using this isomorphism.
\begin{itemize}
	\item We have $\Lambda_\tau\cong\omega_{k(\tau)}\otimes\omega_Q(\tau)^{-1}$. This isomorphism sends $(\overline{x_1}\wedge\cdots\wedge\overline{x_{n-1}})^*$ to $(-1)^{n-1}l(\tau)\otimes l'(\tau)^*$. Here the sign is due to the fact that $D(\m_\tau/m_\tau^2)$ is graded by $\codim_Q\tau=n-1$ and that we used the canonical isomorphism $D(\m_\tau/\m_\tau^2)^{-1}\otimes_\tau D(\m_\tau/\m_\tau^2)\cong(k(\tau),0)$ of graded line bundles over $k(\tau)$, see \cite[Subsection 1.3]{FaLCWG}.
	\item We have $D(\p_y/\p_y^2)^{-1}\otimes\omega_T(y)\otimes\omega_Q(y)^{-1}\cong\omega_{k(y)}\otimes\omega_Q(y)^{-1}$: this isomorphism sends $\overline{x_n}^*\otimes l(y)\otimes l'(y)^*$ to $-dx_1\wedge\cdots\wedge dx_n\otimes l(y)^*$: the sign that appears is due to the fact that $D(\p_y/\p_y^2)$ is graded by $1=\codim_T y$.
	\item We have $\omega_{k(y)}\otimes\omega_Q(y)^{-1}\cong\Lambda_y$: this isomorphism sends $dx_1\wedge\cdots\wedge dx_n\otimes l(y)^*$ to $(-1)^n(\overline{x_1}\wedge\cdots\wedge\overline{x_n})^*$ because $D(\m_y/\m_y^2)$ is graded by $n=\codim_Q y$.
\end{itemize}
On the other hand, according to \cite[Theorem 1.7, 1.]{FaLCWG}, $\partial_v^{x_n}[x_n]=\langle 1\rangle$. It is now clear that $d_\tau^y\omega=(-1)^{n-1}\cdot(-1)\cdot(-1)^n\alpha_n=\alpha_n$. An entirely similar argument shows that $d_\tau^{y'}=\alpha_n'$ which concludes the proof.
\end{proof}

\paragraph*{The $\Kr_0^\MW(k)$-linearity of the Hurewicz homomorphism.} Let $n\geqslant 2$. Let $U$ be a smooth affine $k$-scheme. In \cite[Theorem 1.3.4]{AF}, the authors introduce a \emph{Hurewicz homomorphism} $h:[U,Q_{2n}]_{\A^1}\rightarrow\widetilde{\CH}^n(U)$. Concretely, there exists a generator $\alpha_{Q_{2n}}$ of $\widetilde{\CH}^n(Q_{2n})$ as a $\Kr_0^\MW(k)$-module corresponding to the class of the morphism $Q_{2n}\rightarrow\mathrm{K}(\mathbf{K}_n^\MW,n)$ of the Postnikov tower of $Q_{2n}$ viewed as an element of $[Q_{2n},\mathrm{K}(\mathbf{K}_n^\MW,n)]_{\A^1,\bullet}=\Hr^n(Q_{2n},\mathbf{K}_n^\MW)=\widetilde{\CH}^n(Q_{2n})$; in particular, $\alpha_{Q_{2n}}$ is equal to $\alpha_n$ up to a unit of $\Kr_0^\MW(k)$. Since $U$ is affine, Theorem \ref{naivequadrics} ensures any morphism $U\rightarrow Q_{2n}$ in $\mathcal{H}(k)$ is induced by a scheme morphism $f:U\rightarrow Q_{2n}$ which induces a pullback $f^*:\widetilde{\CH}^n(Q_{2n})\rightarrow\widetilde{\CH}^n(U)$ at the level of Chow--Witt groups (\cite[Theorem 3.10]{FaLCWG}): by definition, we then have $h(f)=f^*\alpha_{Q_{2n}}$. This defines is a group homomorphism (where $[U,Q_{2n}]_{\A^1}$ has the cohomotopy group structure defined above).

Now both the source and the target have the structure of a $\Kr_0^\MW(k)$-module. Indeed, this is true of Chow–Witt groups because the Rost--Schmid complex is a complex of $\Kr_0^\MW(k)$-modules. On the other hand, the functoriality of cohomotopy theories implies that $[U,Q_{2n}]_{\A^1}$ is a module over the ring $[Q_{2n},Q_{2n}]_{\A^1}$ which is isomorphic to $\Kr_0^\MW(k)$ by \cite[Corollary 6.43]{Morel} (modulo the ring isomorphism $[Q_{2n},Q_{2n}]_{\A^1}\cong[(Q_{2n},*),(Q_{2n},*)]_{\A^1,\bullet}$—note that the composition laws induced on the unpointed set do not depend on the choice of base-point).

Now the natural question about $h$ has a positive answer:

\begin{prop}\label{hurhomlin}
The Hurewicz homomorphism $h:[U,Q_{2n}]_{\A^1}\rightarrow\widetilde{\CH}^n(U)$ is $\Kr_0^\MW(k)$-linear.
\end{prop}

\begin{proof}
Denote by $f_a$ the endomorphism $h_a^n$ from Lemma \ref{lem:explicit_rank_one_representative_2n} for $a\in k^\times$. The morphisms $f_a\in[Q_{2n},Q_{2n}]_{\A^1}$ for $a\in k^\times$ give a set of additive generators of $[Q_{2n},Q_{2n}]_{\A^1}$ by \cite[Lemma 3.6, 2)]{Morel} and by definition, for any element $f$ of $[U,Q_{2n}]_{\A^1}$, $\langle a\rangle\cdot f=f_a\circ f$. Hence it suffices to show that $h(f_a\circ f)=\langle a\rangle h(f)$ for any $a\in k^\times$.

Let then $a\in k^\times$. The morphism $f_a$ is smooth (it is an isomorphism), hence it induces a morphism $f_a^*:C_{\RS}(Q_{2n},n)\rightarrow C_{\RS}(Q_{2n},n)$ of chain complexes that, according to \cite[Example 2.11]{FaLCWG}, satisfies \[f_a^*\alpha_n=\langle 1\rangle\otimes(\overline{ax_1}\wedge\overline{x_2}\cdots\wedge\overline{x_n})^*=\langle a\rangle\otimes(\overline{x_1}\wedge\cdots\wedge\overline{x_n})^*=\langle a\rangle\cdot\alpha_n\] at the level of cycles, hence of Chow--Witt groups. Since $\alpha_n$ generates $\widetilde{\CH}^n(Q_{2n})$ as a $\Kr_0^\MW(k)$-module, the $\Kr_0^\MW(k)$-linearity of Chow--Witt groups implies that $f_a^*:\widetilde{\CH}^n(Q_{2n})\rightarrow\widetilde{\CH}^n(Q_{2n})$ is the multiplication by $\langle a\rangle$. Finally we note that \[h(f_a\circ f)=(f_a\circ f)^*\alpha_{Q_{2n}}=f^*\circ f_a^*\alpha_{Q_{2n}}=f^*(\langle a\rangle\cdot\alpha_{Q_{2n}})=\langle a\rangle\cdot(f^*\alpha_{Q_{2n}})=\langle a\rangle\cdot h(f)\] by $\Kr_0^\MW(k)$-linearity of $f^*$ as required.
\end{proof}

\section[First cohomotopical interpretation]{A cohomotopical interpretation of van der Kallen's law}\label{proof}

We now come to the interplay between the two group structures introduced in the previous sections. Recall that the letter $k$ denotes a field. We let $\mathscr{X}$ be $\A^n\setminus 0$ pointed by $(0,\ldots,0,1)$ in Proposition \ref{bijgrpebor}. The goal of this section is to prove the following.

\begin{theo}\label{groupiso}
Let $X$ be a smooth affine $k$-scheme of Krull dimension $d$ and let $n$ be an integer such that $d\leqslant 2n-4$. Then the natural map $\Psi:\Um_n(k[X])/\Er_n(k[X])\rightarrow[X,\A^n\setminus 0]_{\A^1}$ is a group isomorphism, where the left-hand side is endowed with the structure of van der Kallen and the right-hand side is endowed with the group structure defined by means of Proposition \ref{bijgrpebor}.
\end{theo}

The theorem contains two statements: the first is that $\Psi$ is a group homomorphism; the second is that $\Psi$ is a bijection. We study these two statements in two separate subsections.

\subsection{The map $\Psi$ is a bijection}

We proceed in two steps. We first identify the set $\Um_n(k[X])/\Er_n(k[X])$ (for $X$ a smooth and affine $k$-scheme) to the set of morphisms $X\rightarrow\A^n\setminus 0$ up to \emph{naive} $\A^1$-homotopy, and then show that naive and ``true'' $\A^1$-homotopies coincide in the case of morphisms from affine $k$-schemes to $\A^n\setminus 0$. We remind the reader that $k$ is a field that is not required to be perfect and that $n\geqslant 1$ is an integer.

The following theorem is due to Fasel, see \cite[Theorem 2.1]{JF}.

\begin{theo}\label{naiv}
Let $X$ be a smooth affine $k$-scheme. Then the natural map from $\Um_n(k[X])$ to $\Hom_{\mathsf{Sm}_k}(X,\A^n\setminus 0)$ induces a well-defined map $\Um_n(k[X])/\Er_n(k[X])\rightarrow\Hom_{\A^1}(X,\A^n\setminus 0)$ and this map is a bijection.
\end{theo}

The following theorem is noted in \cite{AHW2} as a consequence of the odd $n$ case of Theorem \ref{naivequadrics}.

\begin{theo}[\protect{\cite[Corollary 4.2.6]{AHW2}}]\label{naivnotnaiv}
Let $X$ be a smooth affine $k$-scheme. Then the natural map $\Hom_{\A^1}(X,\A^n\setminus 0)\rightarrow[X,\A^n\setminus 0]_{\A^1}$ is a bijection.
\end{theo}

\begin{proof}
Indeed, $k$ is ind-smooth over its prime subfield which is perfect so that we may apply \cite[Corollary 4.2.6]{AHW2}
\end{proof}

\begin{rema}
This theorem was already remarked in \cite[Remark 8.10]{Morel} for $k$ perfect and $n\neq 2$.
\end{rema}

Combining the two statements, we get:

\begin{theo}\label{bijection}
Let $n\geqslant 1$ be an integer and let $X$ be a smooth affine $k$-scheme. Then the natural map $\Psi:\Um_n(k[X])/\Er_n(k[X])\rightarrow[X,\A^n\setminus 0]_{\A^1}$ is a bijection.
\end{theo}

\subsection{The map $\Psi$ is a group homomorphism}\label{grouphom}

%
Let $n$ be an integer such that $n\geqslant 3$. Let $U_n$ be the open subset $\A^1\times\A^{n-1}\setminus 0$ of $\A^n\setminus 0$, which contains the base point of $\A^n\setminus 0$, and set \[V_{2n}=U_n\times(\A^n\setminus 0)\cup(\A^n\setminus 0)\times U_n\subseteq(\A^n\setminus 0)\times(\A^n\setminus 0).\] It is an open subscheme of $(\A^n\setminus 0)\times(\A^n\setminus 0)$ which contains its base point $(0,\ldots,0,1,0,\ldots,0,1)$; we turn $V_{2n}$ into a pointed space with this point. The map $V_{2n}\rightarrow(\A^n\setminus 0)\times(\A^n\setminus 0)$ is then the inclusion of an open subscheme and a map of pointed spaces. We start by studying the $\A^1$-connectedness of this inclusion. We first prove that $V_{2n}$ is $\A^1$-connected by showing that $V_{2n}$ is $\A^1$-chain connected.

\begin{lem}
Let $m$ be an integer such that $m\geqslant 2$. The $k$-scheme $\A^m\setminus 0$ is $\A^1$-chain connected. Hence $V_{2n}$ is $\A^1$-chain connected and is thus $\A^1$-connected.
\end{lem}

\begin{proof}
We first show that $\A^m\setminus 0$ is $\A^1$-chain connected. Let $L$ be a finitely generated separable extension of $k$ and let $x$ belong to $(\A^m\setminus 0)(L)$: we may think of $x$ as an $m$-tuple $(x_1,\ldots,x_m)$ of elements of $L$ which are not all equal to $0$. If one of the $x_i$'s for $i<m$ is non-zero, then we may find an $m$-tuple $(a_1,\ldots,a_m)$ of elements of $L$ such that $\sum_{i=1}^m a_ix_i=1$ and $a_m=1$. Then $((1-t)x_1,\ldots,(1-t)x_{m-1},(1-t)x_m+t)$ belongs to $\Um_m(L[t])$ since \[\sum_{i=1}^{m-1}a_i(1-t)x_i+a_m((1-t)x_m+t)=(1-t)\sum_{i=1}^m a_ix_i+a_mt=(1-t)+t=1.\] Hence we obtain a morphism $f:\A_L^1\rightarrow\A_L^m\setminus 0$, induced by $((1-t)x_1,\ldots,(1-t)x_m+t)$, such that $f(0)=x$ and $f(1)=(0,\ldots,0,1)$: by definition of the relation $\sim$, this implies that $x\sim(0,\ldots,0,1)$. If $x_i=0$ for all $i<m$, then $x_m\in L^\times$; $(t,0,\ldots,0,x_m(1-t))$ belongs to $\Um_m(L[t])$, since \[1\cdot t+x_m^{-1}\cdot x_m(1-t)=1.\] Thus $(0,\ldots,0,x_m)\sim(1,0,\ldots,0)$ and since $(1,0,\ldots,0)\sim(0,\ldots,0,1)$ by the previous argument, by transitivity, the relation $(0,\ldots,0,x_m)\sim(0,\ldots,0,1)$ holds. Thus $x\sim(0,\ldots,0,1)$ in any case. Therefore, $(\A^m\setminus 0)(L)/{\sim}$ consists of a point and $\A^m\setminus 0$ is $\A^1$-chain connected by definition.

Now $U_n=\A^1\times\A^{n-1}\setminus 0$ is $\A^1$-chain connected since it is the product of $\A^1$-chain connected spaces as $n\geqslant 4$, according to Example \ref{examplesaonechaincon}. Still owing to this example, $U_n\times(\A^n\setminus 0)$ and $(\A^n\setminus 0)\times U_n$ are $\A^1$-chain connected, hence $V_{2n}=(U_n\times(\A^n\setminus 0))\cup((\A^n\setminus 0)\times U_n)$ is $\A^1$-chain connected by Example \ref{examplesaonechaincon} again since the base point of $V_{2n}$ belongs to both $U_n\times(\A^n\setminus 0)$ and $(\A^n\setminus 0)\times U_n$. The fact that it is $\A^1$-connected then results from Proposition \ref{aonechaincoimpaoneco}.
\end{proof}

\begin{lem}\label{conninc}
The inclusion map $u:V_{2n}\rightarrow(\A^n\setminus 0)\times(\A^n\setminus 0)$ is $\A^1$-$(2n-4)$-connected.
\end{lem}

\begin{proof}
Clearly, the source and target of $u$, as well as $u$ itself, are pulled back from the prime subfield $k_0$ of $k$ which is perfect. Thus by Lemma \ref{lem:a1_conn_base_change_map}, we can reason over $k_0$ and thus assume that $k$ is perfect.

We will apply the Blakers--Massey theorem to $u$, hence we first show that $V_{2n}$ is $\A^1$-simply connected. The space $V_{2n}$ is $\A^1$-connected by the previous lemma. Thus to show that $V_{2n}$ is $\A^1$-simply connected, it suffices to prove that $\pi_1^{\A^1}(V_{2n})$ is trivial. Set \[V_{2n}^1=U_n\times(\A^n\setminus 0)\subseteq(\A^n\setminus 0)\times(\A^n\setminus 0),\;V_{2n}^2=(\A^n\setminus 0)\times U_n\subseteq(\A^n\setminus 0)\times(\A^n\setminus 0),\] so that $V_{2n}=V_{2n}^1\cup V_{2n}^2$ and $V_{2n}^1\cap V_{2n}^2=U_n\times U_n$. Since $n\geqslant 3$, the sheaf $\pi_1^{\A^1}(\A^n\setminus 0)$ is trivial. Using the commutation of the homotopy sheaf $\pi_1^{\A^1}$ with products, we see that the map $\pi_1^{\A^1}(V_{2n}^1\cap V_{2n}^2)\rightarrow\pi_1^{\A^1}(V_{2n}^1)$ is given by $(\Id,*):\pi_1^{\A^1}(U_{2n}\times U_{2n})\rightarrow\pi_1^{\A^1}(U_{2n}\times\A^n\setminus 0)$ while the map $\pi_1^{\A^1}(V_{2n}^1\cap V_{2n}^2)\rightarrow\pi_1(V_{2n}^2)$ is given by $(*,\Id):\pi_1^{\A^1}(U_{2n}\times U_{2n})\rightarrow\pi_1(\A^n\setminus 0\times U_{2n})$. It is then clear using the description of $\pi_1^{\A^1}(V_{2n})$ as an amalgamated sum deduced from the motivic van Kampen theorem \cite[Theorem 7.12]{Morel} that $\pi_1^{\A^1}(V_{2n})$ is trivial as required.

We also need to estimate the $\A^1$-connectedness of the cofibre of $u$. The (reduced) complement of $U_n$ is $\mathbb{G}_m\times 0\subseteq\A^n\setminus 0$: this implies that the (reduced) complement of $V_{2n}$ is $(\mathbb{G}_m\times 0)\times(\mathbb{G}_m\times 0)$, and is in particular smooth with trivial normal bundle. By the purity theorem \cite[§3, Theorem 2.23]{MoVo}, we then have a cofibre sequence: \[V_{2n}\rightarrow(\A^n\setminus 0)\times(\A^n\setminus 0)\rightarrow(\mathbb{G}_m\times\mathbb{G}_m)_+\wedge(\mathbb{P}^1)^{\wedge{2n-2}}.\] Therefore, since $(\mathbb{G}_m\times\mathbb{G}_m)_+\wedge(\mathbb{P}^1)^{\wedge{2n-2}}\simeq\Sigma^{2n-2}(\mathbb{G}_m\times\mathbb{G}_m)_+\wedge\mathbb{G}_m^{2n-2}$ is simplicially $(2n-3)$-connected hence $\A^1$-$(2n-3)$-connected by Theorem \ref{simpliciallyconn}, the homotopy cofibre $C$ is $\A^1$-$(2n-3)$-connected.

In particular, $u$ is $\A^1$-simply connected as both its source and target are $\A^1$-simply connected. Hence its homotopy fibre $F$ is $\A^1$-connected and we conclude by the first item of \cite[Theorem 4.1]{AFCompEulerClass} that the comparison map $f$ of this theorem is $\A^1$-$2$-connected. If $n=3$, then the target of this map is $\A^1$-$2$-connected by the above estimate of the $\A^1$-connectedness of $C$: this implies that $F$ is $\A^1$-$2$-connected and thus $u$ is $\A^1$-$2$-connected which is the claim of Lemma \ref{conninc} if $n=3$.

Now assume that $n\geqslant 4$. Observe that the target of $u$ is $\A^1$-$(m-1)$-connected with $m=n-1$ by commutation of $\A^1$-homotopy sheaves with products. Since $n\geqslant 4$, the inequality $2n-4\geqslant m+1=n$ holds. Thus by the argument of Example \ref{exe:application_blakers_massey}, $u$ is $\A^1$-$\delta$-connected where $\delta=\min(d,2m+1)=\min(2n-4,2n-1)=2n-4$ as required.
\end{proof}

\begin{lem}\label{bijincl}
Let $X$ be a smooth $k$-scheme of $\A^1$-cohomological dimension $d\leqslant 2n-4$. Then the natural map $[X,V_{2n}]_{\A^1}\rightarrow[X,(\A^n\setminus 0)\times(\A^n\setminus 0)]_{\A^1}$ is a bijection.
\end{lem}

\begin{proof}
This result follows from applying Lemma \ref{lemobsth} to the inclusion map of Lemma \ref{conninc}.
\end{proof}

The scheme $V_{2n}$ allows us to produce an explicit, scheme-theoretic model for the fold map $(\A^n\setminus 0)\vee(\A^n\setminus 0)\rightarrow(\A^n\setminus 0)$. First, we construct a Jouanolou device for $V_{2n}$. Consider the projection \[p:\A^n\times\A^n=(\A^1\times\A^{n-1})\times(\A^1\times\A^{n-1})\rightarrow\A^{n-1}\times\A^{n-1}=\A^{2n-2}.\] We also let $p$ denote its restriction to the open subset $(\A^n\setminus 0)\times(\A^n\setminus 0)$. Observe that $V_{2n}=p^{-1}(\A^{2n-2}\setminus 0)$\footnote{The author wishes to thank Sophie Morel who allowed him to correct a severe misunderstanding in the interpretation of this equality of subschemes of $(\A^n\setminus 0)\times(\A^n\setminus 0)$ in terms of unimodular rows.}. Recall that the projection $Q_{2n-1}\rightarrow\A^n\setminus 0$ on the first $n$ coordinates is a Jouanolou device. Consider now the (outer) pullback square

\begin{center}
\begin{tikzcd}
\widetilde{V_{2n}} \arrow[rr] \arrow[dd] &                                & Q_{2n-1}\times Q_{2n-1} \arrow[d] \\
                                         & V_{2n} \arrow[d] \arrow[r]     & (\A^n\setminus 0)\times(\A^n\setminus 0) \arrow[d,"p"] \\
Q_{4n-5} \arrow[r]                       & \A^{2n-2}\setminus 0 \arrow[r] & \A^{2n-2}
\end{tikzcd}
\end{center}

The $k$-scheme $\widetilde{V_{2n}}$ is affine as it is the fibre product of affine schemes over an affine base. Furthermore, the map $\widetilde{V_{2n}}\rightarrow\A^{2n-2}$ factors through $\A^{2n-2}\setminus 0$, thus the map $\widetilde{V_{2n}}\rightarrow Q_{2n-1}\times Q_{2n-1}\rightarrow(\A^n\setminus 0)\times(\A^n\setminus 0)$ factors through $p^{-1}(\A^{2n}\setminus 0)=V_{2n}$. Now we have a pullback diagram 
\begin{center}
\begin{tikzcd}
\widetilde{V_{2n}} \arrow[r] \arrow[d]                 & V_{2n}\times_{(\A^n\setminus 0)\times(\A^n\setminus 0)}(Q_{2n-1}\times Q_{2n-1}) \arrow[d] \\
Q_{4n-5}\times_{\A^{2n-2}\setminus 0} V_{2n} \arrow[r] & V_{2n}
\end{tikzcd}
\end{center}
The bottom horizontal and right vertical maps of this diagram are torsors under vector bundles since they are base changes of such morphisms. Thus all morphisms in the diagram are torsors under vector bundles again by base change and it follows that the morphism $\widetilde{V_{2n}}\rightarrow V_{2n}$ is a torsor under a vector bundle by composition. Hence since $\widetilde{V_{2n}}$ is affine, the map $\widetilde{V_{2n}}\rightarrow V_{2n}$ is a Jouanolou device.

The $k$-algebra $k[\widetilde{V_{2n}}]$ is given by the following presentation. Let $x'=(x_2,\ldots,x_n)$, $y'=(y_2,\ldots,y_n)$, $u'=(u_2,\ldots,u_n)$, $v'=(v_2,\ldots,v_n)$, $r'=(r_2,\ldots,r_n)$ and $s'=(s_2,\ldots,s_n)$ be tuples of variables. Then \[k\left[\widetilde{V_{2n}}\right]=k[x_1,x',y_1,y',u_1,u',v_1,v',r',s']/I\] where $I$ is described by the following relations: \[
\begin{array}{rcl}
x_1y_1+x'y'^{T}   &=& 1\\
u_1v_1+u'v'^{T}   &=& 1\\
x'r'^{T}+u's'^{T} &=& 1
\end{array}
\] where $w^T$ is the transpose column of the row $w$. Thus $(x_1,\ldots,x_n,u_1,\ldots,u_n)$ is in some sense the universal (among rows of elements of $k$-algebras) row of length $2n$ such that $(x_1,\ldots,x_n)$, $(u_1,\ldots,u_n)$ and $(x_2,\ldots,x_n,u_2,\ldots,u_n)$ are unimodular rows of length $n$, $n$ and $2n-2$ respectively, with given splitting in the sense of Definition \ref{splitting}.

We are therefore in the situation of the proof of the Mennicke--Newman lemma after the first step for the rows $(x_1,\ldots,x_n)$ and $(u_1,\ldots,u_n)$, so that we may follow its proof to produce an explicit (elementary) $\A^1$-weak equivalence on $\widetilde{V_{2n}}$ that turns $(x_1,\ldots,x_n)$ (respectively $(u_1,\ldots,u_n)$) into $(x,a_1,\ldots,a_n)$ (respectively $(1-x,a_2,\ldots,a_n)$); hence the row $(x(1-x),a_2,\ldots,a_n)$ is unimodular and thus defines a morphism $\widetilde{q_{n}}:\widetilde{V_{2n}}\rightarrow\A^n\setminus 0$. This yields a map $q_n:V_{2n}\xrightarrow{\simeq}\widetilde{V_{2n}}\xrightarrow{q_n}\A^n\setminus 0$ \emph{in} $\mathcal{H}(k)$.

\begin{center}
\begin{tikzcd}
\widetilde{V_{2n}} \arrow[d] \arrow[rd,"\widetilde{q_n}"] & \\
V_{2n} \arrow[r,swap,dotted,"q_n"]                                    & \A^n\setminus 0
\end{tikzcd}
\end{center}

Observe also that the inclusion $(\A^n\setminus 0)\vee(\A^n\setminus 0)\rightarrow(\A^n\setminus 0)\times(\A^n\setminus 0)$ factors through $V_{2n}$. Indeed, on the first factor, it corresponds to the morphism $(x_1,\ldots,x_n)\mapsto(x_1,\ldots,x_n,0,\ldots,0,1)$ which factors through $(\A^n\setminus 0)\times U_n$ thus through $V_{2n}$, while on the second factor, it corresponds to the morphism $(x_1,\ldots,x_n)\mapsto(0,\ldots,0,1,x_1,\ldots,x_n)$ which factors through $U_n\times(\A^n\setminus 0)$, thus through $V_{2n}$ as well. Therefore, we have a factorisation \[(\A^n\setminus 0)\vee(\A^n\setminus 0)\rightarrow V_{2n}\rightarrow(\A^n\setminus 0)\times(\A^n\setminus 0)\] of the inclusion $(\A^n\setminus 0)\vee(\A^n\setminus 0)\rightarrow(\A^n\setminus 0)\times(\A^n\setminus 0)$ and the map $q_n:V_{2n}\rightarrow\A^n\setminus 0$ thus defines a map $(\A^n\setminus 0)\vee(\A^n\setminus 0)\rightarrow\A^n\setminus 0$ by composition. The key lemma is the following.

\begin{lem}\label{modelfold}
The map $(\A^n\setminus 0)\vee(\A^n\setminus 0)\rightarrow\A^n\setminus 0$ is the identity on each factor up to homotopy.
\end{lem}

In other words, the diagram
\begin{center}
\begin{tikzcd}
(\A^n\setminus 0)\vee(\A^n\setminus 0) \arrow[rd,"\nabla"] \arrow[r] & V_{2n} \arrow[d,"q_n"] \\
                                                                     & \A^n\setminus 0
\end{tikzcd}
\end{center}
commutes in $\mathcal{H}(k)$.

\begin{proof}
By definition of the map $q_n:V_{2n}\rightarrow\A^n\setminus 0$, we have to show that if we apply the procedure of the proof of the Mennicke--Newman lemma to the unimodular rows $(x_1,\ldots,x_n)\in\Um_n(k[\widetilde{V_{2n}}])$ and $(0,\ldots,0,1)$, we obtain a unimodular row in the orbit of $(x_1,\ldots,x_n)$ under the action of $\Er_n(k[\widetilde{V_{2n}}])$ (since this procedure is clearly symmetric in the two considered rows). Following the steps of the proof of Proposition \ref{mennnewlem}, with $\alpha_i=0$, $\beta_i=0$ for $i<n$ and $\beta_n=x_1-1$ for the second step (there is nothing to do in the first step), we obtain the unimodular row \[v=(x_1(1-x_1),\ldots,x_{n-1}(1-x_1),x_n+(1-x_n)x_1).\] The row $(x_n+(1-x_n)x_1,1-x_1)$ of length $2$ is unimodular, since \[(1-x_1)(1-x_n)+x_n+(1-x_n)x_1=1.\] It then follows from \cite[Lemma 3.5 (iv)]{vandK} (with $a=x_n+(1-x_n)x_1$ and $r=1-x_1$), whose proof relies entirely on elementary operations, that $v$ is in the orbit of the row \[(x_1,x_2,x_3(1-x_1),\ldots,x_n+(1-x_n)x_1).\] By elementary operations again, we may now replace $x_i-x_1x_i$ by $x_i$ when $3\leqslant i\leqslant n-1$ and $x_n+(1-x_n)x_1$ by $x_n$, thus proving the lemma.
\end{proof}

\begin{rema}
The above result essentially sums up to the observation that the orbit $[0,\ldots,0,1]$ is the neutral element for van der Kallen's law.
\end{rema}

We have thus constructed a scheme-theoretic model for the fold map $\nabla:(\A^n\setminus 0)\vee(\A^n\setminus 0)\rightarrow\A^n\setminus 0$. We can now prove the main result of this subsection.

\begin{prop}
Let $X$ be a smooth affine $k$-scheme of Krull dimension $d$ and let $n\geqslant 4$ be an integer such that $d\leqslant 2n-4$. Then the natural map $\Psi:\Um_n(k[X])/\Er_n(k[X])\rightarrow[X,\A^n\setminus 0]_{\A^1}$ is a group homomorphism.
\end{prop}

\begin{proof}
Recall that the group structure on the source (respectively target) of $\Psi$ is given by the Mennicke--Newman lemma (respectively by Proposition \ref{bijgrpebor}). Consider two unimodular rows $u$ and $v$ of length $n$ with values in $k[X]$; we also denote by $u$ and $v$ the morphisms of $k$-schemes from $X$ to $\A^n\setminus 0$ that they induce. They define a morphism of $k$-schemes \[(u\times v)\circ\Delta_X:X\xrightarrow{\Delta_X} X\times X\xrightarrow{u\times v}(\A^n\setminus 0)\times(\A^n\setminus 0).\] Since $X$ is of dimension less than or equal to $2n-4$, step (1) of the proof of the Mennicke--Newman lemma shows that, up to naive homotopy defined by elementary matrices, the morphism $(u\times v)\circ\Delta_X$ factors through $V_{2n}$ into a scheme-theoretic morphism $\varphi:X\rightarrow V_{2n}$. By Corollary \ref{factbyjouan}, $\varphi$ also factors through $\widetilde{V_{2n}}$ into a scheme-theoretic morphism $\widetilde{\varphi}:X\rightarrow\widetilde{V_{2n}}$. By definition of the law on the source of $\Psi$ and by definition of $\widetilde{q_n}$, the morphism $\widetilde{q_n}\circ\widetilde{\varphi}$ represents $\Psi(u+v)$; on the other hand, by definition of $q_n$, $\widetilde{q_n}\circ\widetilde{\varphi}$ is equal to $q_n\circ\varphi$ in $\mathcal{H}(k)$. The following diagram in $\mathcal{H}(k)$, whose solid arrows indicate morphisms coming from $k$-scheme morphisms, thus commutes.

\begin{center}
\begin{tikzcd}
X \arrow[r] \arrow[rd,"\protect{\varphi}"] \arrow[rdd,swap,"\protect{\widetilde{\varphi}}"]  & (\A^n\setminus 0)\times(\A^n\setminus 0)                        & \A^n\setminus 0 \\
                                                                                             & V_{2n} \arrow[u,hook] \arrow[ru,dotted,"q_n"]                   &                 \\
                                                                                             & \widetilde{V_{2n}} \arrow[u] \arrow[ruu,swap,"\widetilde{q_n}"] &
\end{tikzcd}
\end{center}

On the other hand, we have already observed the following sequence:
\begin{center}
\begin{tikzcd}
(\A^n\setminus 0)\vee(\A^n\setminus 0) \arrow[r] \arrow[rd,swap,"\nabla"] & V_{2n} \arrow[r] \arrow[d,"q_n"] & (\A^n\setminus 0)\times(\A^n\setminus 0)\\
                                                                          & \A^n\setminus 0            &
\end{tikzcd}
\end{center}
It induces bijections \[[X,(\A^n\setminus 0)\vee(\A^n\setminus 0)]_{\A^1}\rightarrow[X,V_{2n}]_{\A^1}\rightarrow[X,(\A^n\setminus 0)\times(\A^n\setminus 0)]_{\A^1}.\] Indeed:
\begin{itemize}
	\item the map $[X,V_{2n}]_{\A^1}\rightarrow[X,(\A^n\setminus 0)\times(\A^n\setminus 0)]_{\A^1}$ is a bijection according to Lemma \ref{bijincl} as $X$ is of Krull dimension $d\leqslant 2n-4$ and the $\A^1$-cohomological dimension is bounded above by the Krull dimension;
	\item the map $[X,(\A^n\setminus 0)\vee(\A^n\setminus 0)]_{\A^1}\rightarrow[X,(\A^n\setminus 0)\times(\A^n\setminus 0)]_{\A^1}$ is a bijection by Proposition \ref{bijgrpebor}, again because $\A^1$-cohomological dimension is bounded above by Krull dimension.
\end{itemize}
According to the second item above, there exists a unique morphism $\overline{(u,v)}:X\rightarrow(\A^n\setminus 0)\vee(\A^n\setminus 0)$ in $\mathcal{H}(k)$ corresponding to $(u\times v)\circ\Delta_X:X\rightarrow(\A^n\times 0)\times(\A^n\setminus 0)$. By definition of the group law on $[X,\A^n\setminus 0]_{\A^1}$, we thus have $\Psi(u)+\Psi(v)=\nabla\circ\overline{(u,v)}$. 

Now, let $\psi$ be the image of $\overline{(u,v)}$ in $[X,V_{2n}]_{\A^1}$. Then, the images of $\psi$ and $\varphi$ in $[X,(\A^n\setminus 0)\times(\A^n\setminus 0)]_{\A^1}$ are equal to $(u\times v)\circ\Delta_X\in[X,(\A^n\setminus 0)\times(\A^n\setminus 0)]_{\A^1}$; since the map $[X,V_{2n}]_{\A^1}\rightarrow[X,(\A^n\setminus 0)\times(\A^n\setminus 0)]_{\A^1}$ is a bijection, it follows that $\psi=\varphi$ in $[X,V_{2n}]_{\A^1}$. Notice also that by Lemma \ref{modelfold}, the equality $q_n\circ\psi=\nabla\circ\overline{(u,v)}$ holds in $[X,\A^n\setminus 0]_{\A^1}$. Thus \[\Psi(u)+\Psi(v)=\nabla\circ\overline{(u,v)}=q_n\circ\psi=q_n\circ\varphi=\Psi(u+v)\] in $[X,\A^n\setminus 0]_{\A^1}$ and $\Psi$ is indeed a group homomorphism.
\end{proof}

This finishes the proof of Theorem \ref{groupiso}.

%

\section{Comparing cohomotopy theories}\label{compcohothe}

The letter $k$ once again denotes an arbitrary field. In this section, we compare the cohomotopy theory $[\text{–},Q_{2n+1}]_{\A^1}$ represented by the quadric $Q_{2n+1}$ to the cohomotopy theory $[\text{–},Q_{2n}]_{\A^1}$ represented by the quadric $Q_{2n}$ and studied in \cite{AF}. We provide morphisms of schemes $Q_{2n+1}\rightarrow Q_{2n}$ and $Q_{2n}\times\mathbb{G}_m\rightarrow Q_{2n+1}$ to allow for these comparisons.

\subsection{The morphism $\eta_n:Q_{2n+1}\rightarrow Q_{2n}$}\label{mapeta}

We fix an integer $n$ with $n\geqslant 2$. There is an isomorphism $[Q_{2n+1},Q_{2n}]_{\A^1}\simeq\Kr_{-1}^\MW(k)$ of $\Kr_0^{\MW}(k)$-modules by \cite[Corollary 6.43]{Morel}: we will construct a morphism $\eta_n:Q_{2n+1}\rightarrow Q_{2n}$ which corresponds to the generator $\eta$ of Milnor--Witt $\Kr$-theory under this isomorphism.

Let us start with a lemma about unimodular rows that will streamline a few proofs in this subsection.

\begin{lem}\label{genidealgenbyunimod}
Let $A$ be a $k$-algebra, let $(a_1,\ldots,a_{n+1})$ be a unimodular row with coefficients in $A$ and let $(b_1,\ldots,b_{n+1})$ be a splitting of $(a_1,\ldots,a_{n+1})$. Then the ideal $I=\langle a_1,\ldots,a_n\rangle$ is generated by $(a_1,\ldots,a_{n-1},a_na_{n+1},s)$ where $s=a_1b_1+\cdots+a_nb_n$; moreover denoting by $\overline{a}$ the image of $a\in I$ in $I/I^2$, the map $\omega_I:(A/I)^n\rightarrow I/I^2$ of $A/I$-modules induced by $(\overline{a_1},\ldots,\overline{a_{n-1}},\overline{a_na_{n+1}})$ is surjective. Furthermore, the tuple $(a_1,\ldots,a_{n-1},a_na_{n+1},b_1a_{n+1}b_{n+1},\ldots,b_{n-1}a_{n+1}b_{n+1},b_nb_{n+1},s)$ belongs to $Q_{2n}(A)$.
\end{lem}

\begin{proof}
Let $J=\langle a_1,\ldots,a_{n-1},a_na_{n+1},s\rangle$; then clearly, $J\subseteq I$; the reverse inclusion is true since $a_{n+1}$ is invertible mod $\langle s\rangle$ (in fact, $a_{n+1}b_{n+1}=1$ mod $\langle s\rangle$), hence mod $J$, so that the fact that $a_na_{n+1}$ belongs to $J$ implies that $a_n$ belongs to $J$: thus, $I$ is equal to $J$. The element $a_{n+1}$ is invertible mod $I$: since $(\overline{a_1},\ldots,\overline{a_n})$ obviously generates $I/I^2$ as an $A/I$-module, it is then also the case of $(\overline{a_1},\ldots,\overline{a_{n-1}},\overline{a_na_{n+1}})$; in other words, $\omega_I$ is surjective as announced. The final claim follows from the following computation: \[(1-s)s=a_{n+1}b_{n+1}\sum_{i=1}^n a_ib_i=\sum_{i=1}^{n-1}a_i(b_ia_{n+1}b_{n+1})+(a_na_{n+1})(b_nb_{n+1})\] which completes the proof.
\end{proof}

By \cite[Theorem 3.1.2]{AF} (and the paragraph preceding \cite[Lemma 2.1.1]{AF}), up to (naive) $\A^1$-homotopy, giving a map from $Q_{2n+1}$ to $Q_{2n}$ amounts to giving an ideal $I$ of $k[Q_{2n+1}]$ and a \emph{local orientation}, that is, a surjective homomorphism $(k[Q_{2n+1}]/I)^n\rightarrow I/I^2$ of $k[Q_{2n+1}]$-modules. We choose $I=\langle x_1,\ldots,x_n\rangle$. Lemma \ref{genidealgenbyunimod} ensures that we may choose $\omega_I$ to be induced by $(\overline{x_1},\ldots,\overline{x_{n-1}},\overline{x_nx_{n+1}})$. It also implies that  we have the equality $I=\langle x_1,\ldots,x_{n-1},x_nx_{n+1},z\rangle$, where the element $z=x_1y_1+\cdots+x_ny_n$ of $I$ satisfies that the tuple 
\begin{equation}\label{tuple}
(x_1,\ldots,x_{n-1},x_nx_{n+1},y_1x_{n+1}y_{n+1},\ldots,y_{n-1}x_{n+1}y_{n+1},y_ny_{n+1},z)
\end{equation}
belongs to $Q_{2n}(k[Q_{2n+1}])$. The map $\eta_n:Q_{2n+1}\rightarrow Q_{2n}$ induced by these data according to \cite[Theorem 3.1.2]{AF} is then precisely the one induced by the tuple in Equation \ref{tuple}; in fact, the image of $\eta_n$ in $\Hom_{\A^1}(Q_{2n+1},Q_{2n})$, \emph{a fortiori} in $[Q_{2n+1},Q_{2n}]_{\A^1}$, is uniquely determined by $(I,\omega_I)$.

The main result of this subsection is then the following proposition:

\begin{theo}\label{genofkmin1}
The morphism $\eta_n:Q_{2n+1}\rightarrow Q_{2n}$ generates $[Q_{2n+1},Q_{2n}]_{\A^1}$ as a $\mathrm{K}_0^\MW(k)$-module.
\end{theo}

As in Section \ref{coh}, let $Y$ be the closed subscheme of $Q_{2n}$ defined by the ideal $I_Y=\langle x_1,\ldots,x_n,z\rangle$ of $k[Q_{2n}]$, so that $k[Y]=k[Q_{2n}]/I_Y$.

\begin{lem}\label{regularembedding}
The closed immersion $\kappa:Y\hookrightarrow Q_{2n}$ is a regular embedding.
\end{lem}

\begin{proof}
Since $Q_{2n}$ and $Y$ are smooth $k$-schemes, hence regular schemes, this follows from \cite[\href{https://stacks.math.columbia.edu/tag/0E9J}{Tag 0E9J}]{stacks-project}.
\end{proof}

We spell out the proof of the two following easy results.

\begin{lem}\label{kernelofmorbetproj}
Let $R$ be a non-zero ring and let $0\rightarrow M\rightarrow P\rightarrow P'\rightarrow 0$ be an exact sequence of $R$-modules. Assume that $P$ is projective of rank $r$ and that $P'$ is projective of rank $s\leqslant r$. Then $M$ is projective of rank $r-s$.
\end{lem}

\begin{proof}
Since $P'$ is projective, the morphism $P\rightarrow P'$ splits, thus we have an isomorphism $P'\oplus M\simeq P$. In particular, $M$ is a summand of a projective module and hence is projective. Let then $\p$ be a prime ideal of $R$; $M_\p$, $P_\p$ and $P_\p'$ are then free $R_\p$-modules since $R_\p$ is local: denoting by $r(\p)$ the rank of $M_\p$ (which is well-defined since $R_\p$ is non-zero), the commutativity of $R_\p$ implies that $r(\p)+s=r$, so that $r(\p)=r-s$ as required.
\end{proof}

\begin{cor}\label{surjisoprojmod}
If $R$ is a non-zero ring and $\varphi:P\rightarrow P'$ is an epimorphism of $R$-modules with $P$ (respectively $P'$) projective of rank $r$, then $\varphi$ is an isomorphism.
\end{cor}

\begin{proof}
By the previous lemma, the kernel of $\varphi$ is projective of rank $0$ and is thus the zero module.
\end{proof}

Let $\overline{\alpha}$ denote the class mod $I_Y^2$ of $\alpha\in I_Y$.

\begin{lem}\label{orientation}
The family $(\overline{x_1},\ldots,\overline{x_n})$ of elements of $I_Y/I_Y^2$ induces an isomorphism $\omega:k[Y]^n\rightarrow I_Y/I_Y^2$ of $k[Y]$-modules.
\end{lem}

\begin{proof}
\emph{The morphism $\omega:k[Y]^n\rightarrow I_Y/I_Y^2$ is surjective.} Indeed, since $z\in I_Y$, $z=x_1y_1+\cdots+x_ny_n$ mod $I_Y^2$ so that $I_Y/I_Y^2$ is generated by $(\overline{x_1},\ldots,\overline{x_n})$ as an $k[Y]$-module.

\emph{The morphism $\omega:k[Y]^n\rightarrow I_Y/I_Y^2$ is an isomorphism.} Since $k[Y]^n$ is a projective $k[Y]$-module of rank $n$, according to Corollary \ref{surjisoprojmod}, it suffices to show that $I_Y/I_Y^2$ is a projective $k[Y]$-module of rank $n$. Since $\kappa:Y\hookrightarrow Q_{2n}$ is a regular embedding, we have an exact sequence \[0\rightarrow I_Y/I_Y^2\rightarrow k[Y]\otimes_{k[Q_{2n}]}\Omega_{k[Q_{2n}]/k}\rightarrow\Omega_{k[Y]/k}\rightarrow 0\] by \cite[\href{https://stacks.math.columbia.edu/tag/063N}{Tag 063N}]{stacks-project} (and \cite[\href{https://stacks.math.columbia.edu/tag/063K}{Tag 063K}]{stacks-project}). The $k$-scheme $Q_{2n}$ is smooth of dimension $2n$ so that $\Omega_{k[Q_{2n}]/k}$ is a projective $k[Q_{2n}]$-module of rank $2n$ as noted in the introduction. Thus $\Omega_{k[Q_{2n}]/k}\otimes_{k[Q_{2n}]} k[Y]$ is a projective $k[Y]$-module of rank $2n$ by base change. Similarly, $Y$ is a smooth $k$-scheme of dimension $n$, thus $\Omega_{k[Y]/k}$ is a projective $k[Y]$-module of rank $n$. By Lemma \ref{kernelofmorbetproj}, $I_Y/I_Y^2$ is a projective $k[Y]$-module of rank $2n-n=n$ as required.
\end{proof}

\begin{cor}
The top exterior product $L_Y=\bigwedge^n(I_Y/I_Y^2)$ is a free $k[Y]$-module of rank $1$ and $\overline{x_1}\wedge\cdots\wedge\overline{x_n}$ is a basis of this $k[Y]$-module.
\end{cor}

We will now be interested in the twisted Chow--Witt groups $\widetilde{\CH}^n(Q_{2n})$ (which is twisted by the trivial bundle) and $\widetilde{\CH}^0(Y,L_Y,n)$.

We view $y$ as a (in fact, the) point of codimension $0$ of $Y$ and we denote by $\p_y$ the maximal ideal of $\mathscr{O}_{Y,y}$---that is, $\p_y$ is the zero ideal. By definition, $D(\p_y/\p_y^2)^{-1}$ is the trivial graded line bundle $(k(y),0)$ over $k(y)$. Consider the element $\langle 1\rangle\otimes((\overline{x_1}\wedge\cdots\wedge\overline{x_n})\otimes 1)$ of $\Kr_0^\MW(k(y),L_Y(y),n)$.

\begin{lem}
The element $\langle 1\rangle\otimes((\overline{x_1}\wedge\cdots\wedge\overline{x_n})\otimes 1)$ of is a cycle in $C(Y,0,L_Y,n)^0$ hence it defines an element $\alpha_Y$ of $\widetilde{\CH}^0(Y,L_Y,n)$. The $\Kr_0^\MW(k)$-module $\widetilde{\CH}^0(Y,L_Y,n)$ is free of rank $1$ with generator $\alpha_Y$.
\end{lem}

\begin{proof}
The isomorphism $k[Y]^n\rightarrow I_Y/I_Y^n$ of Lemma \ref{orientation} induces an isomorphism of complexes from $C_\RS(Y,0)$ to $C_\RS(Y,0,L_Y,n)$ which sends $\langle 1\rangle\in\Kr_0^\MW(k(y))=\Kr_0^\MW(k(y),D(\p_y/\p_y^2)^{-1})\subseteq C_\RS(Y,0)$ to $\langle 1\rangle\otimes((\overline{x_1}\wedge\cdots\wedge\overline{x_n})\otimes 1)$. Thus it suffices to show that $\langle 1\rangle\in\Kr_0^\MW(k(y))$ is a cycle, that its class generates $\widetilde{\CH}^0(Y)$ as a $\Kr_0^\MW(k)$-module and that this module is free. 

The $k$-scheme $Y$ is flat hence the structure morphism $Y\rightarrow\Spec k$ induces a pullback $C_\RS(\Spec k,0)\rightarrow C_\RS(Y,0)$ which sends $\langle 1\rangle\in\Kr_0^\MW(k)\subseteq C_\RS(\Spec k,0)^0$ to $\langle 1\rangle\in\Kr_0^\MW(k(y))\subseteq C_\RS(Y,0)^0$ by definition (\cite[Subsection 2.4]{FaLCWG}). The Rost--Schmid complex $C_\RS(\Spec k,0)$ is of the form \[\cdots\rightarrow 0\rightarrow\Kr_0^\MW(k)\rightarrow 0\rightarrow\cdots\] where $\Kr_0^\MW(k)$ is in degree $0$. In particular, it is clear that $\langle 1\rangle$ is a cycle that generates the free $\Kr_0^\MW(k)$-module $\widetilde{\CH}^0(\Spec k)=\Hr^0(C_\RS(\Spec k,0))$---thus $\langle 1\rangle\in\Kr_0^\MW(k(y))$ is a cycle of $C_\RS(Y,0)$. On the other hand, since $Y$ is isomorphic to $\A^n$, the $\A^1$-invariance of Chow--Witt groups (\cite[Theorem 2.15]{FaLCWG}) implies that the morphism $\widetilde{\CH}^0(\Spec k)\rightarrow\widetilde{\CH}^0(Y)$ induced by the structure morphism $Y\rightarrow\Spec k$ is an isomorphism and it sends the class of $\langle 1\rangle$ in $\widetilde{\CH}^0(\Spec k)$ to the class of $\langle 1\rangle$ in $\widetilde{\CH}^0(Y)$. The previous computation of $\widetilde{\CH}^0(\Spec k)$ allows one to conclude.
\end{proof}

Finally, since $\kappa$ is a regular embedding given by $(x_1,\ldots,x_n,z)$, in particular a local complete intersection morphism, with codimension $n$ image, we have a push-forward map $\kappa_*:\widetilde{\CH}^0(Y,L_Y,n)\rightarrow\widetilde{\CH}^n(Q_{2n})$ between Chow--Witt groups (\cite[Subsection 2.2]{FaLCWG}) and, unwinding the definitions, we see that $\kappa_*(\alpha_Y)=\alpha_n$ where $\alpha_n$ is the explicit generator of Lemma \ref{lem:gen_explicite_z}. 

Now let $X$ be given by the following fibre product:
\begin{center}
\begin{tikzcd}
Q_{2n+1} \arrow[d,"\eta_n"] & X \arrow[d,"\zeta_n"] \arrow[l,"\lambda"] \\
Q_{2n}                      & Y \arrow[l,"\kappa"] 
\end{tikzcd}
\end{center}
Explicitly, $X$ is the closed subscheme of $Q_{2n+1}$ given by the ideal $I_X=\langle x_1,\ldots,x_{n-1},x_nx_{n+1},x_1y_1+\cdots+x_ny_n\rangle=I_YB$ of $Q_{2n+1}$.

\begin{lem}\label{isox}
The ideal $I_X$ of $k[Q_{2n+1}]$ is generated by $(x_1,\ldots,x_{n-1},x_n)$. Thus the $k$-algebra homomorphism $k[t,u_1,\ldots,u_n]\rightarrow k[X]$ sending $t$ to $x_{n+1}$ and $u_i$ to $y_i$ for all $i\leqslant n$ induces an isomorphism $k[t^{\pm 1},u_1,\ldots,u_n]\cong k[X]$, hence an isomorphism $X\cong\A^n\times\mathbb{G}_m$ of $k$-schemes. In particular, $X$ is integral, and smooth as a $k$-scheme.
\end{lem}

\begin{proof}
Let $z=x_1y_1+\cdots+x_ny_n\in I_X\subseteq k[Q_{2n+1}]$. Then $x_{n+1}y_{n+1}=1$ mod $\langle z\rangle$, hence $x_{n+1}y_{n+1}=1$ mod $I_X$ so that $x_{n+1}$ belongs to $k[X]^\times$. Since $x_nx_{n+1}=0$ mod $I_X$, it follows that $x_n$ belongs to $I_X$. Since $z$ belongs to $\langle x_1,\ldots,x_n\rangle$, we then immediately see that $I_X=\langle x_1,\ldots,x_n\rangle$. 

Now we see that \[k[X]\cong k[x,y]/\langle x_1y_1+\cdots+x_{n+1}y_{n+1}-1,x_1,\ldots,x_{n}\rangle\cong k[x_{n+1},y]/\langle x_{n+1}y_{n+1}-1\rangle\cong k[x_{n+1}^{\pm 1},y_1,\ldots,y_{n}]\] where we denote by $k[x,y]$ (respectively $k[x_{n+1},y]$) the polynomial ring $k[x_1,\ldots,x_{n+1},y_1,\ldots,y_{n+1}]$ (respectively $k[x_{n+1},y_1,\ldots,y_{n+1}]$). This is the required isomorphism. 
\end{proof}

The isomorphism from Lemma \ref{isox} also implies that $\dim X=n+1$. We may then compute the codimension of $X$ as an irreducible closed subset of $Q_{2n+1}$: since $Q_{2n+1}$ is an integral $k$-scheme, we see that \[\codim X=\dim Q_{2n+1}-\dim X=2n+1-(n+1)=n.\]

\begin{rema}
Modulo the isomorphism $X\cong\mathbb{G}_m\times\A^n$ from Lemma \ref{isox} and the isomorphism $\alpha:Y\xrightarrow{\cong}\A^n$ from above, the morphism $\zeta=\zeta_n:X\rightarrow Y$ from the fibre product diagram defining $X$ corresponds to the morphism $\mathbb{G}_m\times\A^n\rightarrow\A^n$ given in coordinates by $(t,(u_1,\ldots,u_n))\mapsto(u_1,\ldots,u_{n-1},t^{-1}u_n)$. Indeed, $y_i=y_ix_{n+1}y_{n+1}$ in $k[X]$ for all $i<n$. In particular, $\zeta$ is a smooth morphism.

Note however that if we set $Y'\hookrightarrow k[Q_{2n}]$ to be defined by the ideal $\langle x_1,\ldots,x_n,z-1\rangle$ and if we let $\xi:X'\rightarrow Y'$ denote the corresponding morphism from the fibre product, then the morphism $\xi^*:k[Y']=k[y_1,\ldots,y_n]\rightarrow k[X']$ induces a $k[Y']$-module structure on $k[X']$ but the kernel of $\xi^*$ contains $y_1,\ldots,y_{n-1}$ so that $k[X']$ is not $k[Y']$-torsion free if $n>1$, thus not flat as a $k[Y']$-module. Hence since flatness is preserved by base change, $\eta_n$ is not flat and in particular is not smooth if $n>1$. We do not know whether there exist smooth or even flat models of $\eta_n$ for $n>1$, that is, smooth or flat morphisms of schemes which are $\A^1$-homotopic to $\eta_n$. 

Our constructions so far make sense if $n=1$. In that case, $k[X']$ is $k[Y']$-torsion free since $k[Y']=k[y]$ and $y$ is sent to $x^{-1}y\in k[x^{\pm 1},y]=k[X']$: since $k[Y']$ is then a principal ideal domain, this implies that the morphism $\xi:X'\rightarrow Y'$ is indeed flat (and in fact smooth). We do not know if $\eta_1$ is smooth or flat.
\end{rema}

Since regular immersion are not stable under base change, the following lemma merits a proof.

\begin{lem}\label{Ximmreg}
The closed immersion $\lambda:X\hookrightarrow Q_{2n+1}$ is a regular embedding.
\end{lem}

\begin{proof}
Again, since $Q_{2n+1}$ and $X$ are smooth $k$-schemes hence regular schemes, we may use \cite[\href{https://stacks.math.columbia.edu/tag/0E9J}{Tag 0E9J}]{stacks-project}.
\end{proof}

We denote by $\psi:k[Q_{2n}]\rightarrow k[Q_{2n+1}]$ the morphism of $k$-algebras that induces $\eta_n$.

\begin{lem}\label{identiconormalx}
The natural morphism $I_Y/I_Y^2\otimes_{k[Y]}k[X]\rightarrow I_X/I_X^2$ of $k[X]$-modules which sends $\overline{x}\otimes b$ to $\overline{\psi(x)b}$ is an isomorphism.
\end{lem}

As before, $\overline{b'}$ the image of $b'\in I_X$ in $I_X/I_X^2$.

\begin{proof}
The $k[X]$-module $I_Y/I_Y^2\otimes_{k[Y]}k[X]$ is free of rank $n$ as $I_Y/I_Y^2$ is free of rank $n$, generated by $(\overline{x_1},\ldots,\overline{x_n})$, as a $k[Y]$-module, while an argument entirely similar to that of the second paragraph of the proof of Lemma \ref{orientation} shows that $I_X/I_X^2$ a projective $k[X]$-module of (constant) rank $n$. On the other hand, this morphism is surjective by definition of the ideal $I_X$. Thus we have a surjective homomorphism between projective $k[X]$-modules of rank $n$: Corollary \ref{surjisoprojmod} shows that it is an isomorphism.
\end{proof}

\begin{cor}
The $k[X]$-module $I_X/I_X^2$ is free of rank $n$ and $(\overline{x_1},\ldots,\overline{x_{n-1}},\overline{x_nx_{n+1}})$ is a basis of this module. Thus $\overline{x_1}\wedge\cdots\wedge\overline{x_{n-1}}\wedge\overline{x_nx_{n+1}}$ generates the free $k[X]$-module $\bigwedge^n(I_X/I_X^2)$.
\end{cor}

\begin{proof}
This follows from the preceding lemma, the fact that $(\overline{x_1},\ldots,\overline{x_n})$ is a basis of $I_Y/I_Y^2$ as a $k[Y]$-module and the fact that $\psi(x_i)=x_i$ for all $i<n$ and $\psi(x_n)=x_nx_{n+1}$ by definition of $\psi$.
\end{proof}

Set $L_X=\bigwedge^n(I_X/I_X^2)$. Let $x$ be the generic point of $X$. The trivialisation $L_X=k[X]\overline{x_1}\wedge\cdots\wedge\overline{x_{n-1}}\wedge\overline{x_nx_{n+1}}$ of $L_X$ induces an isomorphism $C_\RS(X,0)\rightarrow C_\RS(X,0,L_X,n)$ of Rost--Schmid complexes which sends $\langle 1\rangle\in\Kr_0^\MW(k(x))\subseteq C_\RS(X,0)^0$ to the element $\langle 1\rangle\otimes((\overline{x_1}\wedge\cdots\wedge\overline{x_{n-1}}\wedge\overline{x_nx_{n+1}})\otimes 1)\in\Kr_0^\MW(k(x),L_X(x),n)$. In particular, since $\langle 1\rangle$ is a cycle of $C_\RS(X,0)$ by definition (the residue homomorphism of \cite[Theorem 1.7]{FaLCWG} vanishes on $\langle 1\rangle$), $\langle 1\rangle\otimes((\overline{x_1}\wedge\cdots\wedge\overline{x_{n-1}}\wedge\overline{x_nx_{n+1}})\otimes 1)$ is a cycle of $C_\RS(X,0,L_X,n)$. Thus it induces an element of $\widetilde{\CH}^0(X,L_X,n)$. 

Since the closed immersion $\lambda:X\hookrightarrow Q_{2n+1}$ is a regular embedding, there is an induced homomorphism $\lambda_*:\widetilde{\CH}^0(X,L_X,n)\rightarrow\widetilde{\CH}^n(Q_{2n+1})$ and by definition, the image of the class of $\langle 1\rangle\otimes((\overline{x_1}\wedge\cdots\wedge\overline{x_{n-1}}\wedge\overline{x_nx_{n+1}})\otimes 1)$ in $\widetilde{\CH}^0(X,L_X,n)$ by $\lambda_*$ is the class in $\widetilde{\CH}^n(Q_{2n+1})$ of the cycle $\langle 1\rangle\otimes(\overline{x_1}\wedge\cdots\wedge\overline{x_{n-1}}\wedge\overline{x_nx_{n+1}})^*\in\Kr_0^\MW(k(x),D(\m_x/\m_x^2)^{-1})$ where we write $\overline{b}$ for the image of $b\in I_X$ in $\m_x/\m_x^2$. We have that $x_{n+1}\in k(x)^\times$; by definition of twisted Milnor--Witt $\Kr$-theory, we then have: \[\langle 1\rangle\otimes(\overline{x_1}\wedge\cdots\wedge\overline{x_{n-1}}\wedge\overline{x_nx_{n+1}})^*=\langle x_{n+1}\rangle\otimes(\overline{x_1}\wedge\cdots\wedge\overline{x_{n-1}}\wedge\overline{x_n})^*.\] We show that this is a cycle of the Rost--Schmid complex of $Q_{2n+1}$ which generates $\widetilde{\CH}^n(Q_{2n+1})$ as a $\Kr_0^\MW(k)$-module.

\begin{lem}
Let $Z$ be the closed subscheme of $\A^{n+1}$ defined by $Z=V(x_1,\ldots,x_n)$; denote by $z$ its generic point, $F$ its intersection with $\A^{n+1}\setminus 0$ and $\theta$ the generic point of $F$ (so that $\theta=z$). Then $\langle x_{n+1}\rangle\otimes(\overline{x_1}\wedge\cdots\wedge\overline{x_n})^*\in\Kr_0^\MW(k(\theta),\Lambda_\theta)$ is a cycle of $C_\RS(\A^{n+1}\setminus 0,n)$. Setting $\gamma=\gamma_{\A^{n+1}\setminus 0}$ as the image of in $\widetilde{\CH}^n(\A^{n+1}\setminus 0)$, there exists a morphism $\Kr_{-1}^\MW(k)\rightarrow\widetilde{\CH}^n(\A^{n+1}\setminus 0)$ of $\Kr_0^\MW(k)$-modules sending $\eta$ to $\gamma$ and this morphism is an isomorphism.
\end{lem}

\begin{proof}
First, we check that $\langle x_{n+1}\rangle\otimes(\overline{x_1}\wedge\cdots\wedge\overline{x_n})^*$ is a cycle of the Rost--Schmid complex $C_\RS(\A^{n+1}\setminus 0,n)$. The $k$-scheme $F$ is smooth of codimension $n$ in $\A^{n+1}\setminus 0$, thus as in the proof of Lemma \ref{lem:gen_explicite_z}, we only need to consider residue homomorphisms in computing the concerned differentials of the Rost--Schmid complex. Let $r$ be a codimension $1$ point of $F$; $r$ is of codimension $n+1$ in $\A^{n+1}\setminus 0$ and thus corresponds to a maximal ideal $\m$ of $k[x_1,\ldots,x_{n+1}]$ that is distinct from $\langle x_1,\ldots,x_n,x_{n+1}\rangle$ and that contains $\langle x_1,\ldots,x_n\rangle$. If $x_{n+1}$ does not lie in $\m$, then $x_{n+1}$ is invertible in $\mathscr{O}_{F,r}$. Thus by definition of the untwisted residue homomorphism in Milnor--Witt $\Kr$-theory \cite[Theorem 1.7, 2.]{FaLCWG}, $\langle x_{n+1}\rangle\in\Kr_0^\MW(k(\theta))$ has trivial residue at $r$ for any choice of valuation. Similarly to the proof of Lemma \ref{lem:gen_explicite_z}, taking twists into account, we deduce from the formula of the twisted residue homomorphism \cite[Equation (1.5)]{FaLCWG} that $d_\theta^r\langle x_{n+1}\rangle\otimes(\overline{x_1}\wedge\cdots\wedge\overline{x_n})^*=0$ in $\Kr_{-1}^\MW(k(r),\Lambda_r)$ if $x_{n+1}$ does not lie in $\m$. However, if $x_{n+1}$ were to lie in $\m$, then we would have an inclusion $\langle x_1,\ldots,x_n,x_{n+1}\rangle\subseteq\m$ and thus equality since $\langle x_1,\ldots,x_n,x_{n+1}\rangle$ and $\m$ are maximal ideals, so that $r$ would be equal to the point $0\in\A^{n+1}$ which is excluded by definition of $F$. Thus $\langle x_{n+1}\rangle\otimes(\overline{x_1}\wedge\cdots\wedge\overline{x_n})^*$ is indeed a cycle of the Rost--Schmid complex $C_\RS(\A^{n+1}\setminus 0,n)$. 

Now denote by $q$ the point $0$ of $\A^{n+1}$, so that $\m_q/\m_q^2$ is a $k(q)$-vector space of dimension $n+1$ generated by $(\overline{x_1},\ldots,\overline{x_{n+1}})$. Recall (\cite[p. 23]{JF}) that by definition, $\Hr_{\{0\}}^{n+1}(\A^{n+1},\mathbf{K}_n^\MW)$ is the $(n+1)$-th cohomology group of the complex concentrated in degree $n+1$ whose component in degree $n+1$ is $\Kr_{-1}^\MW(k(q),\Lambda_q)$, the assignment $\eta\mapsto\eta\otimes(\overline{x_1}\wedge\cdots\wedge\overline{x_{n+1}})^*$ inducing an isomorphism $\Kr_{-1}^\MW(k)\cong\Kr_{-1}^\MW(k(q),\Lambda_q)$ of $\Kr_0^\MW(k)$-modules ($q$ is a rational point hence the extension morphism $k\rightarrow k(q)$ is an isomorphism). On the other hand, there is a short exact sequence \[0\rightarrow C_\RS(\{0\},n,\omega_{\A^{n+1}}(q)^{-1})=C_{\RS,\{0\}}(\A^{n+1},2n)\rightarrow C_\RS(\A^{n+1},n)\rightarrow C_\RS(\A^{n+1}\setminus 0,n)\rightarrow 0\] of complexes explained in \cite[Subsection 2.2]{FaLCWG} ($C_{\RS,\{0\}}(\A^{n+1},2n)$ is the relative Rost--Schmid complex of the pair $(\A^{n+1},\{0\})$). It induces a so-called localisation long exact sequence in cohomology with coefficients in $\mathbf{K}_n^\MW$. This sheaf is strictly $\A^1$-invariant, thus the cohomology of $\A^{n+1}$ with coefficients in $\mathbf{K}_n^\MW$ vanishes in positive degree. Since $n\geqslant 2$, it follows that the connecting homomorphism $\partial:\Hr^n(\A^{n+1}\setminus 0,\mathbf{K}_n^\MW)\rightarrow\Hr_{\{0\}}^{n+1}(\A^{n+1},\mathbf{K}_n^\MW)$ of the localisation exact sequence is an isomorphism. By definition, $\partial\gamma$ is obtained by viewing $\langle x_{n+1}\rangle\otimes(\overline{x_1}\wedge\cdots\wedge\overline{x_n})^*$ as an element of $C_\RS(\A^{n+1},n)^{n}$ (where it is \emph{not} a cycle, since it has a non-trivial residue at $q$), computing its image by the component $d_z^q:\Kr_0^\MW(k(z),\Lambda_z)\rightarrow\Kr_{-1}^\MW(k(q),\Lambda_q)$ of the differential $C_\RS(\A^{n+1},n)^n\rightarrow C_\RS(\A^{n+1},n)^{n+1}$, viewing this image as a cycle of $C_\RS(\{0\},n,\omega_{\A^{n+1}}(q))^n$ and taking its class in $\Hr_{\{0\}}^{n+1}(\A^{n+1},\mathbf{K}_n^\MW)$. Finally, using the uniformiser $x_{n+1}\in\mathscr{O}_{Z,q}$, the formula for the untwisted residue $d:\Kr_0^\MW(k(z))\rightarrow\Kr_{-1}^\MW(k(q))$ (\cite[Theorem 1.7]{FaLCWG}) shows that \[d\langle x_{n+1}\rangle=d(\langle 1\rangle+\eta[x_{n+1}])=d\langle 1\rangle+\eta d[x_{n+1}]=\eta\] since $d$ vanishes on elements of $\mathscr{O}_{Z,z}^\times$ and satisfies $d[x_{n+1}]=\langle x_{n+1}\rangle$ by definition. One may then reason as in the proof of Lemma \ref{lem:gen_explicite_z} regarding twists to show that $\partial\gamma$ is the class of $\eta\otimes(\overline{x_1}\wedge\cdots\wedge\overline{x_{n+1}})^*$, establishing the claim.
\end{proof}

\begin{cor}\label{explicitgencwdenpo}
The element $\langle x_{n+1}\rangle\otimes(\overline{x_1}\wedge\cdots\wedge\overline{x_n})^*$ of $C_\RS(Q_{2n+1},n)^n$ described previously is a cycle. Moreover, there is a morphism $\Kr_{-1}^\MW(k)\rightarrow\widetilde{\CH}^n(Q_{2n+1})$ of $\Kr_0^\MW(k)$-modules sending $\eta$ to the class $\gamma_{n}$ of $\langle x_{n+1}\rangle\otimes(\overline{x_1}\wedge\cdots\wedge\overline{x_n})^*$ and this morphism is an isomorphism.
\end{cor}

\begin{proof}
Let $p:Q_{2n+1}\rightarrow\A^{n+1}\setminus 0$ be the projection on the first $n+1$ coordinates. Then $p$ is a smooth morphism with $p(x)=\theta$ since the ideal defining $X$ is generated by $(x_1,\ldots,x_n)$ as noted in Lemma \ref{isox}. The induced morphism $p^*:\m_\theta/\m_\theta^2\otimes_{k(\theta)}k(x)\rightarrow\m_x/\m_x^2$ sends $x_i$ to $x_i$ for all $i$ so that there is an induced morphism $p^*:C_\RS(\A^{n+1}\setminus 0,n)\rightarrow C_\RS(Q_{2n+1},n)$ between Rost--Schmid complexes which sends $\langle x_{n+1}\rangle\otimes(\overline{x_1}\wedge\cdots\wedge\overline{x_n})^*\in\Kr_0^\MW(k(\theta),\Lambda_\theta)$ to $\langle x_{n+1}\rangle\otimes(\overline{x_1}\wedge\cdots\wedge\overline{x_n})^*\in\Kr_0^\MW(k(x),\Lambda_x)$ by \cite[Example 2.11]{FaLCWG}—in particular, the latter is indeed a cycle of the Rost--Schmid complex computing $\widetilde{\CH}^n(Q_{2n+1})$. On the other hand, $p$ is a Jouanolou device, hence an $\A^1$-weak equivalence so that the morphism $p^*:\widetilde{\CH}^n(\A^{n+1}\setminus 0)\rightarrow\widetilde{\CH}^n(Q_{2n+1})$ induced at the level of Chow--Witt groups is a bijection. The corollary then follows from the previous lemma.
\end{proof}

The following lemma will allow us to use the base-change formula for Chow--Witt groups and conclude.

\begin{lem}\label{torindependent}
The fibre product square
\begin{center}
\begin{tikzcd}
X \arrow[r,"\lambda"] \arrow[d,"\zeta"] & Q_{2n+1} \arrow[d,"\eta_n"] \\
Y \arrow[r,"\kappa"]                   & Q_{2n} 
\end{tikzcd}
\end{center}
is \emph{Tor-independent}: for any $w\in Y$, any $u\in Q_{2n+1}$ with $\eta_n(u)=\kappa(w)=q$ and any $i>0$, the $\mathscr{O}_{Q_{2n},q}$-module $\Tor_i^{\mathscr{O}_{Q_{2n},q}}(\mathscr{O}_{Q_{2n+1},u},\mathscr{O}_{Y,w})$ is zero.
\end{lem}

To prove it, we rely on the following fact.

\begin{lem}\label{conditiontorind}
Let $R$ be a local ring, let $I$ be an ideal of $R$ and let $M$ be an $R$-module. Assume $I$ is generated by a regular sequence $(a_1,\ldots,a_n)$ such that $(a_1,\ldots,a_n)$ is also $M$-regular\footnote{In the usual sense of, \emph{e.g.}, \cite[p. 123]{Matsumura}.}. Then the group $\Tor_i^R(R/I,M)$ is zero for all $i>0$.
\end{lem}

\begin{proof}
Consider the Koszul complex $\Kos(a_1,\ldots,a_n)$ associated with $(a_1,\ldots,a_n)$ (\cite[p. 127]{Matsumura}). Then since $(a_1,\ldots,a_n)$ is regular, $\Kos(a_1,\ldots,a_n)$ is exact, see \cite[Theorem 16.5 (i)]{Matsumura}. It is therefore a free, in particular flat, resolution of $R/\langle a_1,\ldots,a_n\rangle=R/I$. Hence $\Tor_i^R(R/I,M)=\Hr_i(\Kos(a_1,\ldots,a_n)\otimes_R M)$ for all $i>0$. But since $(a_1,\ldots,a_n)$ is $M$-regular, $\Hr_i(\Kos(a_1,\ldots,a_n)\otimes_R M)=0$ for all $i>0$ still by \cite[Theorem 16.5 (i)]{Matsumura}, yielding the result.
\end{proof}

Let us now prove Lemma \ref{torindependent}.

\begin{proof}[Proof of Lemma \ref{torindependent}]
Let $w\in Y$ and $u\in Q_{2n+1}$ be such that $\kappa(w)=\eta_n(u)=q$. Then $\mathscr{O}_{Y,w}=\mathscr{O}_{Q_{2n},q}/(I_Y)_{q}$ where $(I_Y)_{q}$ is generated by $(x_1,\ldots,x_n)$ since $\kappa(w)=q\in D(z-1)\subseteq Q_{2n}$. The sequence $(x_1,\ldots,x_n)$ of elements of $\mathscr{O}_{Q_{2n},q}$ is regular, hence according to the previous lemma, it suffices to show that $(\psi(x_1),\ldots,\psi(x_n))$ is a regular sequence of elements of $\mathscr{O}_{Q_{2n+1},u}$—recall that $\psi:k[Q_{2n}]\rightarrow k[Q_{2n+1}]$ is the morphism of $k$-algebras which induces the scheme morphism $\eta_n$. 

By definition, $\psi(x_i)=x_i$ for $i<n$ and $\psi(x_n)=x_nx_{n+1}$. Note that since $X=Q_{2n+1}\times_{Q_{2n}}Y$, $u$ belongs to $\lambda(X)$. In particular, since $I_X=\langle x_1,\ldots,x_n\rangle$, $x_i$ belongs to $\m_u$ for all $i<n+1$—hence $x_nx_{n+1}\in\m_u$ so that $\mathscr{O}_{Q_{2n+1},u}/\langle x_1,\ldots,x_nx_{n+1}\rangle$ is non-zero. This implies also that $(x_1,\ldots,x_{n-1},x_n)$ is regular as a localisation of the regular sequence $(x_1,\ldots,x_{n-1},x_n)$ of elements of $Q_{2n+1}$. But $x_{n+1}$ is invertible in $\mathscr{O}_{Q_{2n+1},u}/\langle x_1,\ldots,x_n\rangle$: indeed, $x_{n+1}y_{n+1}=1$ in this ring since $u\in\lambda(X)$. The regularity of $(x_1,\ldots,x_{n-1},x_nx_{n+1})$ then results from that of $(x_1,\ldots,x_{n-1},x_n)$ which concludes the proof.
\end{proof}

\begin{proof}[Proof of Theorem \ref{genofkmin1}]
We denote by $x$ the generic point of $X$ and $y$ the generic point of $Y$. Put $\delta=(\overline{x_1}\wedge\cdots\wedge\overline{x_n})\otimes 1\in\zeta^*L_Y=L_Y\otimes_{k[Y]}k[X]$. Following Lemma \ref{identiconormalx}, we have an identification $(\zeta^*L_Y(x),n)\cong(L_X(x),n)$ of graded $k(x)$-vector spaces which sends $\delta\otimes 1\in\zeta^*L_Y(x)=\zeta^*L_Y\otimes_{k[X]}k(x)$ to $(\overline{x_1}\wedge\cdots\wedge\overline{x_{n-1}}\wedge\overline{x_nx_{n+1}})\otimes 1\in L_X(x)$. On the other hand, since $\zeta:X\rightarrow Y$ is a smooth morphism, it induces a morphism $C_\RS(Y,0,L_Y,n)\rightarrow C_\RS(X,0,\zeta^*L_Y,n)$ of complexes which sends $\langle 1\rangle\otimes((\overline{x_1}\wedge\cdots\wedge\overline{x_n})\otimes 1)\in\Kr_0^\MW(k(y),L_Y(y),n)$ to $\langle 1\rangle\otimes(\delta\otimes 1)\in\Kr_0^\MW(k(x),\zeta^*L_Y(x),n)$ by \cite[Example 2.11]{FaLCWG}. Hence we obtain a morphism \[\zeta^*:\widetilde{\CH}^0(Y,L_Y,n)\rightarrow\widetilde{\CH}^0(X,L_X,n)\] which sends the class $\alpha_Y$ of $\langle 1\rangle\otimes((\overline{x_1}\wedge\cdots\wedge\overline{x_n})\otimes 1)$ to the class $\alpha_X$ of $\langle 1\rangle\otimes((\overline{x_1}\wedge\cdots\wedge\overline{x_{n-1}}\wedge\overline{x_nx_{n+1}})\otimes 1)$.

Now because of Lemma \ref{torindependent}, we have a commutative diagram:
\begin{center}
\begin{tikzcd}
\widetilde{\CH}^n(Q_{2n+1}) & \widetilde{\CH}^0(X,L_X,n) \arrow[l,"\lambda_*"] \\
\widetilde{\CH}^n(Q_{2n}) \arrow[u,"\eta_n^*"] & \widetilde{\CH}^0(Y,L_Y,n) \arrow[l,"\kappa_*"] \arrow[u,"\zeta^*"] 
\end{tikzcd}
\end{center}
see \cite[Theorem 3.18]{FaLCWG}. Recall also that we introduced a generator $\gamma_n$ of $\widetilde{\CH}^n(Q_{2n+1})$ in Corollary \ref{explicitgencwdenpo}. We then have that $\gamma_{n}=\lambda_*\alpha_X$ so that with the notation $\alpha_n$ of Lemma \ref{lem:gen_explicite_z}, \[\eta_n^*\alpha_n=\eta_n^*\kappa_*\alpha_Y=\lambda_*\zeta^*\alpha_Y=\lambda_*\alpha_X=\gamma_n.\] In other words, $\eta_n^*$ sends the generator $\alpha_n$ of $\widetilde{\CH}^n(Q_{2n})$ to the generator $\gamma_n$ of $\widetilde{\CH}^n(Q_{2n+1})$.

On the other hand, since $Q_{2n+1}$ and $Q_{2n}$ are of $\A^1$-cohomological dimension $n\leqslant 2n-2$ (with $n\geqslant 2$) by Example \ref{aonecohodimquad},  the naturality statement in \cite[Theorem 1.3.4]{AF} implies that we have a commutative diagram:
\begin{center}
\begin{tikzcd}
\protect{[Q_{2n},Q_{2n}]_{\A^1}} \arrow[r,"\eta_n^*"] \arrow[d,"h"] & \protect{[Q_{2n+1},Q_{2n}]_{\A^1}} \arrow[d,"h"] \\
\widetilde{\CH}^n(Q_{2n}) \arrow[r,"\eta_n^*"]                      & \widetilde{\CH}^n(Q_{2n+1})
\end{tikzcd}
\end{center}
in the category of $\Kr_0^\MW(k)$-modules by Proposition \ref{hurhomlin} whose vertical maps are isomorphisms. It then follows from Corollary \ref{explicitgencwdenpo} that there exists a $\Kr_0^\MW(k)$-module isomorphism from $\Kr_{-1}^\MW(k)$ to $[Q_{2n+1},Q_{2n}]_{\A^1}$ that sends $\eta$ to $\eta_n$, thus concluding the proof.
\end{proof}

\subsection{The morphism $\mu_n:\mathbb{G}_m\times Q_{2n}\rightarrow Q_{2n+1}$}\label{mapmu}

Recall that $Q_{2n+1}\cong\A^{n+1}\setminus 0\simeq\mathrm{S}^n\wedge\mathbb{G}_m^{n+1}$ (\cite[§3, Example 2.20]{MoVo}) in $\mathcal{H}_\bullet(k)$, hence $Q_{2n+1}\simeq\mathbb{G}_m\wedge Q_{2n}$ in $\mathcal{H}_\bullet(k)$. We are going to construct a map $\mu_n:\mathbb{G}_m\times Q_{2n}\rightarrow Q_{2n+1}$ which, up to a slight modification, induces an $\A^1$-weak equivalence $\mathbb{G}_m\wedge Q_{2n}\rightarrow Q_{2n+1}$. It also gives a geometric interpretation for the set-theoretic map $\delta$ from Euler class groups to orbit sets of unimodular rows introduced in \cite[Subsection 2.5]{DTZ1}.

The $k$-schemes $\mathbb{G}_m\times Q_{2n}$ and $Q_{2n+1}$ are affine, hence it suffices to define $\mu_n$ at the level of $R$-points for any $k$-algebra $R$ (in a way that is functorial in $R$). Let then $R$ be a $k$-algebra, let $\lambda\in R^\times$ be a unit and let $(x,y,z)\in Q_{2n}(R)$: by definition, we then have that $\sum x_iy_i=z(1-z)$. Let $v=(x,1-z+\lambda z)\in R^{n}\times R=R^{n+1}$. We are going to show that $v$ is unimodular and specify a splitting of $v$: this will in turn define the image of $(\lambda,(x,y,z))$ by $\mu_n(R):\mathbb{G}_m(R)\times Q_{2n}(R)\rightarrow Q_{2n+1}(R)$.

Set $a(\mu,z)=(1-z)+\mu z\in R$ for all $\mu\in R^\times$. Then \[a(\lambda,z)a(\lambda^{-1},z)=(1-z+\lambda z)(1-z+\lambda^{-1}z)=(1-z)^2+(\lambda+\lambda^{-1})z(1-z)+z^2\] with \[(1-z)^2+z^2=1-2z+2z^2=1-2z(1-z).\] Thus \[a(\lambda,z)a(\lambda^{-1},z)=(\lambda+\lambda^{-1})z(1-z)+1-2z(1-z)=1+(\lambda+\lambda^{-1}-2)z(1-z)=1-b(\lambda)\sum_{i=1}^nx_iy_i\] where $b(\lambda)=2-\lambda-\lambda^{-1}$. Hence the row $w=(b(\lambda)y_1,\ldots,b(\lambda)y_n,a(\lambda^{-1},z))$ is a splitting of $v$: therefore, we obtain an element $(v,w)$ of $Q_{2n+1}(R)$. We set $\mu_n(R)(\lambda,(x,y,z))=(v,w)$.

The datum of the maps $\mu_n(R)$ is clearly functorial in $R$, hence we obtain a morphism $\mu_n:\mathbb{G}_m\times Q_{2n}\rightarrow Q_{2n+1}$ of $k$-schemes. It is induced by the $k$-algebra homomorphism $k[Q_{2n+1}]\rightarrow k[\mathbb{G}_m\times Q_{2n}]$ given by \[(v_1,\ldots,v_{n+1},w_1,\ldots,w_{n+1})\mapsto(x_1,\ldots,x_n,1-z+\lambda z,(2-\lambda-\lambda^{-1})y_1,\ldots,(2-\lambda-\lambda^{-1})y_n,1-z+\lambda^{-1}z)\] where $k[Q_{2n+1}]=k[v_1,\ldots,v_{n+1},w_1,\ldots,w_{n+1}]/\langle v_1w_1+\cdots+v_{n+1}w_{n+1}=1\rangle$.

\subsubsection*{The weak equivalence $\mathbb{G}_m\wedge Q_{2n}\xrightarrow{\sim}Q_{2n+1}$}

Recall that we point $Q_{2n}$ by $(0,\ldots,0)$ and $Q_{2n+1}$ by $(0,\ldots,0,1,0,\ldots,0,1)$ (and $\mathbb{G}_m$ by $1$ as usual). Then $\mu_n$ is a map of pointed spaces; however, $\mu_n$ \emph{does not} induce a map of spaces from $\mathbb{G}_m\wedge Q_{2n}$ to $Q_{2n+1}$. 

To accomplish this, we need to modify $\mu_n$. Let then $R$ be a $k$-algebra. Note that if $(x,y,z)\in Q_{2n}(R)$, then since $a(1,z)=1$ and $b(1)=0$, \[\mu_n(1,(x_1,\ldots,x_n,y_1,\ldots,y_n,z))=(x_1,\ldots,x_n,1,0,\ldots,0,1).\] We then consider the endomorphism $E$ of $Q_{2n+1}$ given in coordinates by \[(x_1,\ldots,x_n,z,y_1,\ldots,y_n,v)\mapsto(x_1-x_1z,\ldots,x_n-x_nz,z,y_1,\ldots,y_n,v+1-zv).\] It is well-defined since, granted that $\sum_{i\leqslant n} x_iy_i+zv=1$, \[\sum_{i=1}^nx_i(1-z)y_i+z(v+1-zv)=(1-z)(1-zv)+zv+z-z^2v=1.\] Furthermore, we see that this automorphism is given by elementary operations: hence, it is naively $\A^1$-homotopic to the identity so that composing with this automorphism does not change the $\A^1$-homotopy class of $\mu_n$. Explicitly, $E(R)\circ\mu_n(R)$ sends $(\lambda,(x_1,\ldots,x_n,y_1,\ldots,y_n,z))$ to \[(x_1z(1-\lambda),\ldots,x_nz(1-\lambda),a(\lambda,z),b(\lambda)y_1,\ldots,b(\lambda)y_n,a(\lambda^{-1},z)z(1-\lambda)+1)\] for any $k$-algebra $R$. We then observe that \[E(R)\circ\mu_n(R)(\lambda,(0,\ldots,0))=(0,\ldots,0,1,0,\ldots,0,1),\;E(R)\circ\mu_n(R)(1,(x,y,z))=(0,\ldots,0,1,0,\ldots,0,1)\] because $a(1,z)=b(1)=0$. We then deduce a scheme morphism $\mu_n'=E\circ\mu_n:\mathbb{G}_m\times Q_{2n}\rightarrow Q_{2n+1}$ which induces a map $m:\mathbb{G}_m\wedge Q_{2n}\rightarrow Q_{2n+1}$ of spaces.

\begin{theo}\label{eqfaibleqpgmqi}
The morphism $m$ of spaces is an $\A^1$-weak equivalence.
\end{theo}

First let us introduce a few notations. As in the last subsection, we set $X=V(x_1,\ldots,x_n)\subseteq Q_{2n+1}$. We consider the following fibre product: 
\begin{center}
\begin{tikzcd}
\mathbb{G}_m\times Q_{2n} \arrow[r,"\mu_n"] & Q_{2n+1} \\
Z \arrow[r,"f"] \arrow[u]                & X \arrow[u]
\end{tikzcd}
\end{center}
Since we aim to make use of the base-change formula for coefficients with values in $\mathbf{K}_{n+1}^\MW$ induced by this diagram, we have to show its $\Tor$-independence:

\begin{lem}\label{torindmu}
The fibre product diagram defining $Z$ is $\Tor$-independent.
\end{lem}

\begin{proof}
Let $w\in\mathbb{G}_m\times Q_{2n}$ be such that $\mu_n(w)$ belongs to $X$, viewed as a closed subset of $Q_{2n+1}$. We have to show that $\Tor_i^{\mathscr{O}_{Q_{2n+1},x}}(\mathscr{O}_{X,x},\mathscr{O}_{\mathbb{G}_m\times Q_{2n},w})=0$ for all $i>0$. According to Lemma \ref{conditiontorind}, it suffices to show that the sequence $(\mu_n^*(x_1),\ldots,\mu_n^*(x_n))$ of elements of $\mathscr{O}_{\mathbb{G}_m\times Q_{2n},w}$ is regular. Recall that $\mu_n^*(x_i)=x_i$ by definition. We then have that $k[\mathbb{G}_m\times Q_{2n}]/\langle x_1,\ldots,x_n\rangle=k[t^{\pm 1},y_1,\ldots,y_n,z]/\langle z(1-z)\rangle$: in particular, this ring is non-zero, hence it suffices to show that $x_i$ is a regular element of $\mathscr{O}_{\mathbb{G}_m\times Q_{2n},w}/\langle x_1,\ldots,x_{i-1}\rangle$ for all $i\leqslant n$. We see that $\mathscr{O}_{\mathbb{G}_m\times Q_{2n},w}/\langle x_1,\ldots,x_{i-1}\rangle$ is a localisation of the ring \[k[x_i,\ldots,x_n,y_i,\ldots,y_n]/\left\langle\sum_{i\leqslant j\leqslant n}x_jy_j=z(1-z)\right\rangle[t^{\pm 1},y_1,\ldots,y_{i-1}]=k[Q_{2(n-i+1)}][t^{\pm 1},y_1,\ldots,y_{i-1}]\] which is a domain so that its localisations are domains, hence it suffices to show that $x_i$ is non-zero in $\mathscr{O}_{\mathbb{G}_m\times Q_{2n},w}/\langle x_1,\ldots,x_{i-1}\rangle$ which is true.
\end{proof}

The scheme $Z$ is the coproduct of $Z_0=V(x_1,\ldots,x_n,z)$ and $Z_1=V(x_1,\ldots,x_n,1-z)$ and there are explicit $k$-scheme isomorphisms $Z_i\cong\mathbb{G}_m\times\A^n$. Note that $Z_i=\mathbb{G}_m\times Y_i$ where $Y_0=V(x_1,\ldots,x_n,z)\hookrightarrow Q_{2n}$ is the subscheme $Y$ defined above and $Y_1=V(x_1,\ldots,x_n,1-z)\hookrightarrow Q_{2n}$ is the scheme $Y'$ described above. Let $\upsilon_i\in Y_i$ denote the generic point. Observe that $(x_1,\ldots,x_n)$ generates $\m_{\upsilon_i}$.

\begin{proof}[Proof of Theorem \ref{eqfaibleqpgmqi}]
First, since $\mathbb{G}_m\wedge Q_{2n}$ and $Q_{2n+1}$ are $\A^1$-equivalent to $\A^{n+1}\setminus 0$, by \cite[Lemma 2.2]{DubFas}, it suffices to show that the map $m^*:\Hr^n(Q_{2n+1},\mathbf{K}_{n+1}^\MW)\rightarrow\Hr^n(\mathbb{G}_m\wedge Q_{2n},\mathbf{K}_{n+1}^\MW)$ induced by $m$ in cohomology is an isomorphism of $\Kr_0^\MW(k)$-modules.

Now we first elucidate the action of $\mu_n':\mathbb{G}_m\times Q_{2n}\rightarrow Q_{2n+1}$ at the level of cohomology with values in $\Kr_{n+1}^\MW$, namely the map \[\mu_n'^*:\Hr^n(Q_{2n+1},\mathbf{K}_{n+1}^\MW)\rightarrow\Hr^n(\mathbb{G}_m\times Q_{2n},\mathbf{K}_{n+1}^\MW).\] The map $\mu_n'$ has the same $\A^1$-homotopy class in $[\mathbb{G}_m\times Q_{2n},Q_{2n+1}]_{\A^1}$ as the map $\mu_n$ since the former is obtained by composing the latter with an elementary endomorphism of $Q_{2n+1}$ which is $\A^1$-homotopic to the identity: since $\mathbf{K}_{n+1}^\MW$ is strictly $\A^1$-invariant, it follows that $\mu_n$ and $\mu_n'$ induce the same map $\Hr^n(Q_{2n+1},\mathbf{K}_{n+1}^\MW)\rightarrow\Hr^n(\mathbb{G}_m\times Q_{2n},\mathbf{K}_{n+1}^\MW)$. We reason with $\mu_n$ which has a simpler expression.

The strategy is the same as in the last section: we use the base change theorem for cohomology with values in $\mathbf{K}_{n+1}^\MW$, that is, \cite[Theorem 3.18]{FaLCWG}, to simplify the computation. Recall the fibre product diagram:
\begin{center}
\begin{tikzcd}
\mathbb{G}_m\times Q_{2n} \arrow[r,"\mu_n"] & Q_{2n+1} \\
Z \arrow[r,"f"] \arrow[u]                & X \arrow[u]
\end{tikzcd}
\end{center}
The map $f$ is explicitly given by $f(t,y_1,\ldots,y_n)=(1,b(t)y_1,\ldots,b(t)y_n)$ on $Z_0$ and by $f(t,y_1,\ldots,y_n)=(t,b(t)y_1,\ldots,b(t)y_n)$ on $Z_1$. Therefore we have a commutative diagram:
\begin{center}
\begin{tikzcd}
Z \arrow[d] \arrow[r,"f"]                    & X \arrow[d] \\
\mathbb{G}_m\sqcup\mathbb{G}_m \arrow[r,"g"] & \mathbb{G}_m 
\end{tikzcd}
\end{center}
where the morphism $g:\mathbb{G}_m\sqcup\mathbb{G}_m\rightarrow\mathbb{G}_m$ is the constant morphism to $1$ on the first factor and is the identity on the second factor; the vertical maps are of the form $\A^n\times\mathbb{G}_m\rightarrow\mathbb{G}_m$ and hence are $\A^1$-weak equivalences. Hence we have a commutative diagram:
\begin{center}
\begin{tikzcd}
\Hr^0(Z,\Kbf_1^\MW) & \arrow[l,"f^*"] \Hr^0(X,\Kbf_1^\MW) \arrow[l] \\
\Hr^0(\mathbb{G}_m\sqcup\mathbb{G}_m,\Kbf_1^\MW) \arrow[u] & \Hr^0(\mathbb{G}_m,\Kbf_1^\MW) \arrow[u] \arrow[l,"g^*"]
\end{tikzcd}
\end{center}
where the vertical arrows, which are induced by pull-back from the scheme morphisms of previous diagram, are isomorphisms since $\Kbf_1^\MW$ is strictly $\A^1$-invariant. For convenience, set $M=\Kr_1^\MW(k)\oplus\Kr_0^\MW(k)[t]\cong\Hr^0(\mathbb{G}_m,\Kbf_1^\MW)$. Then $g^*$ can be seen as a map $M\rightarrow M\oplus M$ and the second statement of Lemma \ref{cohomologyprodwithgm} shows that in this way, $g^*$ takes $\alpha+\beta[t]$ to $(\alpha,\alpha+\beta[t])$. Since $t_X=x_{n+1}$, it follows that $f^*([x_{n+1}])=f^*(0+1[t_X])=(0,0+1[t_{Z_1}])=(0,[t_{Z_1}])$.

Recall that $I_X$ is the ideal defining $X$ as a closed subscheme of $Q_{2n+1}$ and that $L_X$ is the determinant of $I_X/I_X^2$. We adopt similar notations for $Z$, $Z_0$ and $Z_1$ (as closed subschemes of $\mathbb{G}_m\times Q_{2n}$); we note that $I_{Z_1}/I_{Z_1}^2$ is free of rank $n$ generated by $(\overline{x_1},\ldots,\overline{x_n})$. We then have a commutative diagram:
\begin{center}
\begin{tikzcd}
\Hr^0(Z,\Kbf_1^\MW(L_Z,n)) & \Hr^0(Z,\Kbf_1^\MW(f^*L_X,n)) \arrow[l] & \Hr^0(X,\Kbf_1^\MW(L_X,n)) \arrow[l] \\
                           & \Hr^0(Z,\Kbf_1^\MW) \arrow[u]           & \Hr^0(X,\Kbf_1^\MW) \arrow[l] \arrow[u]
\end{tikzcd}
\end{center}
In this diagram, the map $\Hr^0(Z,\Kbf_1^\MW(f^*L_X,n))\rightarrow\Hr^0(Z,\Kbf_1^\MW(L_Z,n))$ is induced by the obvious map $f^*L_X\rightarrow L_Z$ which is an isomorphism. The horizontal maps of the squares are induced by pull-back by $f$. Finally, the right vertical arrow is induced by the trivialisation $L_X\cong k[X]\overline{x_1}\wedge\cdots\wedge\overline{x_{n-1}}\wedge\overline{x_nx_{n+1}}$ and the left vertical arrow by the induced trivialisation of $f^*L_X$.

Now denoting by $x$ the generic point of $X$ as before and by $\zeta_i$ that of $Z_i$, $[x_{n+1}]\otimes((\overline{x_1}\wedge\cdots\wedge\overline{x_n})\otimes 1)\in\Kr_1^\MW(k(x),L_X(x),n)$ is sent to \[[t_{Z_1}]\otimes((\overline{x_1}\wedge\cdots\wedge\overline{x_n})\otimes 1)\in\Kr_1^\MW(k(\zeta_1),L_{Z_1}(\zeta_1),n)\hookrightarrow\Hr^0(Z,\Kbf_1^\MW(L_Z,n))\] by the map $\Hr^0(X,\mathbf{K}_1^\MW(L_X,n))\rightarrow\Hr^0(Z,\mathbf{K}_1^\MW(L_Z,n))$ since $[x_{n+1}]$ is sent to $0$ on the first factor of $\Hr^0(Z,\Kbf_1^\MW)$. Finally, thanks to Lemma \ref{torindmu}, we can apply the base change theorem for cohomology with coefficients in $\Kbf_{n+1}^\MW$ and for the fibre product
\begin{center}
\begin{tikzcd}
\mathbb{G}_m\times Q_{2n} \arrow[r,"\mu_n"] & Q_{2n+1} \\
Z \arrow[u] \arrow[r] \arrow[u]             & X \arrow[u]
\end{tikzcd}
\end{center}
According to the above discussion, \cite[Theorem 3.18]{FaLCWG} yields a commutative diagram:
\begin{center}
\begin{tikzcd}
\Hr^n(\mathbb{G}_m\times Q_{2n},\Kbf_{n+1}^\MW) & \Hr^n(Q_{2n+1},\Kbf_{n+1}^\MW) \arrow[l,"\mu_n^*"] \\
\Hr^0(Z_1,\Kbf_1^\MW(L_{Z_1},n)) \arrow[u] & \protect{\Kr_0^\MW(k)\cdot[x_{n+1}]\otimes((\overline{x_1}\wedge\cdots\wedge\overline{x_n})\otimes 1)} \arrow[l] \arrow[u]
\end{tikzcd}
\end{center}
The right-hand vertical morphism sends $[x_{n+1}]\otimes((\overline{x_1}\wedge\cdots\wedge\overline{x_n})\otimes 1)$ to the class $\beta_n=[x_{n+1}]\otimes(\overline{x_1}\wedge\cdots\wedge\overline{x_n})^*$ of \cite[Lemma 4.2.5]{ADF} which generates the free $\Kr_0^\MW(k)$-module $\Hr^n(Q_{2n+1},\Kbf_{n+1}^\MW)$ according to that lemma. The left-hand vertical morphism can be viewed as a map \[\Kr_1^\MW(k)\oplus\Kr_0^\MW(k)[t_{Z_1}]\rightarrow\Hr^n(Q_{2n},\Kbf_{n+1}^\MW)\oplus\widetilde{\CH}^n(Q_{2n})[t_{Q_{2n}}]\] and $[t_{Z_1}]$ is sent to $\alpha_n'[t_{Q_{2n}}]$. All in all, we see that $\mu_n^*\beta_n=\alpha_n'[t_{Q_{2n}}]$ where $\alpha_n'$ is the generator of the free $\Kr_0^\MW(k)$-module $\widetilde{\CH}^n(Q_{2n})$ described in Lemma \ref{lem:gen_explicite_1-z}.

Finally recall that by Lemma \ref{lem:cohomology_algebraic_suspension}, the map \[\pi^*:\Hr^n(\mathbb{G}_m\wedge Q_{2n},\mathbf{K}_{n+1}^\MW)\rightarrow\Hr^n(\mathbb{G}_m\times Q_{2n},\mathbf{K}_{n+1}^\MW)\cong\Hr^n(Q_{2n},\mathbf{K}_{n+1}^\MW)\oplus\widetilde{\CH}^n(Q_{2n})[t_{Q_{2n}}]\] induced by the quotient map $\pi:\mathbb{G}_m\times Q_{2n}\rightarrow\mathbb{G}_m\wedge Q_{2n}$ factors into an isomorphism $\Hr^n(\mathbb{G}_m\wedge Q_{2n},\mathbf{K}_{n+1}^\MW)\cong\widetilde{\CH}^n(Q_{2n})[t_{Q_{2n}}]$. Since $m\circ\pi=\mu_n'$, $\mu_n^*=\mu_n'^*=\pi^*\circ m^*$ hence the above equality $\mu_n^*\beta_n=\alpha_n'[t_{Q_{2n}}]$ shows that $m^*:\Hr^n(Q_{2n+1},\Kbf_{n+1}^\MW)\xrightarrow{\cong}\Hr^n(\mathbb{G}_m\wedge Q_{2n},\Kbf_{n+1}^\MW)$ is an isomorphism of $\Kr_0^\MW(k)$-modules as required.
\end{proof}

\subsection{Euler class groups and motivic stable cohomotopy theories}

We first recall the definition of Euler class groups and the Segre class—we follow \cite[Section 3]{AF}. If $n$ is a non-negative integer and $A$ is a ring, we denote by $\Ob_n(A)$ the set of ordered pairs $(I,\omega_I)$ where $I$ is a finitely generated ideal of $A$ and $\omega_I:(A/I)^n\rightarrow I/I^2$ is a surjective homomorphism of $A/I$-modules.

\begin{theo}[\protect{\cite[Theorem 3.1.2]{AF}}]\label{numgenidealpolring}
Let $A$ be a smooth $k$-algebra. Then given an element $(I,\omega_I)$ of $\Ob_n(A)$, there exists an element $(x,y,z)$ of $Q_{2n}(A)$ such that $I=\langle x_1,\ldots,x_n,z\rangle$; moreover, if $(x',y',z')\in Q_{2n}(A)$ possesses the same property, then the morphisms $\Spec A\rightarrow Q_{2n}$ induced by $(x',y',z')$ and $(x,y,z)$ are equal in $\Hom_{\A^1}(\Spec A,Q_{2n})$.
\end{theo}

\begin{rema}
The proof of the above theorem (relying on \cite[Lemma 2.1.1]{AF}) implies that we may additionally require of $(x,y,z)$ that the $A/I$-module homomorphism $\omega_I:(A/I)^n\rightarrow I/I^2$ is induced by $(\overline{x_1},\ldots,\overline{x_n})$ where $\overline{a}$ is the class of $a\in I$ in $I/I^2$. 

Let us also mention that the set $\Hom_{\A^1}(X,Y)$ makes sense for essentially arbitrary schemes (over the base $\Z$). The theorem above then holds for any Noetherian ring $A$ as is in fact clear from the proof (see also \cite[Theorem 2.0.9]{FaselNumGenIdPol}). Since we work in the context of $\A^1$-homotopy theory, we prefer to restrict our attention to objects of $\mathsf{Sm}_k$.
\end{rema}

Let $A$ be a smooth $k$-algebra. If $(x,y,z)$ is an element of $Q_{2n}(A)$, we denote by $[(x,y,z)]$ the image of $(x,y,z)$ by the quotient map $Q_{2n}(A)\rightarrow\Hom_{\A^1}(\Spec A,Q_{2n})$.

\begin{defi}[\protect{\cite[Definition 3.1.3]{AF}}]
Let $(I,\omega_I)$ be an element of $\Ob_n(A)$. We set $s(I,\omega_I)=[(x,y,z)]$ where $(x,y,z)$ satisfies the conclusion of the previous theorem, which shows that this is a well-defined association $s:\Ob_n(A)\rightarrow\Hom_{\A^1}(\Spec A,Q_{2n})$: we call $s$ the \emph{universal Segre class} map.
\end{defi}

\begin{exe}\label{imasegunimod}
Assume $(I,\omega_I)$ is given by \[(I,\omega_I)=(\langle a_1,\ldots,a_n\rangle,(\overline{a_1},\ldots,\overline{a_{n-1}},\overline{a_na_{n+1}}))\] where $(a_1,\ldots,a_n,a_{n+1})$ is an element of $\Um_{n+1}(A)$ with splitting $(b_1,\ldots,b_n,b_{n+1})$ as in the statement of Lemma \ref{genidealgenbyunimod}. Then setting $x_i=a_i$ for $i<n$, $x_n=a_na_{n+1}$, $y_i=b_ia_{n+1}b_{n+1}$ for $i<n$, $y_n=b_nb_{n+1}$ and $z=a_1b_1+\cdots+a_nb_n$, by the same lemma, we have the equality $[(x,y,z)]=s(I,\omega_I)$ in $\Hom_{\A^1}(\Spec A,Q_{2n})$.
\end{exe}

The next step is to somehow turn $s$ into a group homomorphism. It is not clear how to accomplish this since $\Ob_n(A)$ is not a group. As usual, the free abelian group on $\Ob_n(A)$ is not manageable and a quotient must be considered—in fact, a subquotient. Indeed, let $\Ob_n'(A)$ be the subset of $\Ob_n(A)$ whose elements are ordered pairs $(I,\omega_I)$ where $I$ is of height $n$. Let $I$ be an ideal of $A$ (respectively of height $n$); then the group $\Er_n(A/I)$ acts on the subset of $\Ob_n(A)$ (respectively of $\Ob_n'(A)$) whose elements are those ordered pairs $(J,\omega_J)$ with $J=I$: the action of $\sigma\in\Er_n(A/I)$ is given by $\sigma\cdot(I,\omega_I)=(I,\omega_I\circ\sigma^{-1})$.

\begin{defi}\protect{\cite[Definition 3.1.5]{AF}}\label{def:euler_class_groups}
Let $A$ be a smooth $k$-algebra. The \emph{Euler class group} $\Er^n(A)$ is the quotient of the free abelian group $\Z[\Ob_n'(A)]$ on $\Ob_n'(A)$ by the subgroup $G$ generated by the following elements.
\begin{enumerate}
	\item \emph{Complete intersection.} If $I=\langle a_1,\ldots,a_n\rangle$ and $\omega_I:(A/I)^n\rightarrow I/I^2$ is induced by $(\overline{a_1},\ldots,\overline{a_n})$, then $(I,\omega_I)$ belongs to the subgroup $G$.
	\item \emph{Elementary action.} If $(I,\omega_I)\in\Ob_n'(A)$ and if $\sigma\in\Er_n(A/I)$, then $(I,\omega_I)-\sigma\cdot(I,\omega_I)$ belongs to the subgroup $G$.
	\item \emph{Disconnected sum.} If $I$ is an ideal of $A$ of height $n$ such that $I=JK$ where $J$ and $K$ are ideals of height $n$ with $K+J=A$, then a surjection $\omega_I:(A/I)^n\rightarrow I/I^2$ induces surjections $\omega_J:(A/J)^n\rightarrow J/J^2$ and $\omega_K:(A/K)^n\rightarrow K/K^2$ and $(I,\omega_I)-(J,\omega_J)-(K,\omega_K)$ belongs to the subgroup $G$.
\end{enumerate}
\end{defi}

\paragraph*{The Segre class homomorphism.} \emph{From now on, we assume that $n$ satisfies $n\geqslant 2$.} Let $A$ be a smooth $k$-algebra; assume that $X=\Spec A$ has Krull dimension $d\leqslant 2n-2$. Compose the universal Segre class $s:\Ob_n'(A)\rightarrow\Hom_{\A^1}(X,Q_{2n})$ (restricted to $\Ob_n'(A)$) with the natural map $\Hom_{\A^1}(X,Q_{2n})\rightarrow[X,Q_{2n}]_{\A^1}$. The latter map is a bijection as already noted and the right-hand side of this map has a natural Borsuk-type abelian group structure by \cite[Proposition 1.2.1]{AF}, hence a map $\tilde{s}:\Z[\Ob_n'(A)]\rightarrow[X,Q_{2n}]_{\A^1}$. The answer to the obvious question is positive:

\begin{prop}[\protect{\cite[Proposition 3.1.9]{AF}}]
Assume that $A$ is a smooth $k$-algebra of Krull dimension $d\leqslant 2n-2$. Then the map $\tilde{s}$ factors through the quotient map $\Z[\Ob_n'(A)]\rightarrow\Er^n(A)$ into a group homomorphism $s:\Er^n(A)\rightarrow[X,Q_{2n}]_{\A^1}$. We call $s$ the \emph{Segre class homomorphism}.
\end{prop}

\begin{rema}
See \cite[Remark 3.1.14]{AF} regarding the assumption on the characteristic of $k$ (or lack thereof) in light of \cite{As}. We thank the referee for pointing out that we do not have to require that $k$ is infinite in this proposition.
\end{rema}

\begin{rema}
The hypotheses in the statement of \cite[Proposition 3.1.9]{AF} are not known to be sufficient: the proof requires one to assume that the Krull dimension of $A$, and not the $\A^1$-cohomological dimension of $X$, is bounded above by $2n-2$. This is because of the reliance on \cite[Theorem 2.2.10]{AF} in proving the disconnected sum relation from Definition \ref{def:euler_class_groups}: $A$ is assumed to be of Krull dimension at most $2n-2$ in the statement of \cite[Theorem 2.2.10]{AF}. Indeed, the proof of this theorem itself relies on \cite[Lemma 2.2.7]{AF} and in the proof of this lemma, the hypothesis that $A$ is of Krull dimension at most $2n-2$ allows one to conclude that given a height $n$ ideal $J$ in $A$, $A/J$ is of dimension at most $n-2$ and thus to use \cite[Lemma 2.1.5]{AF}.
\end{rema}

The fundamental theorem is the following.

\begin{theo}[\protect{\cite[Theorem 3.1.13]{AF}}]
Assume that the field $k$ is infinite. Let $d$ be an integer such that $d\leqslant 2n-2$ and let $A$ be a smooth $k$-algebra of Krull dimension $d$; set $X=\Spec A$. Then the Segre class homomorphism $s:\Er^n(X)\rightarrow[X,Q_{2n}]_{\A^1}$ is an isomorphism of abelian groups.
\end{theo}

\begin{rema}
The remark above on the lack of any assumption on the characteristic of $k$ still stands: in light of \cite{As}, $k$ may be allowed to be of characteristic $2$. However, the field $k$ is indeed required to be infinite, at least to prove that $s$ is injective, since as noted in \cite[Remark 3.1.14]{AF}, the proof of this fact relies on the homotopy invariance of Euler class groups (\cite[Theorem A.1.4]{AF}) which is only known for an infinite base field.
\end{rema}

\emph{From now on, the field $k$ is assumed to be infinite.}

\subsubsection*{The homomorphism $\Um_{d+1}(A)/\Er_{d+1}(A)\rightarrow\Er^d(A)$ and the map $\eta_d$}

In \cite[Appendix A, A.2]{AF}, the author constructs a group homomorphism $\phi$ from the orbit set $\Um_{d+1}(A)/\Er_{d+1}(A)$ to $\Er^d(A)$ after \cite[Section 3]{DAS2015185} when $A$ is a smooth $k$-algebra of Krull dimension $d\geqslant 2$. It is defined in the following way. Let $(a_1,\ldots,a_d,a_{d+1})$ be an element of $\Um_{d+1}(A)$. Up to adding multiples of $a_{d+1}$ to $a_i$ for $i\leqslant d$, which does not change the orbit under $\Er_{d+1}(A)$, we may assume that the ideal $J_0=\langle a_1,\ldots,a_d\rangle$ is of height $d$. Consider the surjection $\omega_0:A^d\rightarrow J_0$ induced by $(a_1,\ldots,a_d)$; it induces a surjection from $(A/J_0)^d$ to $J_0/J_0^2$ which we still denote by $\omega_0$. We see that $a_{d+1}$ is invertible mod $J_0$, hence there exists a diagonal matrix $g\in\GL_d(A/J_0)$ with determinant $\overline{a_{d+1}}=a_{d+1}$ mod $J_0$. The class of $(J_0,\omega_0\circ g)$ in $\Er^d(A)$ does not depend on the choice of $g$ diagonal with determinant $\overline{a_{d+1}}$ thanks to the following well-known lemma.

\begin{lem}
Let $R$ be a commutative ring, let $n\geqslant 1$ be an integer and let $D$ be a diagonal matrix with coefficients in $R$ such that $\det(D)=1$. Then $D$ is an elementary matrix.
\end{lem}

\begin{proof}
We denote by $\diag(a_1,\ldots,a_n)$ the diagonal matrix with coefficients $a_1,\ldots,a_n$. If $n=1$, then $D=1$ is elementary. If $n=2$, then $D=\diag(a,a^{-1})$ for some $a\in R$ and the result is then known as Whitehead's lemma \cite[Lemma (1.4)]{Srinivas}. It follows that $\diag(a,b)=\diag(1,ab)\diag(a^{-1},a)$ and $\diag(1,a)$ are elementarily equivalent if $a$ and $b$ lie in $R^\times$, namely $\diag(a,b)\diag(1,ab)^{-1}$ lies in $\Er_2(R)$. From there, it follows easily by induction on $n$ that $D=\diag(a_1,\ldots,a_n)$ is elementarily equivalent to $\diag(1,\ldots,1,\det(D))$ if $\det(D)=a_1\cdots a_n$ is invertible, hence to the identity matrix if $\det(D)=1$. Thus $D$ is elementary in this case which completes the proof.
\end{proof}

We denote the class of $(J_0,\omega\circ g)$ by $(J_0,\overline{a_{d+1}}\omega_0)$. We then set $\phi([a_1,\ldots,a_d,a_{d+1}])=(J_0,\overline{a_{d+1}}\omega_0)$. This formula induces a well-defined map $\phi$ from $\Um_{d+1}(A)/\Er_{d+1}(A)$ to $\Er^d(A)$ which is in fact a group homomorphism (\cite[Proposition A.2.1, Proposition A.2.2]{AF}).

Now let $X$ be the $k$-scheme $\Spec A$. Then the map $p:Q_{2d+1}\rightarrow\A^{d+1}\setminus 0$ is a pointed $\A^1$-weak equivalence, thus $p_*:[X,Q_{2d+1}]_{\A^1}\rightarrow[X,\A^{d+1}\setminus 0]_{\A^1}$ is a bijection, hence a group isomorphism by functoriality of cohomotopy theories. Therefore we have a natural sequence \[\Um_{d+1}(A)/\Er_{d+1}(A)\xrightarrow{\widetilde{\Psi}}[X,Q_{2d+1}]_{\A^1}\xrightarrow{p_*}[X,\A^{d+1}\setminus 0]_{\A^1}\] where the composite is the map $\Psi$ studied in Section \ref{proof}. Thus the map $\widetilde{\Psi}:\Um_{d+1}(A)/\Er_{d+1}(A)\rightarrow[X,Q_{2d+1}]_{\A^1}$ is a group homomorphism and in fact, since it is bijective, a group isomorphism.

Recall the $k$-scheme morphism $\eta_d:Q_{2d+1}\rightarrow Q_{2d}$ defined in Subsection \ref{mapeta}.

\begin{prop}\label{linkbetweenetamor}
Let $A$ be a smooth $k$-algebra of Krull dimension $d\geqslant 2$; denote by $X$ its spectrum. Then the diagram
\begin{center}
\begin{tikzcd}
\Um_{d+1}(A)/\Er_{d+1}(A) \arrow[d,"\widetilde{\Psi}"] \arrow[r,"\phi"] & \Er^d(A) \arrow[d,"s"]\\
\protect{[X,Q_{2d+1}]_{\A^1}} \arrow[r,"(\eta_d)_*"] & \protect{[X,Q_{2d}]_{\A^1}}
\end{tikzcd}
\end{center}
is a commutative diagram in the category of abelian groups.
\end{prop}

Recall that $s$ and $\widetilde{\Psi}$ are group isomorphisms. Hence modulo these explicit isomorphisms, $\phi$ and $(\eta_d)_*$ coincide for any smooth $k$-algebra $A$ of Krull dimension $d\geqslant 2$: in this sense, $\eta_d$ is a geometric model for the homomorphism $\phi$.

\begin{proof}[Proof of Proposition \ref{linkbetweenetamor}]
Since all the morphisms in this diagram are group homomorphisms (in the case of $(\eta_d)_*$, by functoriality of cohomotopy theories), it suffices to show that this diagram commutes in the category of sets to conclude.

Let $a=(a_1,\ldots,a_d,a_{d+1})\in\Um_{d+1}(A)$; since we are only concerned with the class of the unimodular row $(a_1,\ldots,a_d,a_{d+1})$ under elementary equivalence, we may assume that the ideal $J=\langle a_1,\ldots,a_d\rangle$ is of height $d$. We choose a splitting $b=(b_1,\ldots,b_d,b_{d+1})$ of $(a_1,\ldots,a_d,a_{d+1})$. We let $\omega$ be the surjective homomorphism $\omega:A^d\rightarrow J$ given by $(a_1,\ldots,a_d)$ and we consider the diagonal endomorphism $g=\diag(1,\ldots,1,\overline{a_{d+1}})$ of $(A/J)^d$, whose determinant is $\overline{a_{d+1}}$: by definition, $\phi([a_1,\ldots,a_d,a_{d+1}])$ is the class of $(J,\omega\circ g)$ in $\Er^d(A)$. Then $\omega\circ g:(A/J)^d\rightarrow J/J^2$ is induced by $(\overline{a_1},\ldots,\overline{a_{d-1}},\overline{a_da_{d+1}})$. Set $(x_1=a_1,\ldots,x_{d-1}=a_{d-1},x_d=a_da_{d+1},z)$ with $z=a_1b_1+\cdots+a_db_d$ and the elements $y_1,\ldots,y_d$ of $A$ given by $y_i=b_ia_{d+1}b_{d+1}$ for $i<d$ and $y_d=b_db_{d+1}$. Then by Example \ref{imasegunimod}, $s(J,\omega_0\circ g)$ is equal to the class $[(x,y,z)]\in[X,Q_{2d}]_{\A^1}$. 

On the other hand, let $f:X\rightarrow Q_{2d+1}$ by the scheme morphism given by $(a,b)\in Q_{2d+1}(A)$: by definition of $\widetilde{\Psi}$, we then have $f=\widetilde{\Psi}([a_1,\ldots,a_d,a_{d+1}])$ in $[X,Q_{2d+1}]_{\A^1}$ while by definition of the $k$-scheme morphism $\eta_d$ and of the elements $(x,y,z)$ introduced above, the $k$-scheme morphism $\eta_d\circ f:X\rightarrow Q_{2d}$ is induced by $(x,y,z)\in Q_{2d}(A)$. Hence the equality $(\eta_d)_*\circ\widetilde{\Psi}([a_1,\ldots,a_d,a_{d+1}])=[(x,y,z)]$ holds in $[X,Q_{2d}]_{\A^1}$. This equality precisely expresses the commutativity of the diagram of the proposition which completes the proof.
\end{proof}

\subsubsection*{The map $\Er^d(A)\rightarrow\Um_{d+1}(A)/\Er_{d+1}(A)$ and the morphism $\mu_d$}

We still assume $k$ to be infinite. \emph{From now on, we also require $k$ to be perfect and the characteristic of $k$ to be different from $2$.}

Recall from \cite[Subsection 2.5]{DTZ1} the map $\delta_A:\Er^d(A)\rightarrow\Um_{d+1}(A)/\Er_{d+1}(A)$ for suitable $A$. It is constructed in the following way. Let $A$ be a smooth $k$-algebra of dimension $d\geqslant 2$. Let $(I,\omega_I)$ be an element of $\Ob_d'(A)$. There exist elements $a_1,\ldots,a_d$ of $I$ such that $\omega_I$ is induced by $(\overline{a_1},\ldots,\overline{a_d})$ where $\overline{a_i}$ is the class of $a_i$ mod $I$. Then there exist $s\in I$ and $(b_1,\ldots,b_d)\in A^d$ such that $(a,b,s)\in Q_{2d}(A)$, that is, such that $s-s^2=\sum_{i=1}^d a_ib_i$ (see \cite[Lemma 2.1.1]{AF}: the proof consists in applying Nakayama's lemma to $M=I/\langle a_1,\ldots,a_d\rangle$; the lemma yields $s\in I$ such that $(1-s)M=0$ and it then suffices to write out the relation obtained by considering $\overline{s}\in M$). Now we observe that \[(1-2s)^2=1-4(s-s^2)=1-\sum_{i=1}^d(2a_i)(2b_i)\] so that $(1-2s)^2=1$ mod $\langle 2a_1,\ldots,2a_s\rangle$. Hence we obtain a row $(2a_1,\ldots,2a_d,1-2s)\in\Um_{d+1}(A)$: by definition, we then have \[\delta_ A(I,\omega_I)=[2a_1,\ldots,2a_d,1-2s].\] By \cite[Subsection 2.5]{DTZ1}, this assignment is well-defined.

The morphism $\mu_d:\mathbb{G}_m\times Q_{2d}\rightarrow Q_{2d+1}$ constructed in Subsection \ref{mapmu} induces by evaluation at $\lambda=-1\in\mathbb{G}_m(k)$ a morphism $\mu_d(-1):Q_{2d}\rightarrow Q_{2d+1}$. Recall also that the isomorphism $[Q_{2d+1},Q_{2d+1}]_{\A^1}\cong\Kr_0^\MW(k)$ given by Morel's $\A^1$-Brouwer degree (\cite[Corollary 6.43]{Morel}) induces a $\Kr_0^\MW(k)$-module structure on $[X,Q_{2d+1}]_{\A^1}$. Explicitly, if $\alpha\in k^*$, then $\langle\alpha\rangle\in\Kr_0^\MW(k)$ corresponds to the endomorphism of $Q_{2d+1}$ given in coordinates by $g_\alpha^d:(x_1,\ldots,x_{d+1},y_1,\ldots,y_{d+1})\mapsto(\alpha x_1,x_2,\ldots,x_{d+1},\alpha^{-1}y_1,y_2,\ldots,y_{d+1})$ according to Corollary \ref{cor:explicit_rank_one_representative_2n+1_quadric}. In particular, $\langle\alpha^{d}\rangle=\langle\alpha\rangle^d$ corresponds to the endomorphism of $Q_{2d+1}$ given by \[(x_1,\ldots,x_d,x_{d+1},y_1,\ldots,y_d,y_{d+1})\mapsto(\alpha x_1,\ldots,\alpha x_d,x_{d+1},\alpha^{-1}y_1,\ldots,\alpha^{-1}y_d,y_{d+1})\] (note that one may permute coordinates since the matrix $\begin{pmatrix}
0 & 1 \\
-1 & 0
\end{pmatrix}$ is elementary).

The following statement makes sense since $2\in k^\times$ as $k$ is of characteristic different from $2$.

\begin{prop}
Let $A$ be a smooth $k$-algebra of Krull dimension $d\geqslant 2$ and let $X$ be the spectrum of $A$. Then the following set-theoretic diagram:
\begin{center}
\begin{tikzcd}
\Er^d(A) \arrow[d,"s"] \arrow[r,"\delta_A"] & \Um_{d+1}(A)/\Er_{d+1}(A) \arrow[d,"\widetilde{\Psi}"] \\
\protect{[X,Q_{2d}]_{\A^1}} \arrow[r,"\langle 2^d\rangle\mu_d(-1)_*"] & \protect{[X,Q_{2d+1}]_{\A^1}}
\end{tikzcd}
\end{center}
commutes.
\end{prop}

\begin{proof}
A direct computation shows that $\mu_d(-1)$ sends $(a,b,s)\in Q_{2d}(A)$ to \[(a_1,\ldots,a_d,1-2s,4b_1,\ldots,4b_d,1-2s)\in Q_{2d+1}(A).\] It now follows from the above discussion of the action of $\Kr_0^\MW(k)$ on $[X,Q_{2d+1}]_{\A^1}$ that \[\langle 2^d\rangle\mu_d(-1)_*([(a,b,s])=[(2a_1,\ldots,2a_d,1-2s,2b_1,\ldots,2b_d,1-s)]=\widetilde{\Psi}([2a_1,\ldots,2a_d,1-2s])\] so that we indeed have $\langle 2^d\rangle\mu_d(-1)_*\circ s(I,\omega_I)=\widetilde{\Psi}\circ\delta_A(I,\omega_I)$ which completes the proof. 
\end{proof}

If $2^d$ is a square in $k$, for instance if $d$ is even or if $k$ is the field of real numbers in which $2$ is a square, then $\langle 2^d\rangle=1$ in $\Kr_0^\MW(k)$. Thus we then have $\mu_d(-1)_*=\delta_A$ modulo the isomorphisms $s$ and $\widetilde{\Psi}$: consequently, the morphism of schemes $\mu_d(-1)$ provides a geometric interpretation of $\delta_A$ in the sense already explained previously for $\phi$ and $\eta_d$ in this case. Moreover, in any case, the proposition above allows us to produce a new proof of \cite[Theorem 6.3]{das2018euler}.

\begin{cor}
Let $A$ be as above. Then the map $\delta_A$ is a group homomorphism.
\end{cor}

\begin{proof}
Indeed, $s$, $\widetilde{\Psi}$ and $\langle 2^d\rangle\mu_d(-1)_*$ are group homomorphisms and the first two are isomorphisms: the result now follows from the previous proposition.
\end{proof}

\pagestyle{plain}

\printbibliography

\end{document}